\documentclass[onefignum,onetabnum]{siamart190516}

\usepackage{lipsum}
\usepackage{amsfonts}
\usepackage{graphicx}
\usepackage{epstopdf}
\usepackage{amsmath}
\usepackage{amssymb}
\usepackage[utf8]{inputenc}
\usepackage[english]{babel}
\usepackage{graphicx}
\usepackage{tcolorbox}

\usepackage{pgfplots} 
\usepackage{tikz}

\usepackage[normalem]{ulem}

\usepackage{neuralnetwork}
\ifpdf
  \DeclareGraphicsExtensions{.eps,.pdf,.png,.jpg}
\else
  \DeclareGraphicsExtensions{.eps}
\fi

\usetikzlibrary{positioning}	
\tikzset{%
	every neuron/.style={
		circle,
		draw,
		minimum size=1cm
	},
	neuron missing/.style={
		draw=none, 
		scale=4,
		text height=0.333cm,
		execute at begin node=\color{black}$\vdots$
	},
}

\newsiamremark{remark}{Remark}
\newsiamremark{hypothesis}{Hypothesis}
\crefname{hypothesis}{Hypothesis}{Hypotheses}
\newsiamthm{claim}{Claim}
\newtheorem{assumption}[theorem]{Assumption}
\numberwithin{equation}{section}
\numberwithin{table}{section}
\newcommand{\norm}[1]{\left\| #1 \right\|}
\newcommand{\abs}[1]{ \left\vert #1 \right\vert}
\newcommand{\set}[1]{ \left\{ #1  \right\}}
\newcommand{\f}[1]{\mathbf{#1}}

\newcommand{\R}{ \mathbb{R}}

\newcommand{\N}{ \mathbb{N}}
\newcommand{\om}{\Omega}

\definecolor{myblue}{rgb}{0.00,0.44700,0.78}

\headers{Optimization with learning-informed differential equations}{G. Dong, M. Hinterm\"uller, and K. Papafitsoros}

\title{Optimization with learning-informed differential equation constraints and its applications\thanks{Submitted to the editors 25th August 2020.
\funding{This work is supported by a MATHEON Research Center project CH12 funded by the Einstein Center for Mathematics (ECMath) Berlin, and also funded by the Deutsche Forschungsgemeinschaft (DFG, German Research Foundation) under Germany's Excellence Strategy – The Berlin Mathematics Research Center MATH+ (EXC-2046/1, project ID: 390685689).}}}

\author{Guozhi Dong\footnotemark[2]\,\;\footnotemark[3]
	\and Michael Hinterm\"uller\footnotemark[3]\,\;\thanks{Institute for Mathematics, Humboldt University of Berlin,
		Unter den Linden 6, 10099 Berlin, Germany (\email{guozhi.dong@hu-berlin.de}, \email{hint@hu-berlin.de}).}
	\and Kostas Papafitsoros\thanks{Weierstrass Institute for Applied Analysis and Stochastics (WIAS), Mohrenstrasse 39, 10117 Berlin, Germany (\email{guozhi.dong@wias-berlin.de},
		\email{michael.hintermueller@wias-berlin.de}, \email{kostas.papafitsoros@wias-berlin.de}).}}

\usepackage{amsopn}

\usepackage{booktabs}

\usepackage[top=2cm,bottom=2cm,left=3.5cm,right=3.5cm]{geometry}

\ifpdf
\hypersetup{
  pdftitle={Optimization with learning-informed differential equation constraints and its applications},
  pdfauthor={G. Dong, M. Hinterm\"uller, and K. Papafitsoros}
}
\fi

\begin{document}

\maketitle

\begin{abstract}
Inspired by applications in optimal control of semilinear elliptic partial differential equations and physics-integrated imaging,  differential equation constrained optimization problems with constituents that are only accessible through data-driven techniques are studied. A particular focus is on the analysis and on numerical methods for problems with machine-learned components. For a rather general context, an error analysis is provided, and particular properties resulting from artificial neural network based approximations are addressed. Moreover, for each of the two inspiring applications analytical details are presented and numerical results are provided.
\end{abstract}

\begin{keywords}
Artificial neural network; optimal control; semilinear PDEs;  integrated physics-based imaging; learning-informed model; quantitative MRI; semi-smooth Newton; SQP algorithm
\end{keywords}

\begin{AMS}
 49M15, 65J15, 65J20, 65K10, 90C30, 35J61, 68T07
\end{AMS}

\section{Introduction}

Consider the optimization problem  
\begin{equation}
	\label{eq:optimal_control}
	\begin{aligned}
		&\text{minimize}\;   J(y,u):=\frac{1}{2}\norm{Ay-g}^2_H + \frac{\alpha}{2} \norm{u}_U^2 ,\quad\text{over }(y,u)\in Y\times U,\\
		&\text{subject to (s.t.)}\; e(y,u)=0,  \text{ and } u \in \mathcal{C}_{ad},
	\end{aligned}
\end{equation}
where $y\in Y$,  $u \in U$  are the state and control variables, respectively, with $Y$ a suitable Banach space and $U$ a Hilbert space. Moreover, $g \in H$ denotes given data with $H$ the pertinent Hilbert space, $\alpha>0$ is the control cost, and $A:Y \to H$ is a bounded linear (observation) operator, i.e., $A\in\mathcal{L}(Y,H)$. While in \eqref{eq:optimal_control} feasible controls $u$ are confined to a nonempty, closed, and convex set $\mathcal{C}_{ad}$, the relationship between admissible controls and states is through the equality constraint associated with a possibly nonlinear operator $e: Y\times U \to Z$, with $Z$ a Banach space. Often, $e(y,u)=0$ is given by (a system of) ordinary or partial differential equations (ODEs or PDEs) describing, e.g., underlying physics. For the ease of discussion we assume that, for given $u\in U$, there is a unique $y\in Y$ such that $e(y,u)=0$. This allows us to write
\begin{equation*}
	y=\Pi(u),
\end{equation*}
where $\Pi$ denotes the (implicitly defined) control-to-state map with $e(\Pi(u),u)=0$.  
Given $\Pi$, a popular approach in the study of \eqref{eq:optimal_control} is based on the reduced problem
\begin{equation}
	\label{eq:optimal_control_reduced}
	\begin{aligned}
		& \text{minimize}\;   \mathcal{J}(u):=\frac{1}{2}\norm{Q(u)-g}^2_H + \frac{\alpha}{2} \norm{u}_U^2 ,\quad\text{over }u\in U,\\
		&\text{s.t. }\; u \in \mathcal{C}_{ad},
	\end{aligned}
\end{equation}
where $Q:=A \Pi(\cdot): U\to H$. Note that $\mathcal{J}(u)=J(\Pi(u),u)$.

In general, \eqref{eq:optimal_control} or its reduced form \eqref{eq:optimal_control_reduced} 
represent a class of optimal control problems, for which a plethora of studies exist in the literature; see, e.g., \cite{Tro10} for an introduction and \cite{MR1669395, MR0271512, MR2183776} as well as the references therein for more details. In contrast, in many applications one is confronted with control problems where $e$ or, alternatively, $\Pi$ are only partly known along with measurement data which can be exploited to obtain (approximations) of missing information. Such minimization tasks have barely been treated in the literature and motive the present work.
In order to inspire such a setting, we briefly highlight here two classes of applications which will be further studied from Section \ref{sec:appl_1} onwards.

Our first motivating example is related to the fact that many phenomena in engineering, physics or life sciences, for instance, can be modeled by elliptic partial differential equations of the form
\begin{equation}\label{mh.semilinear}\left.
	\begin{aligned}
		&Ly + f(x,y)=Ru  & \quad \text{ in } \; \Omega, \\
		&b(x)\partial_\nu y + d(x)y=0  & \quad \text{ on } \;  \partial  \Omega .
	\end{aligned}\right\}
\end{equation}
Here $L$ denotes a second-order linear elliptic partial differential operator with measurable, bounded and symmetric coefficients, $f(x,y)$ is a nonlinearity, and $R$ models the impact of the control action $u$. Moreover, $b$ and $d$ are given coefficient functions. The
set $\Omega\subset\R^d$ represents the underlying domain with boundary $\partial\Omega$, and $\partial_\nu$ denotes the derivative along the outward (unit) normal $\nu$ to $\Omega$. Often the precise form of $f$ is unknown, but rather only accessible through a data set $D:=\{(y_i,u_i): e(y_i,u_i)\approx 0, i=1,\ldots,{n_{D}}\}$, ${n_{D}}\in\mathbb{N}$, i.e., given pre-specified control actions, one collects associated state responses (through measurements or computations). Utilizing data-driven approximation techniques such as artificial neural networks (ANNs), one may then get access to a data-driven model of $f$ which can be used even outside the range of the data set $D$ to yield a valid model of the underlying real-world process. In such a setting, associated optimal control problems depend on approximations $\mathcal{N}$ of $f$, and theoretical investigations as well as numerical solutions of the control problem need to take the construction of $\mathcal{N}$ into account.  

The second example comes from quantitative magnetic resonance imaging - qMRI. In this context, one integrates a mathematical model of the acquisition physics (the Bloch equations \cite{DonHinPap19}) into the associated image reconstruction task in order to relate qualitative information (such as the net magnetization $y=\rho m$) with objective, tissue dependent quantitative information (such as $T_1$ and $T_2$, the longitudinal and the transverse relaxation times, respectively, or the proton spin density $\rho$). This model is then used to obtain quantitative reconstructions from subsampled measurement data $g$ in k-space by a variational approach. The provision of such quantitative reconstructions is highly important, e.g., for subsequent automated image classification procedures to identify tissue anomalies. Moreover, in \cite{DonHinPap19} it is demonstrated that such an {\it integrated physics-based} approach is superior to the state-of-the-art technique of magnetic resonance fingerprinting (MRF) \cite{Ma_etal13} and its improved variants  \cite{DavPuyVanWia14, MazWeiTalEld18}.
Specifically in MRI, acquisition data are obtained at different pre-specified times (read-out times) $t_{1}, \ldots, t_{L}$, during which the magnetization of the matter is excited through the control of a time dependent external magnetic field $B$. Given $u=(T_1,T_2,\rho)$, the magnetization time vector at $t_{1}, \ldots t_{L}$ is then given by 
$y=\Pi(u)$,
where $\Pi$ denotes the solution map associated with a discrete version of the Bloch equations. Crucial to this approach is the fact that, at least for specific variations of the external magnetic field $B$, explicit formulas for the solution map of the Bloch equations are available. For instance, in \cite{DavPuyVanWia14} and \cite{DonHinPap19} Inversion Recovery balanced Steady-State Free Precession (IR-bSSFP) \cite{Sche99} is used which involves certain flip angle sequence patterns that characterize the external magnetic field $B$. These flip angle patterns  allow for a simple approximation of the solutions of the Bloch equations at the read-out times through a recurrence formula. 
However, in general, it is quite typical that for more complicated external magnetic fields one does not have at hand explicit representations for the Bloch solution map. More generally, for most nonlinear differential equations (including those relevant in image reconstruction tasks) explicit solution maps might be too complicated to  obtain.  However,  one may employ numerical methods to approximate their solutions $(y_i)_{i=1}^{{n_{D}}}$ given a specific (coarse) selection of parameters $(u_i)_{i=1}^{{n_{D}}}$ within a certain range. This generates a data set $D$ which is then employed in a learning procedure to generate an ANN based approximation $\Pi_{\mathcal{N}}$ of $\Pi$. This gives rise to $Q_{\mathcal{N}}:=A\Pi_{\mathcal{N}}$ in \eqref{eq:optimal_control_reduced} and requires an associated analytical as well as numerical treatment of the (reduced) minimization problem.

In general, learning-informed models are getting nowadays increasingly more popular in different scientific fields. Some works focus on the design of ANNs, e.g., by  constructing novel network architectures \cite{BakGupNaiRas17}, or on developing fast and reliable algorithms in order to train ANNs more efficiently \cite{BotCurNoc18}.  
More relevant for our present work,  ANNs have been applied to the simulation of differential dynamical systems \cite{QinWuXiu18} and high dimensional partial differential equations \cite{HanJenE18,SirSpi18}, as well as to the coefficient estimation in nonlinear partial differential equations  \cite{LonLuMaDon18}, also in connection with optimal control  \cite{E17,HabRut18} and inverse problems \cite{ArrMaaOekSch19}.
Note, however, that in our approach neural networks do not aim to approximate the solution of \eqref{eq:optimal_control}, but rather they are part of the physical process encoded in $\Pi$. We emphasize that this is a different strategy to some of the recent works \cite{AdlOek17,Bal_etal18} in the literature that focus on learning the entire model or reconstruction process. More precisely, in the present work we suggest to use an operator $\Pi_{\mathcal{N}}$ that is induced by trained neural networks modelling the equality constraint (with, e.g., $f$ replaced by an ANN-based model $\mathcal{N}$ in our example \eqref{mh.semilinear}) or its (implicitly defined) solution map $\Pi$.
In such a setting, existence, convergence, stability and error bounds of the corresponding approximations need to be analyzed.  Particularly, we are interested in the error propagation from the neural network based approximation to the solution of the optimal control problem.
Moreover, in the case of partial differential equations, when replacing $f$ by $\mathcal{N}$, the regularity of solutions has to be checked carefully before approaching the optimal control problem.
Further, from a numerical viewpoint, in order to use derivative-based numerical methods, it is important for these approximating solution maps to have certain smoothness.
This aspect is typically tied to the regularity of the activation functions employed in ANN approximations.

The remaining part of the paper is organized as follows:
Section \ref{sec:analysis} provides a general error analysis for solutions of the proposed learning-informed framework.
Some basic definitions and approximation properties of artificial neural networks are recalled in Section \ref{sec:ANN}, and
Section \ref{sec:appl_1} presents a concrete case study on optimal control of semilinear elliptic equations with general nonlinearities, including both error analysis and numerical results.
Section \ref{sec:appl_2} contains another case study on quantitative magnetic resonance imaging, again including computational results.

\section{Mathematical analysis of the general framework problem}
\label{sec:analysis}
We start our analysis by studying \eqref{eq:optimal_control_reduced}
or its variant where $Q$, the original physics-based operator, is replaced by a (data-driven) approximation. 
Existence of a solution to \eqref{eq:optimal_control_reduced} follows from standard arguments which are provided here for the sake of completeness.
\begin{proposition}\label{pro:existence_wsc}
	Suppose that  $Q$ is weakly-weakly sequentially closed, i.e., if $u_{n}\overset{U}{\rightharpoonup} u$ and $Q(u_{n})\overset{H}{\rightharpoonup}\bar{g}$, then $\bar{g}=Q(u)$.
	Then \eqref{eq:optimal_control_reduced}  admits a solution $\bar u\in U$.	
	In the special case where $\mathcal{C}_{ad}$ is a bounded set of a subspace $\hat{U}$ which is compactly embedded into $U$, it suffices that  $Q$ is strongly-weakly sequentially closed to guarantee existence of a solution to  \eqref{eq:optimal_control_reduced}. 
\end{proposition}
\begin{proof}
	Suppose that  $Q$ is weakly-weakly sequentially closed and let $(u_{n})_{n\in\mathbb{N}}\subset \mathcal{C}_{ad}$ be an infimizing sequence for \eqref{eq:optimal_control_reduced}. Since $\alpha>0$,  $(u_n)_{n\in\mathbb{N}}$ is bounded in $U$,  and thus  we can extract an (unrelabelled) weakly convergent subsequence,  i.e., $u_{n}\overset{U}{\rightharpoonup} \bar u$ for some $\bar u\in U$. Since $\mathcal{C}_{ad}$ is strongly closed and convex, it is weakly closed and therefore $\bar u\in \mathcal{C}_{ad}$. Moreover, since the sequence $(Q(u_{n}))_{n\in\mathbb{N}}$ is also bounded in $Y$, passing to a subsequence if necessary, we get that there exists a $\bar{g}\in H$ such that $Q(u_{n})\overset{H}{\rightharpoonup} \bar{g}$. Due to the weak sequential closedness  we have $\bar{g}=Q(\bar u)$. Finally, from the weak lower semicontinuity of  $\|\cdot\|_{H}$ and $\|\cdot\|_{U}$ we have $\mathcal{J}(\bar u) \le  \liminf_{n\to\infty} \mathcal{J}(u_{n})= \inf_{u\in \mathcal{C}_{ad}}  \mathcal{J}(u)$
	and hence $\bar u$ is a solution of \eqref{eq:optimal_control_reduced}. 
	
	For the special case let $(u_{n})_{n\in\mathbb{N}}$ again be an infimizing sequence for \eqref{eq:optimal_control_reduced}. Due to the compact embedding, we have that $(u_{n})_{n\in\mathbb{N}}$ has an (unrelabelled) subsequence such that $u_{n}\to \bar u$ strongly in $U$ as $n\to \infty$. Then the proof follows the same steps as above.
\end{proof}

\begin{remark}\label{rem:continuity}
	We  note here that in many examples in optimal control of (semilinear) PDEs, the control-to-state map  actually maps $U$ to a solution space $Y$ which is of higher regularity than $H$ and even compactly embeds into it; e.g., $Y:= H^{1}(\om)\hookrightarrow L^{2}(\om)=:H$. Provided that the control-to-state map is bounded, in that case weak convergence in $U$ results, up to subsequences, in strong convergence in $H$ with the latter used to show closedness of the control-to-state operator.
\end{remark}

Assuming that $Q$ is Fr\'echet differentiable with derivative $Q'(\cdot)\in\mathcal{L}(U,H)$, the first-order optimality condition of \eqref{eq:optimal_control_reduced} is
\begin{equation}\label{eq:first_optimality}
	\langle \mathcal{J}^\prime (\bar u),u-\bar u\rangle_{U^\ast,U}\geq 0  \quad \text{ for all } \; u\in \mathcal{C}_{ad}, 
\end{equation}
where $\mathcal{J}'(\bar u)\in\mathcal{L}(U,\mathbb{R})=:U^\ast$ is the Fr\'echet derivative of $\mathcal{J}$ at $\bar u$, and $\langle\cdot,\cdot\rangle_{U^\ast,U}$ denotes the duality pairing between $U$ and its dual $U^\ast$.
Utilizing the structure of $\mathcal{J}$ we get
\begin{align*}
	&\big\langle(Q^\prime(\bar u))^\ast \iota_H^{-1}(Q(\bar u)-g)  + \alpha \iota_U^{-1}\bar u, u-\bar u \big\rangle_{U^{\ast}, U} \ge 0\quad \text{for all }\; u\in \mathcal{C}_{ad},
\end{align*}
or alternatively
\[ \bar u = \mathcal{P}_{\mathcal{C}_{ad}}\left (-\frac{\iota_U(Q^\prime (\bar u))^\ast \iota_H^{-1}(Q(\bar u)-g) }{\alpha }\right),\]
where $\mathcal{P}_{\mathcal{C}_{ad}}$ is the projection in $U$ onto $\mathcal{C}_{ad}$, and $\iota_H:H^*\to H$ as well as $\iota_U:U^*\to U$ are Riesz isomorphisms, respectively. For ease of notation, however, we will leave off the Riesz maps in what follows whenever there is no confusion.

We now  proceed to the error analysis of \eqref{eq:optimal_control_reduced}, where we
assume  that $(Q_{n})_{n\in\mathbb{N}}$ is a family of operators  approximating $Q$,  and clarify the convergence of the associated minimizers $u_n\in\mathcal{C}_{ad}$.
\begin{theorem}\label{thm:convergence}
	Let $Q$ and $Q_{n}$, $n\in\mathbb{N}$, be weakly sequentially closed operators with 
	\begin{equation} \label{eq:operator_err}
		\|Q(u)-Q_{n}(u)\|_{H}\le \epsilon_{n}, \quad 
		\text{ for all } \;  u\in \mathcal{C}_{ad},
	\end{equation}
	and  {$\epsilon_{n}\downarrow 0$. Furthermore let  $(u_{n})_{n\in\N}$  be a sequence of minimizers of \eqref{eq:optimal_control_reduced}} with $Q$ replaced by $Q_n$ for all $n\in\mathbb{N}$.
	Then, we have the strong convergences
	\begin{equation}\label{eq:strong_convergence}
		u_{n} \to \bar{u}   \; \text{ in }\; U,\quad\text{ and }\quad   Q_n(u_{n}) \to  Q(\bar{u}) \; \text{ in } \;H, \quad \text{ as } \; n \to \infty,
	\end{equation}
	where $\bar{u}$ is a minimizer of \eqref{eq:optimal_control_reduced}. 
\end{theorem}	
\begin{proof}
	As $(u_{n})_{n\in\N}$ is a sequence of minimizers, we have for $C:=\max_{n} \epsilon_{n}<\infty$ and every $u\in \mathcal{C}_{ad}$:
	\begin{equation*}
		\begin{aligned}
			\frac{1}{2} \norm{Q_{n}(u_{n}) -g }_{H}^2+\frac{\alpha}{2}\norm{ u_{n}}_{U}^2
			\leq     \norm{Q(u) -g}_{H}^2+ C^2+\frac{\alpha}{2}\norm{ u}_{U}^2 .
		\end{aligned}
	\end{equation*}	
	Note also that $\|Q(u_{n})\|_{H}\le \|Q_{n}(u_{n})\|_{H}+ \epsilon_{n}$.
	Hence $(u_{n})_{n\in\mathbb{N}}$, $(Q(u_{n}))_{n\in\mathbb{N}}$ and $(Q_{n}(u_{n}) )_{n\in\mathbb{N}}$  are bounded sequences and therefore there exist (unrelabelled) subsequences and $\bar u\in U$ such that
	$u_{n}  \stackrel{U}{\rightharpoonup} \bar{u}$ with $\bar{u}\in\mathcal{C}_{ad}$ by weak closedness, $Q(u_{n})\stackrel{H}{\rightharpoonup} Q(\bar{u})$, and ${Q_{n}(u_{n}) \stackrel{H}{\rightharpoonup} Q(\bar{u})}$, 
	where we have also used that $Q$ is weakly sequentially closed for the second limit.
	For the third limit, note that for an arbitrary $\tilde{g}\in H$, by using \eqref{eq:operator_err}, we get
	\begin{align*}
		\abs{( Q_{n}(u_{n}) -Q(\bar{u}),\tilde{g})_H}
		&\le \abs{( Q_{n}(u_{n}) -Q(u_{n}),\tilde{g})_H}+ \abs{( Q(u_{n}) -Q(\bar{u}),\tilde{g})_H}\\
		&\le \epsilon_{n} \|g\|_{H} + \abs{( Q(u_{n}) -Q(\bar{u}),\tilde{g})_H} \to 0,
	\end{align*}
	where $(\cdot,\cdot)_H$ denotes the inner product in $H$.

	Using the lower semicontinuity of the norms, we have for every $u\in \mathcal{C}_{ad}$ that
	\begin{align*}
		\frac{1}{2} \norm{Q(\bar{u}) -g}_{H}^{2} &+ \frac{\alpha}{2}  \norm{\bar{u}}_{U}^2
		\le  \liminf_{n}  \frac{1}{2}\norm{Q_{n}(u_{n})-g}_{H}^{2} + \frac{\alpha}{2}  \norm{u_n}_{U}^{2}\\
		&\le \lim_{n} \frac{1}{2}\norm{Q_{n}(u)-g}_{H}^{2} + \frac{\alpha}{2}  \norm{u}_{U}^{2}
		= \frac{1}{2} \norm{Q(u)-g}_{H}^{2} +  \frac{\alpha}{2} \norm{u}_{U}^{2}.
	\end{align*}
	Thus, we conclude that $\bar{u}$ is a minimizer of \eqref{eq:optimal_control_reduced}.
	We still need to show that $u_{n} \to \bar{u}$ strongly in $U$.
	Suppose there exists a $\mu>0$ such that $\mu=\limsup_{n}\norm{u_n}_{U}> \norm{\bar{u}}_{U}$.
	Let  $(u_{n_{k}})_{k\in\mathbb{N}}$ be a subsequence with $\norm{u_{n_{k}}}_{U} \to \mu$ as $k\to \infty$.
	Then we have
	\begin{equation}\label{eq:upper_limit}
		\begin{aligned}
			\limsup_k \;\frac{1}{2}\norm{Q_{n_{k}}(u_{n_{k}}) -g}_{H}^2
			&=\limsup_k \left(\frac{1}{2}\norm{Q_{n_{k}}(u_{n_{k}}) -g}_{H}^2+\frac{\alpha}{2} (\norm{u_{n_{k}}}_{U}^2 - \mu^2)\right)\\
			& \leq \lim_k \frac{1}{2}\norm{Q_{n_{k}}(\bar{u}) -g}_{H}^2 +\frac{\alpha}{2}( \norm{\bar{u}}_{U}^2-\mu^2)\\
			&=  \frac{1}{2}\norm{Q(\bar{u}) -g}_{H}^2 +\frac{\alpha}{2}( \norm{\bar{u}}_{U}^2-\mu^2)
			<\frac{1}{2}\norm{Q(\bar{u}) -g}_{H}^2 .
		\end{aligned}
	\end{equation}
	This contradicts the lower semicontinuity of the norm and $Q_{n}(u_{n})  \rightharpoonup Q(\bar{u}) $.
	 Thus, $\|u_n\|_{U} \to  \|\bar{u}\|_{U}$ as $n\to\infty$.
	Together with the weak convergence $u_{n}  \rightharpoonup \bar{u}$ we get $u_{n} \to \bar{u}$ strongly in $U$ and further
	\[\limsup_n \norm{ Q_{n}(u_{n}) -  g }_{H}\leq \norm{Q(\bar{u})-g}_H \leq \liminf_n \norm{ Q_{n}(u_{n}) -  g }_{H} .\]
	Hence, $\lim_n \norm{Q_{n}(u_{n})}_H=\norm{Q(\bar{u})}_H $, which 
	implies the second limit in \eqref{eq:strong_convergence}.
\end{proof}

For a quantitative convergence result, we invoke the following assumptions which are motivated by the analysis of nonlinear inverse problems  \cite{Han10,LuFle12}.
\begin{assumption}\label{assum:operator_derivative}
	Assume that $Q$ is Fr\'echet differentiable and that there exists $L_0>0$ such that
	\begin{equation}\label{eq:der_bounded}
		\norm{ Q^\prime(u)}_{\mathcal{L}(U,H)}\leq L_0 \quad  \text{ for  all  }\; u\in \mathcal{C}_{ad}.
	\end{equation}
	Assume further that the Fr\'echet derivative is locally Lipschitz with modulus $L_1>0$, i.e.,
	\begin{equation}\label{eq:second_Lip}
		\norm{Q^\prime (u_a)-  Q^\prime (u_b)}_{\mathcal{L}(U,H)} \leq L_1\norm{u_a-u_b}_U, \quad  \text{ for all }\; u_a, u_b \in \mathcal{C}_{ad}.
	\end{equation}
	Moreover, let the Fr\'echet derivatives of $Q$ and $Q_{n}$ satisfy the following error bounds 
	\begin{equation} \label{eq:operator_deriv_err}
		\norm{Q^\prime(u)-Q^\prime_{n}(u)}_{\mathcal{L}(U,H)}\leq \eta_n, \; 
		\text{ for all } \;  u\in \mathcal{C}_{ad},
	\end{equation}
	where $\eta_n\in (0,1)$ for all $n\in\mathbb{N}$ and $\eta_n \downarrow 0$.
	Finally, let the two constants $L_0$ and $L_1$ satisfy
	\begin{equation}\label{eq:Lip_const_condition2}
		L_0(L_0+1) +L_1 \norm{Q(\bar{u})-g}_H<\alpha,
	\end{equation}
	with $\bar{u}$ being the minimizer of \eqref{eq:optimal_control_reduced}.
\end{assumption}
The condition in \eqref{eq:der_bounded} indicates that
\begin{equation}\label{eq:Q_Lip}
	\norm{Q(u_a)-  Q(u_b)}_{H} \leq L_0\norm{u_a-u_b}_U, \quad  \text{ for all }\; u_a,u_b \in \mathcal{C}_{ad}.
\end{equation}	

\begin{theorem}\label{thm:error_bound2}
	Let the assumptions of Theorem \ref{thm:convergence} as well as  Assumption \ref{assum:operator_derivative} hold. 
	Then, we have 
	\begin{equation}\label{eq:error_bound2}
		\norm{u_{n}- \bar{u}}_U\leq \frac{1 }{ \alpha -L_0(L_0+\eta_n) -L_1\norm{Q(\bar{u})-g}_H}\left(  L_0 \epsilon_n + \epsilon_{n}\eta_n+ \norm{Q(\bar{u})-g}_H \eta_n  \right) .
	\end{equation}
\end{theorem}
\begin{proof}
	First-order optimality  yields
	\begin{equation}
		\bar{u} =\mathcal{P}_{\mathcal{C}_{ad}}\left(- (Q^\prime(\bar{u}))^\ast w \right) \quad  \text{ and } \quad
		u_{n} =\mathcal{P}_{\mathcal{C}_{ad}}\left( - (Q_{n}^\prime(u_{n}))^\ast w_{n} \right),
	\end{equation}
	where $w=\frac{Q(\bar{u})-g}{\alpha}$ and $w_{n} =\frac{ Q_{n}(u_n)-g}{\alpha} $.
	The inequalities in \eqref{eq:der_bounded}, \eqref{eq:second_Lip},  \eqref{eq:operator_deriv_err}, and  \eqref{eq:Q_Lip}  and the fact that $\norm{Q'(u)}_{\mathcal{L}(U,H)}=\norm{(Q'(u))^{\ast}}_{\mathcal{L}(H^*,U^*)}$ imply
	\begin{equation*}
		\begin{aligned}
			\norm{ u_{n} - \bar{u}}_{U}
			\leq &\norm{ (Q_{n}^\prime(u_{n}))^\ast w_{n}- (Q^\prime(\bar{u}))^\ast w  }_{U^*} \\
			\leq &
			\norm{ (Q_{n}^\prime(u_{n}))^\ast \left(w_{n}-   w \right) }_{U^*}
			+\norm{ \left((Q_{n}^\prime(u_{n}))^\ast - (Q^\prime(\bar{u}))^\ast \right)   w  }_{U^*}\\
			\leq & (L_0+\eta_n) \norm{ w_{n}-   w }_{H}
			+  \norm{w}_{H} \eta_n
			+ L_1 \norm{ w}_{H} \norm{u_{n}-  \bar{u}}_{U} \\
			\leq & \frac{L_0+\eta_n }{\alpha}\norm{Q(\bar{u})-  Q_{n}(u_{n})}_{H}
			+\norm{w}_{H}\eta_n+ L_1 \norm{ w}_{H} \norm{u_{n}-  \bar{u}}_{U} \\
			\leq & \frac{L_0+\eta_n }{\alpha}(\epsilon_n +L_0\norm{u_{n} - \bar{u}}_{U})
			+  \norm{w}_{H}\eta_n + L_1 \norm{ w}_{H} \norm{u_{n}-  \bar{u}}_{U}.
		\end{aligned}
	\end{equation*}	
	Moving all terms that involve $\norm{u_{n}-  \bar{u}}_{U} $ to the left-hand side we get
	\[ (1- \frac{L_0(L_0+\eta_n)}{\alpha } -L_1\norm{w}_{H}) \norm{ u_{n} - \bar{u}}_{U} \leq \frac{L_0}{\alpha}\epsilon_n +\frac{\epsilon_{n}\eta_n}{\alpha}+ \norm{w}_{H}\eta_n.  \]
	Finally, using $w=\frac{Q(\bar u)-g}{\alpha}$ we find \eqref{eq:error_bound2}.
\end{proof}

Observe that for $Q(\bar{u})=g$ (perfect matching) 
the a priori bound is essentially controlled by $\epsilon_n$  only:
\[\norm{u_{n}-  \bar{u}}_{U} \leq  \frac{L_0+\eta_n  }{ \alpha -L_0(L_0+\eta_n) }\epsilon_{n}.  \]
Note further that the error bound depends on a sufficiently large $\alpha$ such that \eqref{eq:Lip_const_condition2} is satisfied.

In the special case where $\mathcal{C}_{ad}$ is redundant, i.e., when $\mathcal{J}'(\bar{u})=0$, improved error bounds can be derived.
This is in particular true for perfect matching which also allows to relax the conditions on $\alpha$. 
\begin{theorem}\label{thm:error_bound}
	Let the assumptions of Theorem \ref{thm:convergence} hold and suppose that the Lipschitz condition  \eqref{eq:second_Lip} is satisfied with the constant $L_1$ such that
	\begin{equation}\label{eq:Lip_const_condition}
		L_1 \norm{Q(\bar{u})-g}_H<  \alpha.
	\end{equation}
	If $\mathcal{J}'(\bar u)=0$, then for  sufficiently large $n\in\mathbb{N}$ we have the following error bound 
	\begin{equation}\label{eq:error_bound}
		\norm{u_{n}- \bar{u}}_U\leq \sqrt{ \frac{3}{\alpha -L_1 \norm{g-Q(\bar{u})}_H}}  \sqrt{ \epsilon_n^2   +2\norm{Q(\bar{u})-g}_H^2}.
	\end{equation}
\end{theorem}	
\begin{proof}
	Since $u_{n}$ is a minimizer for every $n\in\mathbb{N}$,  we have that $\mathcal{J}_n(u_n)\leq \mathcal{J}_n(\bar u)$ with $\mathcal{J}_n(u):=J(Q_n(u),u)$.
	Adding $\frac{\alpha}{2}(\norm{ u_n- \bar{u}}_{U}^2 -  \norm{u_n}_{U}^2)$ to both sides of the inequality gives 
	\begin{equation}\label{eq:error1}
		\frac{1}{2} \norm{Q_{n}(u_{n}) -g }_{H}^2+\frac{\alpha}{2}\norm{ u_n- \bar{u}}_{U}^2 \leq  \frac{1}{2} \norm{Q_{n}(\bar{u}) -g}_{H}^2+\alpha \langle \iota_U^{-1} \bar{u},\bar{u}-  u_{n} \rangle_{U^\ast,U}.
	\end{equation}	
	Using Theorem \ref{thm:convergence}, Taylor's expansion and \eqref{eq:second_Lip}, we get for  sufficiently large $n\in\mathbb{N}$
	\[Q(u_{n}) - Q(\bar{u})=Q^\prime(\bar{u})(u_{n}-\bar{u}) +q(u_{n},\bar{u}), \text{
		where } \norm{q(u_{n},\bar{u})}_{H}\leq  \frac{L_1}{2}\norm{ u_{n}-\bar{u}}_{U}^2 .\]

	By our assumptions and first-order optimality we have 
	$\bar{u} =   -\iota_U(Q^\prime(\bar{u}))^\ast w$   
	where  	$w=\alpha^{-1}(Q(\bar{u})-g)$
	with  $L_1 \norm{w}_{H}< 1$ because of \eqref{eq:Lip_const_condition}.
	This leads to 
	\begin{equation}\label{eq:error2}
		\begin{aligned}
			& \langle \iota_U^{-1} \bar{u},\bar{u}-  u_{n} \rangle_{U^\ast,U} =  \left( -w,Q^\prime(\bar{u})(\bar{u}-  u_{n}) \right)_H \leq  \norm{w}_{H}\norm{Q^\prime(\bar{u})(\bar{u}-  u_{n})}_{H} \\
			\leq &  \norm{w}_{H}\left(\frac{L_1}{2}\norm{ u_{n}-\bar{u}}_{U}^2
			+ \norm{Q(u_{n}) - Q_{n}(u_{n}) }_{H} + \norm{ Q_{n}(u_{n}) -g }_{H} +\norm{g - Q(\bar{u})}_{H} \right)\\
			\leq &  \frac{\norm{w}_{H}L_1 }{2}\norm{ u_{n}-\bar{u}}_{U}^2 + \frac{1}{2}\left(\alpha \norm{w}^2  + \frac{1}{\alpha}\norm{ Q_n(u_{n}) -g }^2 \right)  \\
			& + \left(\alpha \norm{w}_{H}^2  + \frac{1}{2\alpha} \norm{Q(u_{n}) - Q_n(u_{n}) }_{H}^2   + \frac{1}{2\alpha}\norm{g - Q(\bar{u})}_{H}^2  \right) ,
		\end{aligned}
	\end{equation}
	where we have used the identity $ab\le \frac{1}{2\alpha} a^{2}+ \frac{\alpha}{2} b^{2} $.
	Returning to \eqref{eq:error1} and using \eqref{eq:error2}, 
	we derive
	\begin{equation}
		\label{eq:error3}
		\begin{aligned}
			\norm{ u_n- \bar{u}}_{U}^2 \leq &\frac{1}{\alpha} \norm{Q_{n}(\bar{u}) -g }_{H}^2+ 
			\norm{w}_{H}L_1 \norm{ u_{n}-\bar{u}}_{U}^2 + 3\alpha  \norm{w}_{H}^2 \\
			& + \frac{1}{\alpha }( \norm{Q(u_{n}) - Q_n(u_{n}) }_{H}^2   + \norm{g - Q(\bar{u})}_{H}^2)\\
			\leq & \frac{2}{\alpha }\norm{Q_{n}(\bar{u}) -  Q(\bar{u})}_{H}^2 + 
			\norm{w}_{H}L_1 \norm{ u_{n}-\bar{u}}_{U}^2 + 3\alpha \norm{w}_{H}^2 \\
			& + \frac{1}{\alpha}\norm{Q(u_{n}) - Q_n(u_{n}) }_{H}^2   +\frac{3}{\alpha} \norm{g - Q(\bar{u})}_{H}^2.
		\end{aligned}
	\end{equation}
	
	Taking into account \eqref{eq:Lip_const_condition}, we get
	\[
	\begin{aligned}
	\norm{ u_n- \bar{u}}_{U}^2 \leq \frac{1}{\left( 1- \norm{w}_{H}L_1 \right)} \frac{3}{\alpha}\left( \epsilon_n^2 + \alpha^2\norm{w}_{H}^2 +\norm{g - Q(\bar{u})}_{H}^2\right), 
	\end{aligned}
	\]
	for  sufficiently large $n\in\mathbb{N}$.
	Replacing now $\norm{w}_{H}$ by $\frac{\norm{g-Q(\bar{u})}_{H}}{\alpha}$ yields \eqref{eq:error_bound}.
\end{proof}
Note that in the case of perfect matching $Q(\bar{u})=g$, \eqref{eq:error_bound} becomes 
\begin{equation}\label{eq:error_bound_zero_res}
	\norm{u_{n}- \bar{u}}_{U}\leq \epsilon_n\sqrt{\frac{3}{\alpha }}\quad\text{for sufficiently large }n\in\mathbb{N}.
\end{equation}
As stated earlier, our aim is to use approximations $Q_n=Q_{\mathcal{N}_n}=A\Pi_{\mathcal{N}_n}$ resulting from artificial neural networks to replace the partially unknown exact control-to-state map $\Pi$ and $Q=A\Pi$. Therefore, we next collect some fundamental properties of such neural network based approximations.

\section{A brief primer on artificial neural networks (ANNs)}\label{sec:ANN}
Here, we briefly review some (well-known) results for ANNs as they will be useful in what follows. For more introduction on ANNs, one may refer to many textbooks of this topic, e.g., \cite{GooBenCou16}.
We recall that a standard feedforward ANN with one hidden layer is  a function $\mathcal{N}: \R^{r}\to \R^{s}$ of  the following structure: 
\begin{equation}\label{NN_def_section}
	\mathcal{N}(x)= W_0\sigma (W_1x+b_1)+b_0,\quad x\in \R^{r},
\end{equation}
where $W_1\in \R^{l\times r}$, $b_1 \in \R^{l}$,  $W_0\in \R^{s\times l}$ and $b_0\in \R^{s}$. In that case we say that the hidden layer has $l$ \emph{neurons}.
Here, $\sigma: \R \to \R$ is an infinitely differentiable activation function which acts component-wise on a vector in $\R^l$. 
In the output layer, the activation function is usually the identity map, therefore ignored in \eqref{NN_def_section}, while in the other hidden layers, it involves nonlinear transformations.
Some standard smooth activation functions are the following ones:
\begin{itemize}
	\item[$\bullet$]  Sigmoid: a term denoting a family of functions, e.g., tansig ($\sigma(z)=\frac{e^{z}-e^{-z}}{e^{z}+e^{-z}}$), logsig ($\sigma(z)=\frac{1}{1+e^{-z}}$)), arctan ($\sigma(z)=\arctan(z))$,  etc.
	\item[$\bullet$]  Probability functions, e.g., softmax ($\sigma_i(z)=\frac{e^{-z_i}}{\sum_j e^{-z_j}}$). Here the index $i$ denotes the $i$-th neuron in a given layer, with the summation indexed by  $j$ being taken over all the neurons of the same layer.
\end{itemize}
We see that for the softmax function, neurons of the same layer may have different activition functions. Notice that the smoothness of the activation function is the one that determines the smoothness of $\mathcal{N}$. 

Next we state  a classical result, see, for instance, \cite[Theorem 3.1]{Pin99}. Below ``$\cdot$''  denotes the standard inner product in the underlying Euclidean space.
\begin{theorem}\label{thm:function_app}
	Let $\sigma \in C(\R)$ and consider the set
	\[R_\sigma:=\set{\mathcal{N}:\R^r \to \R \,|\, \mathcal{N}(x)=   w_0\cdot \sigma(W_1  x+ b_1) ,\text{ with }  w_0 \in \R^l, \; W_1\in \R^{l\times r},\;   b_1\in \R^l}.\]
	Then $R_\sigma$ is dense in $C(\R^r)$ in the topology of uniform convergence on compact sets if and only if $\sigma$ is not a polynomial function.
\end{theorem}
Hence, for any $\epsilon>0$, and for any given function $f\in C(K)$, $K\subset \R^r$ compact, there exists a function $\mathcal{N}=\mathcal{N}^\epsilon \in R_\sigma$ such that
\[ \max_{ x\in K}\abs{f( x )- \mathcal{N}^\epsilon( x)}< \epsilon.\]
This approximation property can be also carried over to the derivatives of a given smooth function; see, e.g., \cite[Theorem 4.1]{Pin99}.
\begin{theorem}\label{thm:deriv_app}
	Let  $\,m=\max \set{\abs{ m^i}:\; i=1,2,\ldots, s}$, where each $m^{i}$ is a standard differentiation multi-index, and define $C^{m^1, \ldots , m^s}(\R ^{r}):=\bigcap_{i=1}^s C^{m^i}(\R^r)$.
	Then $R_\sigma$ is dense in $C^{m^1, \ldots , m^s}(\R^r)$ if $\sigma\in C^m(\R)$ is not a polynomial function.
\end{theorem}
As a consequence, for any $f\in C^{m^1, \ldots , m^s}(K)$, for every compact $K\subset \R^{r}$ and every $\epsilon>0$, there exists a function $\mathcal{N}=\mathcal{N}^\epsilon\in R_\sigma$  such that
\[\max_{ x\in K}\abs{D^{k} f( x)- D^{ k} \mathcal{N}^\epsilon( x)}< \epsilon,\]
for all multi-indices $ k$ such that $ 0\leq k \le  m^{i}$ for some $i$.

Note that 
these results imply analogous error bounds for \eqref{NN_def_section}, i.e., for the vector-valued case. They can be also  generalized to mutiple-hidden-layer networks  as the next theorem shows, see  \cite{LesLinPinSch93}.
\begin{theorem}\label{thm:mul_lay_deriv_app}
	A standard multi-layer feedforward network with a continuous activation function can uniformly approximate any continuous function to any degree of accuracy  if and only if its activation function is not a polynomial.
\end{theorem}

One of the main tasks of deep learning, a specific branch of machine learning, is to identify suitable choices for $W_0\in\mathbb{R}^{s\times l_{\ell}}$, $W_1\in \mathbb{R}^{l_{1}\times r}$, $W_i\in\mathbb{R}^{l_{i}\times l_{i-1}}$ for $i=2,\ldots,\ell$, and  $b_0\in \mathbb{R}^{s}$, $b_i\in\mathbb{R}^{l_i}$, where $i=1,\ldots,\ell$ represents the $i$-th hidden layer of the underlying ANN, from a given data set $D=\{(x_j,f_j)\in\mathbb{R}^r\times\mathbb{R}^s:j=1,\ldots,n_D\}$, with $n_D\in\mathbb{N}$ sufficiently large. A typical approach in this context seeks to find a (global) solution to the nonconvex minimization problem
\begin{equation}\label{mh.ann.min}
	\text{minimize }\sum_{j=1}^{n_D}\mathfrak{d}(\mathcal{N}(x_j),f_j)+\mathfrak{r}(W,b)\quad\text{over }(W,b)\in\mathcal{F}_{ad},
\end{equation}
where $\mathcal{N}$ results from a multi-layer ANN that depends on $\Theta:=(W,b)$, with $W:=(W_{0}, \ldots, W_{\ell})$ and $b:=(b_{0}, b_{1}, \ldots, b_{\ell})$. Further, $\mathfrak{d}$ denotes a suitable distance measure, $\mathfrak{r}$ is an optional regularization term inducing some a priori properties of $\Theta$, and $\mathcal{F}_{ad}$ encodes possible additional constraints. While the study of \eqref{mh.ann.min} is an interesting and challenging subject in its own right, here we rather assume that the learning process, i.e., the computation of a suitable $\Theta$, has been completed. We then study analytical properties of the resulting $\mathcal{N}$, or the solution map $\Pi_\mathcal{N}$ or $Q_\mathcal{N}$ in view of \eqref{eq:optimal_control_reduced}, in the context of our target applications and report on associated numerical results.

\section{Application: Distributed control of semilinear elliptic PDEs }
\label{sec:appl_1}
In our first application we consider the following model problem associated with the distributed optimal control of a semilinear elliptic PDE:
\begin{align}\label{eq:cost}
	&\text{minimize}\quad J(y,u):= \frac{1}{2}\|y-g\|_{L^2(\Omega)}^{2} +\frac{\alpha}{2} \|u\|_{L^2(\Omega)}^{2},\quad\text{over }\;(y,u)\in  H^1(\Omega)\times L^2(\Omega)\\
	&\text{s.t. }\quad \label{eq:state_eq}
	-\Delta y + f(x,y)=u\;\;  \text{ in }\;\om,\quad \partial_{\nu}y=0\;\;  \text{on }\; \partial \om,\\
	&\label{eq:control_constr} \phantom{\text{s.t. }}\quad\; u\in \mathcal{C}_{ad}:=\{v\in L^2(\Omega):\underline{u}(x)\le v(x) \le \overline{u}(x),\quad \text{for a.e. }x\in\om\},
\end{align}
where $\underline{u},\overline{u}$ with $\underline{u}\leq \overline{u}$ belong to $L^\infty(\Omega)$, and 'a.e.' stands for 'almost every' in the sense of the Lebesgue measure. Moreover, we have $g\in L^{2}(\om)$, and $\om \subset \R^{d}$, $d\ge 2$, is a bounded domain with Lipschitz boundary. In view of our general model problem class \eqref{eq:optimal_control} we have $H=U=L^2(\Omega)$, $Y=H^1(\Omega)$, $Z=H^{-1}(\Omega)$, $A=\operatorname{id}$, and $e$ is given by the PDE in \eqref{eq:state_eq}. For more details on the involved Lebesgue and Sobolev spaces we refer to \cite{MR2424078}. Concerning $f$ we invoke the following assumption throughout this section:

\begin{assumption}\label{assu:non_monotone}
	The nonlinear function $f=f(x,z):\Omega \times \R \to \R$ is measurable with respect to $x$ for every $z\in \R$ and continuously differentiable with respect to $z$ for almost every $x\in\om$.
	There exists a function $F:\Omega \times \R \to \R$ so that $\partial_z F(\cdot,z)=f(\cdot,z)$. $F$ and $f$ satisfying the following  conditions, for all $z\in\R$
	\begin{align}\label{eq:growth_rate_Ff}
		\abs{f(\cdot,z)}\le  b_1+ c_{1}\abs{z}^{p-1}\quad \text{ and }\quad 
	   -f(\cdot,z)z+F(\cdot,z)\leq b_2,
	\end{align}
	which combined also result to
	 \begin{equation}\label{eq:growth_rate_Fminusf}
		F(\cdot,z)\leq b_0+c_{0}\abs{z}^{p},
	\end{equation}
	for some constants $b_{0}, b_{1}, b_{2}\in\R$ and $c_{0}, c_{1}>0$ and for $p$ with  $1<p\leq \frac{2d}{d-2}$ for $d\geq 3$,   $ 1<p<+\infty$ for $d=2$, or  $1< p \leq +\infty$ for $d=1$.The interpretation of $p=\infty$ for $d=1$ is that the growth conditions in \eqref{eq:growth_rate_Ff} are not required to hold.
	Finally, we assume that $F$ is coercive in the sense that  $\lim_{\norm{y}_{L^p(\Omega)}\to \infty} \frac{\int_{\Omega}F(x,y)dx}{\norm{y}_{L^p(\Omega)}} \to \infty$, and $F$ is bounded from below, i.e., $F(x,z)\geq F_0$  for some $F_0\in \R$, for all $z\in\R$ and for almost every $x\in\om$.
\end{assumption}
The above assumption particular indicates that both $f$ and $F$ satisfy the Carath\'eodory condition, and thus induce some operators of Nemytskii type. 

Moreover, observe also that the conditions on $p$ enable the embedding $H^1(\Omega)\subset L^{p}(\Omega)$.
Also note that the Assumption \ref{assu:non_monotone} is satisfied for $F(x,z)=\alpha(x) \pi_{p}(z)$ with $\alpha\in L^{\infty}(\om)$ and $\alpha(x)>0$ for almost every $x\in\om$ and $\pi_{p}$ being a polynomial of degree $p$ and positive coefficient on the term of degree $p$; the latter being  equal to $|z|^{p}$ if $p$ is odd such that the coercivity assumption is not violated.

Given the above assumption, the PDE \eqref{eq:state_eq} is related to the variational problem
\begin{equation}\label{eq:non_convex_variational}
	\text{minimize}\quad G(y):=\frac{1}{2} \|\nabla y\|_{L^{2}(\Omega)}^{2} +\int_\Omega F(x,y)\,dx-\int_{\Omega} uy\,dx\quad\text{ over }\;y\in H^{1}(\Omega).
\end{equation}
A particular example is given by a Ginzburg-Landau model for superconductivity where
$f(z) = \eta^{-1}(z^3-z)$ with a parameter $\eta>0$. It gives rise to the double-well type variational model 
\begin{equation}\label{eq:double_well_variational}
	\text{minimize}\quad \frac{1}{2} \|\nabla y\|_{L^{2}(\Omega)}^{2} +\frac{1}{4\eta}\int_{\Omega} (y^{2}-1)^{2}dx -\int_{\Omega} uy\,dx\quad\text{ over }\;y\in H^{1}(\Omega),
\end{equation}
for given $u\in L^2(\Omega)$ or in fact, to a more a general space. The next proposition shows  existence of solutions for  \eqref{eq:non_convex_variational}.

\begin{proposition}\label{prop:existence_variation}
	Let Assumption \ref{assu:non_monotone} hold, and suppose that $u\in L^r(\Omega)$ for some  $r \geq \frac{p}{p-1}$.
	Then the optimization problem \eqref{eq:non_convex_variational} admits a solution in $H^1(\Omega)$.
\end{proposition}
\begin{proof}
	Notice that due to the coercivity assumption we can find a $C>0$ such that $\|u\|_{L^{r}(\om)}<CC_{1}$ with $C_{1}$ being the constant involved in the embedding $L^{p}(\om) \subset L^{\frac{r}{r-1}}(\om)$ such that
	\begin{equation}\label{F_coercive}
		\begin{aligned}
			\int_{\om} F(x,y)\, dx- \int_{\om} uy\, dx 
			&\ge C\|y\|_{L^{p}(\om)}-\|u\|_{L^{r}(\om)} \|y\|_{L^{\frac{r}{r-1}}(\om)}\\
			&\ge (CC_{1} -\|u\|_{L^{r}(\om)}) \|y\|_{L^{\frac{r}{r-1}}(\om)}\ge 0, 
		\end{aligned}
	\end{equation}
	provided $\|y\|_{L^{p}(\om)}$ is large enough. This together with the lower bound $F\ge F_{0}$ implies that the energy $G$ is bounded from below and thus there is an infimizing sequence $(y_{n})_{n\in\N}\in H^{1}(\om)\subset L^{p}(\om)$. Using the above inequality one easily deduces that $\|y_{n}\|_{L^{\frac{r}{r-1}}(\om)}$ is bounded, and with the help of the Poincar\'e inequality a uniform $H^{1}(\om)$ bound is also obtained for that sequence. 
	Therefore, we only need to show that $G(\cdot)$ is weakly lower semicontinuous in $H^{1}(\om)$.
	For this, it suffices to check the term involving $F$, since the arguments for the other two terms are straightforward.
	Assuming $y_{n}\rightharpoonup y$ in $H^{1}(\om)$,  by the compact embedding of $H^{1}(\Omega){\hookrightarrow} L^{1}(\Omega)$, we have that 
	$y_{n}\to y$ almost everywhere, {up to a subsequence}. Due to the continuity of $F$ with respect to the second variable, we have $F(\cdot,y)=\lim_{n\to \infty} F(\cdot,y_n) $ almost everywhere.
	Since $F(\cdot,y_n),F(\cdot,y)\geq F_0$, by Fatou's lemma we have
	\[\int_{\Omega}F(x,y)\,dx\le \liminf_{n\to\infty} \int_{\Omega} F(x,y_n)\,dx, \]
	and thus $G(\cdot)$ is weakly lower semicontinuous.
\end{proof}
Before we proceed, it is useful to recall the following standard result on linear elliptic PDEs  \cite{Eva10,Tro10}.
\begin{theorem}\label{thm:exi_linear}
	Let $ v\in L^r(\Omega)$, $a\in L^\infty(\Omega)$ with $a>0$. Then the following equation admits a unique solution
	\[ -\Delta s + a  s =v\quad \text{ in }\;\Omega, \qquad \partial_\nu s=0\;\text{ on }\;\partial \Omega.   \]
	Furthermore there exist constants $C_h>0$ and $C_l >0$ independent of $a$ and $v$ such that
	\begin{equation}\label{eq:linearPDE_energy}
		\norm{s}_{H^1(\Omega)} \leq C_h \norm{v}_{L^r(\Omega)}   \quad \text{ and }  \quad \norm{s}_{C(\overline{\Omega})} \leq C_l \norm{v}_{L^r(\Omega)}.
	\end{equation}
\end{theorem}
Using the polynomial growth of $F$ together with the continuous embedding $H^1(\Omega)\subset L^{\frac{r}{r-1}}(\Omega)$, one verifies the Fr\'echet differentiability of $G:H^1(\Omega)\to\mathbb{R}$. The Euler-Lagrange equation associated with \eqref{eq:non_convex_variational} is given by
\begin{equation}\label{eq:non_monotone_y0}
	-\Delta y+f(x,y)=u \quad \text{ in }\;\Omega,\qquad \partial_\nu y=0\;\text{ on }\;\partial\Omega,
\end{equation}
and it is satisfied for every solution $y$ of \eqref{eq:non_convex_variational}.
Under Assumption \ref{assu:non_monotone},  the solutions of  \eqref{eq:non_monotone_y0} can be uniformly bounded with respect  to $\|\cdot\|_{C(\overline{\om})}$, as shown next.
\begin{proposition}\label{lem:uniform_C_norm_bounds}
	Let the Assumption \ref{assu:non_monotone} be satisfied, and let $\mathcal{C}_{ad}\subset L^{\infty}(\Omega)$ be bounded. Then there exists a constant $K>0$ such that for all solutions of \eqref{eq:non_monotone_y0}, it holds 
	\begin{equation}\label{eq:y0_C_estimate}
		\|y\|_{H^{1}(\om)}+\|y\|_{C(\overline{\Omega})}\le   K,  \quad \text{ for all } \;  u\in  \mathcal{C}_{ad}.
	\end{equation} 
\end{proposition}
\begin{proof}
	From the fact that $y\in L^{p}(\Omega)$,  the growth condition \eqref{eq:growth_rate_Ff} and the  measurability of $f$, we have $f(\cdot,y)\in L^{\frac{p}{p-1}}(\Omega)$. 
	We can rewrite \eqref{eq:non_monotone_y0} in the following form
	\begin{equation}\label{eq:re_elliptic_y0}
		-\Delta y+\epsilon y=u +\epsilon y - f(x,y) \quad \text{ in }\;\Omega,\qquad \partial_\nu y=0\;\text{ on }\;\partial\Omega,
	\end{equation}
	for some $\epsilon>0$.
	Let us define $\tilde{r}:= \min\set{\frac{r}{r-1}, \frac{p}{p-1} }$. 
	Then $u+\epsilon y+ f(\cdot,y)\in  L^{\tilde{r}}(\Omega) $ since $u\in \mathcal{C}_{ad}\subset L^\infty(\Omega)$.
	Applying \eqref{eq:linearPDE_energy} to \eqref{eq:re_elliptic_y0} yields 
	\begin{equation}\label{eq:est_y0}
		\|y\|_{H^{1}(\Omega)}+ \|y\|_{C(\overline{\Omega})}
		\leq   (C_h+C_l) \left(\norm{u}_{L^{\tilde{r}}(\Omega)} +\epsilon\norm{y}_{L^{\tilde{r}}(\Omega)} +  \norm{f(\cdot, y)}_{L^{\tilde{r}}(\Omega)}\right) .
	\end{equation}
	As all solutions of \eqref{eq:re_elliptic_y0} are stationary points of $G$, in view of \eqref{eq:growth_rate_Ff}, every weak solution $y$ satisfies
	\begin{equation}\label{eq:Lp_energy_bounds}
		G(y)=\frac{1}{2} \|\nabla y\|_{L^{2}(\Omega)}^{2} +\int_\Omega F(x,y)\,dx-\int_{\Omega} uy = \int_\Omega -f(x,y)y+F(x,y)\,dx\leq b_{2} |\om|,
	\end{equation}
	where we use the weak formulation of \eqref{eq:re_elliptic_y0} tested with $y$.

	Using the coercivity of $G$, we can find some constant $M>0$ independent of $y$ such that $\norm{y}_{L^p(\Omega)}\leq M$. 
	Since $(p-1)\tilde{r}\leq p$, by \eqref{eq:growth_rate_Ff}, we have
	\begin{equation}\label{eq:y0_C_estimate_01}
		\norm{f(\cdot,y)}_{L^{\tilde{r}}(\Omega)}\leq  d_0+d\norm{y^{p-1}}_{L^{\tilde{r}}(\Omega)}\leq d_0+\tilde{d}\norm{y}^{p-1}_{L^p(\Omega)}\le \tilde{M}.
	\end{equation}

	Returning to \eqref{eq:est_y0}, we choose a sufficiently small $\epsilon>0$ such that the second term on the right-hand side of \eqref{eq:est_y0} is absorbed by $\norm{y}_{H^1(\Omega)}$. Since $L^{\tilde{r}}(\Omega)\subset L^\infty(\Omega)$ and $\mathcal{C}_{ad}$ is bounded, $\norm{u}_{L^{\tilde{r}}(\Omega)}$ is uniformly bounded for all $u\in \mathcal{C}_{ad}$. Finally, taking into account \eqref{eq:est_y0} and \eqref{eq:y0_C_estimate_01} we have
	\begin{equation}\label{eq:y0_C_estimate_02}
		\|y\|_{H^{1}(\Omega)} +\norm{y}_{C(\overline{\Omega})} \leq  (\tilde{C}_h+\tilde{C}_l) (\norm{u}_{L^{\tilde{r}}(\Omega)} + \tilde{M})\le  K,
	\end{equation}
	which is the conclusion.
\end{proof}
Notice that for monotone $f$, one can directly refer to standard results in the literature, e.g., \cite{Tro10}, where uniform bounds on the solution of \eqref{eq:non_monotone_y0} are shown for that case.

\subsection{Continuity and sensitivity of the control-to-state map}
Since $f(\cdot,\cdot)$ might be nonmonotone with respect to the second variable, this  may give rise to a lack of uniqueness of a solution to the semilinear PDE \eqref{eq:state_eq}.  In the monotone case, the continuity result is more direct to show, thus we focus on the nonmonotone case here.

Under our standing assumptions, \eqref{eq:state_eq} has a nonempty  set of solutions $y$ satisfying
\[\|y\|_{H^{1}(\Omega)} + \norm{y}_{C(\overline{\Omega})}\leq K\]
for some constant $K$ independent of $u$ since $\mathcal{C}_{ad}$ is bounded. The associated continuity result stated next, relies on a $\Gamma$--convergence technique. We note that for this section we take $r=2$.
\begin{proposition}\label{prop:double_well_Gamma_Conv}
	Let $u_{n}\to u$ in $L^{2}(\om)$ and $G_{n},G: H^{1}(\Omega)\to \R$ be the corresponding energies in \eqref{eq:non_convex_variational}. Then $G_{n}$ $\Gamma$--converges to $G$ with respect to the $H^{1}$ topology. Furthermore, $G_{n}$ is equi-coercive. 
\end{proposition}
\begin{proof}
	{Observe first that one easily checks that $G_{n}$ $\Gamma$--converges to $G$. This is because the function $\frac{1}{2}\|\nabla (\cdot)\|_{L^{2}(\om)}^{2}+\int_\om F(x, \cdot)\,dx$ is weakly lower semicontinuous with respect to the $H^{1}(\om)$ convergence (and hence it $\Gamma$--converges to itself), while the function $y\mapsto\int_{\om}u_{n}y\,dx$ continuously converges to the function $y\mapsto \int_{\om}u y\,dx$ (see \cite[Def. 4.7]{dalmasogamma} for the notion of continuous convergence). The assertion follows from the stability of $\Gamma$-convergence under continuous perturbations \cite[Prop. 6.20]{dalmasogamma}.}
	
	In order to see that $G_{n}$ is equi-coercive, it suffices to find a lower semicontinuous coercive function $\Psi: H^{1}(\om) \to \R$ such that $G_{n}\ge \Psi$ on $H^{1}(\om)$, cf. \cite[Prop. 7.7]{dalmasogamma}. This follows from the fact that $(\|u_{n}\|_{L^{2}(\om)})_{n\in\mathbb{N}}$ is a bounded sequence and from the coercivity condition in Assumption \eqref{assu:non_monotone}, see also \eqref{F_coercive}.
\end{proof}

With the help of $\Gamma$--convergence and equi-coercivity one can get the classical results on $\Gamma$--convergence with respect to global and local minimizers. It is of particular interest whether $y_{0}$ is an isolated local minimizer of $G$ (and in particular satisfies \eqref{eq:state_eq}). In this case there exists a sequence $\tilde{y}_{n}$ with $\tilde{y}_{n}\to y_{0}$ in $H^{1}(\om)$ such that for all sufficiently large $n$, $\tilde{y}_{n}$ is a local minimizer of $G_{n}$ (hence it also satisfies  \eqref{eq:state_eq}); see \cite{braides2014convergence}. 
{This implies that if $u_{n}\to u_{0}$ in $L^{2}(\om)$ and $y_{0}\in \Pi (u_{0})$ is an isolated local minimizer of $G$, then there exists a sequence $(y_{n})_{n\in\N}$ in $H^{1}(\om)$ such that $y_{n}\in \Pi (u_{n})$ and $y_{n}\to y_{0}$ in $H^{1}(\om)$.}

\begin{remark}\label{rem:local_minimizer}
	We note that solutions of the PDE \eqref{eq:state_eq} are not necessarily local minimizers of the variational problem \eqref{eq:non_convex_variational}. In order to make sure that $y_0$ is an isolated local minimizer, one can check second-order conditions on \eqref{eq:non_convex_variational}. In this context, second-order sufficiency relates to 
	$(s,-\Delta s + \partial_y f(\cdot,y_0)s)>\epsilon \norm{s}^2_{H^1(\Omega)}$ for all $s\in H^1(\Omega)$ with some $\epsilon>0$.
	Therefore, if $f(\cdot,\cdot)$ is a strictly monotone function with respect to its second variable, then the positive definiteness condition is automatically guaranteed.
	For the more general case, it turns out that a similar, but yet milder condition (see \eqref{eq:smallness} below) helps to establish the sensitivity result for the control-to-state map.
\end{remark}

Given this approximating sequence $(y_{n})_{n\in\N}$ for $y_{0}\in \Pi(u_{0})$, convergence rates and differentiability of the control-to-state map in a certain sense are shown next. For this, we also assume that 
\begin{equation}\label{lip_partialy_f}
\text{$\forall \,M>0\:$ $\exists\, L_{M}>0\,$: }	\text{$|\partial_{y}f(x,y_{1})-\partial_{y}f(x,y_{2})|\le L_{M}|y_{1}-y_{2}|$,  }
\end{equation}
for almost every $x\in\om$ and for all $y_{1}, y_{2}\in [-M,M]$.
This also implies
\begin{equation}\label{partialz_f_bounded}
	\text{$\forall \,M>0\:$ $\exists\, C>0\,$: }\quad |\partial_{y}f(x,y)|<C \text{ for a.e. $x\in\om$ and  $\forall \,y\in[-M,M]$.}
\end{equation}

\begin{theorem}\label{thm:continuity_multi_map}
	Assume that \eqref{lip_partialy_f} holds for $f$, let $\Pi: L^{2}(\Omega) \rightrightarrows H^{1}(\Omega)$ be the possibly multi-valued control-to-state map of \eqref{eq:state_eq} and fix some $u_{0}, h\in L^{2}(\Omega)$ as well as $y_{0}\in \Pi(u_{0})$. 
	Define $(\partial_y  f(\cdot,y_0))^-:=\min \set{\partial_y  f(\cdot,y_0),0}$, and assume that
	\begin{equation}\label{eq:smallness}
		\norm{(\partial_y  f(\cdot,y_0))^-}_{L^2(\Omega)}< \frac{1}{C_l} \quad \text{  and }  \quad 	\norm{(\partial_y  f(\cdot,y_0))^-}_{L^\infty(\Omega)}< \frac{1}{C_h},
	\end{equation}
	where $C_l$ and $C_h$ are the positive constants defined in \eqref{eq:linearPDE_energy}.
	Suppose  $u_{n}=u_{0}+t_{n}h$  for a sequence $t_{n}\to 0$, and suppose there exists $y_{{n}}\in \Pi(u_{{n}})$ with $y_{{n}}\to y_0$ in $H^{1}(\om)$. Then we have
	\begin{equation}\label{multi_valued_cnt}
		\|y_{{n}}-y_{0}\|_{H^{1}(\om)}\le C t_{n},
	\end{equation}
	for some constant $C$ and large enough $n\in\mathbb{N}$. Moreover, one has that every weak cluster point of $\frac{y_{{n}}-y_{0}}{t_{n}} $, denoted by $p$,
	solves the following linear PDE
	\[
	-\Delta p +\partial_y  f(\cdot,y_0)p=h \quad  \text{ in } \;\Omega,\qquad
	\partial_\nu p=0  \;  \text{ on }\; \partial \Omega.
	\]
	In particular, for every $h\in L^2(\Omega)$, $p$ satisfies the energy bounds:
	\begin{equation}\label{eq:adjoint_bounds}
		\norm{p}_{H^1(\Omega)}\leq C_H\norm{h}_{L^2(\Omega)} \quad   \text{ and }   \quad	\norm{p}_{C(\overline{\Omega})}\leq C_c\norm{h}_{L^2(\Omega)} ,
	\end{equation}
	with constants $C_H$ and $C_c$ depending on $C_h$ and $C_l$.
\end{theorem}

\begin{proof}
	Subtracting  the  equations that correspond to the pairs  $(u_{n},y_{n})$ and $(u_{0}, y_{0})$ and using the mean value theorem,  we get
	\begin{equation}\label{eq:difference}
		-\Delta (y_{{n}}-y_{0}) =t_{n}h +  f(\cdot,y_{0})- f(\cdot,y_{{n}})= t_n h - \partial_y  f(\cdot,y_0+\gamma_h(y_{{n}}-y_{0})) (y_{{n}}-y_{0}),
	\end{equation}
	where $\gamma_h\in L^\infty(\Omega)$ with $\norm{\gamma_h}_{L^\infty(\Omega)}\leq 1$, see Remark \ref{rem:mean_value_thm} regarding measurability of such $\gamma_{h}$. 
	Note that $y_{n},y_0\in C(\overline{\Omega})$ with a uniform bound $K>0$, therefore from \eqref{partialz_f_bounded} we have  $\partial_y  f(\cdot,y_0+\gamma_h(y_{{n}}-y_{0}))\in L^\infty(\Omega)$.
	Then, given $\epsilon>0$,  we rewrite \eqref{eq:difference} as 
	\begin{equation}\label{eq:difference1}
		-\Delta (y_{{n}}-y_{0})+(\epsilon+(\partial_y  f(\cdot,\xi_n^h))^+)(y_{{n}}-y_{0})=t_{n}h + (\epsilon+(\partial_y  f(\cdot,\xi_n^h))^+ -\partial_y  f(\cdot,\xi_n^h) )(y_{{n}}-y_{0}),
	\end{equation}
	where $\xi_n^h:=y_0+\gamma_h(y_{{n}}-y_{0})$, and $(\partial_y  f(\cdot,\xi_n^h))^+=\max\set{ \partial_y  f(\cdot,\xi_n^h),0}$.  
	Now, using \eqref{eq:linearPDE_energy}, we have
	\begin{equation}\label{eq:pde_residual}
		\begin{aligned}
			&\frac{ \epsilon}{C_h} \norm{y_{n}-y_0}_{H^1(\Omega)}+\norm{y_{n}-y_0}_{L^\infty(\Omega)} \\ \leq& (\epsilon+C_l)\left(t_{n}\norm{h}_{L^2(\Omega)} +\norm{ (\epsilon+(\partial_y  f(\cdot,\xi_n^h))^+ -\partial_y  f(\cdot,\xi_n^h))(y_{{n}}-y_{0})}_{L^2(\Omega)}\right)\\
			\leq &(\epsilon+C_l) \left(t_{n}\norm{h}_{L^2(\Omega)} + \norm{\epsilon +(\partial_y  f(\cdot,\xi_n^h))^-}_{L^2(\Omega)} \norm{y_{{n}}-y_{0}}_{L^\infty(\Omega)}\right).
		\end{aligned}
	\end{equation}
	The last inequality holds since both $y_n$ and $y_0$  are $C(\overline{\Omega})$ functions.
	Because  $y_n \to y_0$ in $H^1(\Omega)$, we also have that $\xi_n^h\to y_0$ in $L^2(\Omega)$.
	From the continuity of $\partial_y  f(x,\cdot)$, the fact that $y_{n}$, $y_{0}$ are uniformly bounded in $C(\overline{\om})$ and from  dominated convergence,  we have that $\partial_y  f(\cdot,\xi_n^h)\to \partial_y  f(\cdot,y_0)$ in $L^2(\Omega)$.
	Thus, because of \eqref{eq:smallness}, there exists $\epsilon=\epsilon_0$ small enough, such that for sufficiently large $n$, we have $(\epsilon_0+C_l) \norm{\epsilon_0 +(\partial_y  f(\cdot,\xi_n^h))^-}_{L^2(\Omega)}\leq 1$.
	Then \eqref{eq:pde_residual} leads to
	\begin{equation}\label{eq:diff_bounds}
		\norm{y_{n}-y_0}_{H^1(\Omega)} \leq  \frac{C_h(\epsilon_0+C_l)}{\epsilon_0} \norm{h}_{L^2(\Omega)}t_{n}.
	\end{equation}

	From the above inequalities we have that $(\frac{y_{{n}}-y}{t_{n}})_{n\in\mathbb{N}}$ is uniformly bounded in $H^{1}(\om)$ and therefore admits a weakly convergent subsequence (unrelabelled) with weak limit $p$.
	Then, dividing by $t_{n}$ and letting $t_{n}\to 0$ in \eqref{eq:difference}, we have that $p$ satisfies the following equation
	\begin{equation}\label{eq:direc_derivative}
		-\Delta p +\partial_y  f(\cdot,y_{0})p=h \quad  \text{in } \Omega,\quad 
		\partial_\nu  p=0 \;  \text{ on }\; \partial \Omega.
	\end{equation}
	Note that \eqref{eq:diff_bounds} readily implies the first energy bound in \eqref{eq:adjoint_bounds}.
	For the second bound in \eqref{eq:adjoint_bounds}, the procedure is similar. For this we consider 
	\begin{equation}\label{eq:pde_residual2}
		\begin{aligned}
			& \norm{y_{n}-y_0}_{H^1(\Omega)}+\frac{ \epsilon}{C_l}\norm{y_{n}-y_0}_{C(\overline{\Omega})} \\ 
			\leq &(\epsilon+C_h) \left(t_{n}\norm{h}_{L^2(\Omega)} + \norm{\epsilon +(\partial_y  f(\cdot,\xi_n^h))^-}_{L^\infty(\Omega)} \norm{y_{{n}}-y_{0}}_{L^2(\Omega)}\right).
		\end{aligned}
	\end{equation}
	Invoking now the second condition in \eqref{eq:smallness}, and using exactly the same steps as for the first bound of \eqref{eq:adjoint_bounds}, we find some $\epsilon'_0>0$ to conclude the second bound in \eqref{eq:adjoint_bounds} when $n$ is sufficiently large.
\end{proof}
\begin{remark}\label{rem:sensitivity}
	The proof of Theorem \ref{thm:continuity_multi_map} provides an alternative strategy for proving existence and energy estimates of solutions for certain type of linear elliptic PDEs, e.g. as in \eqref{eq:direc_derivative} when the elliptic coercivity is mildly violated. Also note that in the monotone case,  $(\partial_y  f(\cdot,y_0))^-\equiv 0$, and thus the conditions in \eqref{eq:smallness} are always fulfilled.
\end{remark}

\subsection{Existence results for learning-informed semilinear PDEs}
As motivated in the introduction, in many applications the precise form of $f$ is not known explicitly, but rather it can be inferred from given data only. 
Here we are particularly interested in neural networks to learn the hidden physical law or nonlinear mapping from such  data.
The corresponding existence result for PDEs that include such neural network approximations is stated next. 
\begin{proposition}\label{prop:first_aprox}
	Let $f:\Omega \times \R \to \R$  and $F:\Omega\times \R \to \R $ be given as in Assumption \ref{assu:non_monotone} with the extra assumption that $f\in C(\overline{\om}\times \R)$.
	Then, for every $\epsilon>0$ there exists a neural network $\mathcal{N}\in C^{\infty}(\R^{d}\times \R)$ such that
	\begin{equation}\label{g_N}
		\sup_{\| y\|_{L^{\infty}(\Omega)}< K} \|f(\cdot,y)-\mathcal{N}(\cdot,y)\|_{U}<\epsilon,
	\end{equation}
	with $K$ cf. \eqref{eq:y0_C_estimate}. Moreover, the learning-informed PDE
	\begin{equation}\label{eq:nonconvex_learn}
		\begin{aligned}
			-\Delta y + \mathcal{N}(\cdot,y)&=u\quad  \text{ in }\; \Omega,\qquad
			\partial_\nu y=0\;  \text{ on }\; \partial \Omega,
		\end{aligned}
	\end{equation}
	admits a weak solution which also satisfies \eqref{eq:y0_C_estimate} for sufficiently small $\epsilon>0$.
\end{proposition}
\begin{proof}
	From Theorem \ref{thm:function_app} we have that for every $\tilde{\epsilon}>0$ there exists a neural network $\mathcal{N}\in C^{\infty}(\R^{d}\times \R)$ such that $|f(x,y)-\mathcal{N}(x,y)|< \tilde{\epsilon}$ for every $(x,y)\in \overline{\om}\times [-K-1, K+1]$. 

	Thus, the existence of $\mathcal{N}$ such that \eqref{g_N} holds can be directly shown; note that $U=L^{\infty}(\om)$ is feasible in \eqref{g_N}.
	
	Consider next the function $N:\om\times \R \to \R$ given by 
	\[
	N(x,t):=
	\begin{cases}
	\int_{0}^{t} \mathcal{N}(x,s)\,ds + F(x,0), & -(K+1)\le t\le K+1,\\
	r_{0}(x) +F(x,t),    & t>K+1,\\
	r_{1}(x) +F(x,t),	& t<-(K+1),
	\end{cases}
	\]
	with $r_{0}(x):=\int_{0}^{K+1}\mathcal{N}(x,s)\,ds+F(x,0)-F(x,K+1)$, $r_{1}(x):=\int_{0}^{-K-1}\mathcal{N}(x,s)\,ds+F(x,0)-F(x,-K-1)$. Notice that $N(x,t)$ is continuous with $|{{N}}(x,t)-F(x,t)|< \epsilon(K+1)$ for every $t\in \R$ and $x\in \Omega$. Next we apply some  smoothing of $N(x,\cdot)$ in a small neighbourhood of $\om\times \{-K-1\}$ and $\om\times \{K+1\}$ such that the previous approximation estimate still holds true,  and continue to use the symbol $N$ for the result. Then $N(x,\cdot)$ is differentiable with respect to the second variable for every $x\in \Omega$. Consider now the minimization problem 
	\begin{equation}\label{eq:nonconvex_variational_learn}
		\inf_{y\in H^{1}(\Omega)} \frac{1}{2} \|\nabla y\|_{L^{2}(\Omega)}^{2} +\int_{\Omega} {N}(x,y)\,dx -\int_{\Omega} uy\,dx.
	\end{equation}
	One can now prove existence of a solution to \eqref{eq:nonconvex_variational_learn} analogously to the proof of Proposition \ref{prop:existence_variation} for \eqref{eq:non_convex_variational}. 
	We can show that the functional in $y\mapsto \int_{\Omega} N(x,y)\,dx$ is Frech\'et differentiable in $H^{1}(\Omega)$ with Frech\'et derivative $h\mapsto \int_\Omega \partial_y N(x,y)h\,dx$,  see discussion after this proof. Thus any solution to \eqref{eq:nonconvex_variational_learn} satisfies the PDE
	\begin{equation}\label{eq:nonconvex_learn_ex}
		-\Delta y + \partial_y N(\cdot,y)=u,\quad  \text{ in }\; \Omega , \quad
		\partial_\nu y=0\;  \text{ on }\; \partial \Omega.
	\end{equation}
	By following estimates analogous to the ones leading to \eqref{eq:y0_C_estimate}, we have  in view of \eqref{eq:y0_C_estimate_01}--\eqref{eq:y0_C_estimate_02} and \eqref{g_N}, that any solution $y_{0}$ also satisfies $\|y_{0}\|_{C(\overline{\Omega})}<K$ when $\epsilon$ is sufficiently small. Since $\partial_y N=\mathcal{N}$ on $\Omega\times [-K,K]$ we conclude that $y_{0}$ is a solution of \eqref{eq:nonconvex_learn}.
\end{proof}
Concerning the announced differentiability of $\Phi_N(y):=  \int_{\Omega} {N}(x,y)\,dx$, define 
\[\Phi_N^\prime(y)h:=\int_{\Omega}\partial_y  {N}(x,y)h\,dx.\]
Since $\frac{\abs{\Phi_N(y+h)-\Phi_N(y)-\Phi_N^\prime(y)h}}{\norm{h}_{H^1(\Omega)}}
		=\frac{\abs{\Phi_N^\prime(y+\tau_h h)h-\Phi_N^\prime(y)h}}{\norm{h}_{H^1(\Omega)}}$ for some $\tau_h\in L^\infty(\Omega)$ with $\norm{\tau_h}_{L^\infty(\Omega)}\leq 1$, using the mean value theorem along with $H^1(\Omega)\subset L^{\frac{r}{r-1}}(\Omega)$, we have for a $C>0$
\begin{equation}\label{eq:differ_N}
	\begin{aligned}
		\frac{\abs{\Phi_N(y+h)-\Phi_N(y)-\Phi_N^\prime(y)h}}{\norm{h}_{H^1(\Omega)}}
		\leq C \norm{\partial_y  ({N} (\cdot,y+\tau_h h)-{N} (\cdot,y))}_{L^{r}(\Omega)}.
	\end{aligned}
\end{equation}
 Note that by definition, the growth rate of ${N}(x,\cdot)$ outside of $[-K-1,K+1]$ is exactly the same as the one of $F(x,\cdot)$.
Therefore $\partial_y  {N}(\cdot,y)$ is indeed an element of $L^{r}(\Omega)$. 
Finally, we need to verify that 
\begin{equation*}
	\begin{aligned}
		\lim_{h\to 0} \norm{\partial_y  {N}(y+\tau_h h)-\partial_y  N(y)}_{L^{r}(\Omega)}=0 \quad \text{ for } h\in H^1(\Omega).
	\end{aligned}
\end{equation*}
This is true due to the continuity of the Nemytskii operator $\partial_y  {N}:L^{\frac{r}{r-1}}(\Omega)\to L^{r}(\Omega)$.
\begin{remark}\label{rem:mean_value_thm}
Notice that in \eqref{eq:differ_N} the mean value theorem is applied for every $x\in\om$ and $\tau_{h}$ is defined as a selector function of the multi-valued map $\tau:\om\rightrightarrows [0,1]$ with 
\[\tau(x)=\{\lambda\in [0,1]:\, N(x,y(x)+h(x))-N(x,y(x))- \partial_y  N(x,y(x)+\lambda h(x))h(x)=0\}.\]
Even though by definition $\tau_{h}$ is  a bounded function, one still needs to show its measurability such that $\tau_{h}\in L^{\infty}(\om)$. Such a measurable selector function is indeed guaranteed by the Kuratowski--Ryll--Nardzewski selection theorem \cite[Theorem 18.13]{aliprantis} whose conditions can be verified in our case. In fact, we may choose $\tau_{h}(x):=\max \tau(x)$; see \cite[Theorem 18.19]{aliprantis}.
\end{remark}

Note that the above set up covers a wide range of problems, including the class of problems where the nonlinear function $f(\cdot,\cdot)$ is strictly monotone with respect to the second variable. In that case, the nonlinear PDE \eqref{eq:state_eq} admits a unique solution  \cite{Tro10}. We also point out that in the monotone case direct methods allow to prove the existence of solutions and energy bounds  for a wider array of monotone nonlinearities (such as, e.g., exponential functions). Moreover in that case, the regularity  and growth conditions on the nonlinear function $f$  can be relaxed. 
However, as pursuing such a generality is not the focus of the current paper, we skip detailed discussions here. We note however that structural aspects of the control problem such as first-order optimality, adjoints etc. remain intact even under relaxed conditions.

In order to give an example on this, we show in the next proposition how strict monotonicity  for the learning-based model can indeed be preserved.
\begin{proposition}\label{prop:N_approx_f}
	Let $f:\Omega\times \R\to \R$ satisfy Assumption \ref{assu:non_monotone} and  $\partial_y  f(x,y)\geq C_f$  for almost every $x\in\om$ and $y\in\R$ for some $C_f>0$. We additionally assume that $f\in C(\om\times \R)$.
	Then for every $\epsilon>0$, for every compact set $\Omega_c \subset \Omega$, and for every $M>0$, there exists a neural network $\mathcal{N}:=\mathcal{N}^{\epsilon}_{\Omega_c , M} \in C^{\infty}(\R^{d} \times \R)$ such that
	\begin{align}
		&|f(x,z)-\mathcal{N}(x,z)|<\epsilon, \quad \text{ for every } x\in \Omega_c \text{ and every } z\in [-M, M], \label{NN_f_1}\\\
		&\partial_{z}\mathcal{N}(x,z)\ge C_{\mathcal{N}},\quad \text{ for all } x\in \Omega \text{ and } z\in[-M, M] \text{  for some } C_{\mathcal{N}}>0.\label{NN_f_2}
	\end{align}
	If $f\in C^1(\Omega\times \R)$, then we have in addition that 
	\begin{align}
		&|\partial_{z} f(x,z)-\partial_{z}\mathcal{N}(x,z)|<\epsilon, \quad \text{ for all } x\in \Omega_c \text{ and } z\in [-M, M]. \label{NN_nablaf_1}
	\end{align}
\end{proposition}

\begin{proof}
	Let $\epsilon>0$, $\Omega_c\subset \Omega$ compact, and $M>0$. Further, let $\tilde{f}:\R^{d} \times \R\to \R$ be the extension by zero of $f$ outside $\Omega \times \R$, $\rho_{\delta}$ a standard mollifier \cite[Sec.2.2.2]{MR3288271}, and $\tilde{f}_{\delta}:=\tilde{f}\ast \rho_{\delta}: \mathbb{R}^{d}\times \mathbb{R}\to \mathbb{R}$. Next we choose $\delta>0$ such that the following hold true: (i) $\bar B(x,\delta):=\{{\hat{x}}\in\mathbb{R}^d:\|{\hat{x}}-x\|_2\leq\delta\}\subset\Omega$ for every $x\in \Omega_c$, (ii) $\tilde{f}_{\delta}(x,y)=f_{\delta}(x,y)$ for $(x,y)\in \Omega_c\times \R$, and (iii) $|f(x,y)-\tilde{f}_{\delta}(x,y)|<\epsilon/2$ for every $x\in \Omega_c$, $y\in [-M,M]$. Moreover, one finds that for sufficiently small $\delta>0$ it holds that $\partial_{y} \tilde{f}_{\delta}(x,y)\ge C_{\tilde{f}}$ for some $C_{\tilde{f}}>0$ for all $x\in\Omega$, $y\in \R$. Indeed, note that Assumption \ref{assu:non_monotone} and the mean value theorem yield for almost every $x'\in\Omega$, $y_{1}<y_{2}$
	\begin{equation}\label{bigger_linear}
		f(x',y_{2})-f(x',y_{1})\ge C_{f} (y_{2}-y_{1}).
	\end{equation}
	Hence, using $\rho_\delta(\cdot)=\delta^{-(d+1)}\rho(\cdot / \delta)$ \cite[Sec.2.2.2]{MR3288271}, we have 
	\begin{align*}
	&	\tilde{f}_{\delta}(x,y_{1})
		= \int_{B_{\delta}(x,y_{1})\cap (\Omega\times \R)} \tilde{f}(x',y') \delta^{-d-1}\rho\left (\frac{(x,y_{1})-(x',y')}{\delta} \right)d(x',y')\\
	& \le \int_{B_{\delta}(x,y_{2})\cap(\Omega\times \R)} \left(\tilde{f}(x',y') -C_{f}(y_{2}-y_{1})\right)  \delta^{-d-1}\rho\left (\frac{(x,y_{2})-(x',y')}{\delta} \right)d(x',y')\\
	&	= \tilde{f}_{\delta}(x,y_{2})- C_{f} \Big (\underbrace{\int_{B_{\delta}(x,y_{2})\cap (\Omega\times \R)} \delta^{-d-1}  \rho\left (\frac{(x,y_{2})-(x',y')}{\delta} \right)}_{=:\tilde{C}} d(x',y')\Big)(y_{2}-y_{1})\\
	&	= \tilde{f}_{\delta}(x,y_{2})- C_{f}\tilde{C}(y_{2}-y_{1}).
	\end{align*}
	We now use the fact that the boundary of $\Omega$ is Lipschitz to deduce that for some small enough $\delta>0$ we have $\tilde{C}:=\tilde{C}_{x,y}>c$ for some $c>0$, for every $x\in\Omega$, $y\in \R$, and set $C_{\tilde{f}}:=C_{f}c$.  Hence from the last inequality above we deduce $\partial_{y} \tilde{f}_{\delta}(x,y)\ge C_{\tilde{f}}$. Utilizing now Theorems \ref{thm:function_app} and \ref{thm:deriv_app} for the compact set $\overline{\Omega}\times [-M,M]\subset \R^{d}\times \R$, we find a neural network $\mathcal{N}\in C^{\infty}(\R^{d}\times \R)$ such that $|\tilde{f}_{\delta}(x,y)-\mathcal{N}(x,y)|<\epsilon/2$ as well as $|\partial_{y} \tilde{f}_{\delta}(x,y)-\partial_{y} \mathcal{N}(x,y)|<C_{\tilde{f}}/4$  for every $x\in \overline{\Omega}$ and $y\in [-M,M]$. Then with the use of the triangle inequality we get \eqref{NN_f_1} and \eqref{NN_f_2} for $C_{\mathcal{N}}=\frac{3}{4}C_{\tilde{f}}$.
	
	Finally, when $f$ is also continuously differentiable in $\Omega \times \R$, we can proceed as before with the extra care to choose $\delta>0$ such that $|\partial_{y}f(x,y)-\partial_{y}\tilde{f}_{\delta}(x,y)|<\epsilon/2$ for every $x\in \Omega_c$, $y\in [-M,M]$.
\end{proof}

Note that if $f$ is bounded on $\om\times [-K,K]$, for instance  if $f\in C(\overline{\om}\times \R)$ as in Proposition \ref{prop:first_aprox}, then the estimate \eqref{g_N} holds here as well and if analogous conditions hold for the derivative of $f$ then with the help of \eqref{NN_nablaf_1} we also have
\begin{equation}\label{sup_L2_app_der}
	\sup_{\|y\|_{L^{\infty}(\Omega)}<K} \|\partial_{y}f(\cdot,y)-\partial_{y}\mathcal{N}(\cdot,y)\|_{U}\le \epsilon.
\end{equation}

\subsection{Error analysis for the control-to-state map}
Our next target is to show the error bounds \eqref{eq:operator_err} and \eqref{eq:operator_deriv_err} for the solution maps (control-to-state maps) of the learning-informed versus the original PDE. 
Before we proceed, we first show the local Lipschitz conditions \eqref{eq:Q_Lip} and \eqref{eq:second_Lip}. 
For the ease of presentation we confine ourselves to a
monotone $f(x,\cdot)$ here. For the nonmonotone $f(x,\cdot)$, we would require \eqref{eq:smallness} to be satisfied for solutions uniformly bounded by $K$. 
Consider the following pairs of equations for $i\in\{1,2\}$
\begin{equation}\label{eq:pair3}
	\left\{
	\begin{aligned}
		-\Delta y_i+ f(\cdot,y_i) &=u_i\;\text{ in }\;\Omega,\;\; \\
		\partial_{\nu} y_i &=0\;\;\text{ on }\;\partial\Omega,\;\;  
	\end{aligned} \right.\quad 
	\text{ and } \quad
	\left\{
	\begin{aligned}
		-\Delta p_i+\partial_y  f(x,\bar{y}_{i}) p_i &=v\;\text{ in }\;\Omega,\;\; \\
		\partial_{\nu} p_i &=0\;\text{ on }\;\partial\Omega,\;\;  
	\end{aligned} \right.\quad 
\end{equation}
where $v\in U$ is  unitary, $\overline{y}_{i}=\Pi (u_{i})$,  and $p_i=\Pi^\prime(u_i) v$ for $i=1,2$.
Taking the difference of the first equations in  \eqref{eq:pair3} for $i=1,2$, testing with $y_1-y_2$, and using the mean value theorem
we get for some $C_f>0$ that
\[ \begin{aligned}
C_{f} \norm{y_1-y_2}_{H}^2
&\leq \|\nabla y_{1}-\nabla y_{2}\|_{L^{2}(\Omega)}^{2}+ \int_{\Omega} (f(x,y_{1})-f(x,y_{2}))(y_{1}-y_{2})\,dx\\&= \int_{\Omega} (u_{1}-u_{2})(y_{1}-y_{2})\,dx
\leq   \norm{u_1-u_2}_U \norm{y_1-y_2}_H,
\end{aligned}
\]
which yields the Lipschitz property $\norm{y_1-y_2}_{H}\leq \frac{1}{C_f} \norm{u_1-u_1}_U  $.

In order to show the local Lipschitz continuity of $\Pi'$, we need to further assume condition \eqref{lip_partialy_f}.
Consider now the difference of the right-hand side equations for $i=1,2$ in  \eqref{eq:pair3}. Using standard PDE arguments (see, e.g., \cite[Theorem 4.7]{Tro10}) we find
\[ \begin{aligned}
\norm{p_1-p_2}_{H^1(\Omega)} &+\norm{p_1-p_2}_{C(\bar{\Omega})} 
\leq C\norm{(\partial_y  f(\cdot,\bar{y}_1) - \partial_y  f(\cdot,\bar{y}_2))p_1}_{L^2(\Omega)}\\
&\leq CL\norm{p_1}_{C(\overline{\Omega})}\norm{\bar{y}_1-\bar{y}_2}_{L^2(\Omega)}
\leq C\frac{L}{C_{f}} c\norm{v}_{L^2(\Omega)}\norm{u_1-u_2}_{L^2(\Omega)}.  
\end{aligned}
\]
Here, we also used the estimate $\|p_{1}\|_{C(\overline{\Omega})}\le c \|v\|_{L^{2}(\Omega)}$ from Theorem \ref{thm:continuity_multi_map}.

For the desired error bounds we focus now on the state equations
\begin{equation}\label{eq:pair1}
	\left\{
	\begin{aligned}
		-\Delta y +\mathcal{N}(x,y)&=u\;\text{ in }\;\om,\;\;  \\
		\partial_{\nu}y &=0\;\text{ on }\; \partial \om,\;\;  
	\end{aligned} \right.\quad 
	\text{ and } \quad
	\left\{
	\begin{aligned}
		-\Delta y +f(x,y)&=u  \;\text{ in }\;\om,\\
		\partial_{\nu}y&=0 \; \text{ on }\; \partial \om,
	\end{aligned}\right.
\end{equation}
and the associated adjoints
\begin{equation}\label{eq:pair2}
	\left\{
	\begin{aligned}
		-\Delta p +\partial_{y}\mathcal{N}(x,\bar{y})p&=v\;\text{ in }\;\om,\;\; \\
		\partial_{\nu} p &=0\;\text{ on }\; \partial \om,\;\;  
	\end{aligned} \right.\quad 
	\text{ and } \quad
	\left\{
	\begin{aligned}
		-\Delta p + \partial_y  f(x,\bar{y}) p &= v\;\;  \text{ in }\;\om,\\
		\partial_{\nu}p &=0\;\;  \text{ on }\; \partial \om.
	\end{aligned}\right.
\end{equation}

The main approximation result is stated below. It guarantees that the uniform approximation properties of the control-to-state operator $\Pi$ and its derivative (compare   \eqref{eq:operator_err} and \eqref{eq:operator_deriv_err} of Theorem \ref{thm:convergence} and Assumption \ref{assum:operator_derivative}, respectively) are met by the corresponding learning-informed operators.
\begin{proposition}\label{prop:state_error}
	Let $\epsilon>0$ and $M> K >0$, with $K$ being the constant from \eqref{eq:y0_C_estimate}.   Suppose the first inequality in \eqref{eq:smallness} holds for $f$ for every  $y$ such that $\|y\|_{L^{\infty}(\om)}\le K$.
	Assume 
	that $\mathcal{N}\in C^{\infty}(\R^{d} \times \R)$ satisfies the approximation property
	\begin{equation}\label{eq:f_app_error}
		\sup_{\|y\|_{L^{\infty}(\Omega)}<M}\norm{f(\cdot,y)-\mathcal{N}(\cdot,y)}_U \leq \epsilon,
	\end{equation}
	for $\epsilon>0$ sufficiently small.
	Then,  the following error estimate holds :
	\begin{equation}\label{eq:ope_est}
		\norm{y_{0}-y_{\epsilon}}_{H}\leq C \epsilon,  \quad \text{ for all  }\; u \in \mathcal{C}_{ad},
	\end{equation}
	where the constant $C>0$ depends only on $f$, and $y_{\epsilon}$, $y_{0}$ are solutions of the left and right equations of \eqref{eq:pair1} respectively.
	Moreover, assuming  \eqref{lip_partialy_f} and also that the condition
	\begin{equation}\label{eq:der_f_app_error}
		\sup_{\|y\|_{L^{\infty}(\Omega)}<M} \|\partial_{y}f(\cdot,y)-\partial_{y}\mathcal{N}(\cdot,y)\|_{U}\le \epsilon_{1},
	\end{equation}
	holds for sufficiently small $\epsilon_{1}>0$, then, there exist some constants $C_0>0$ and $C_1>0$ so that
	\begin{equation}\label{eq:ope_der_est}
		\norm{p_{0}-p_{\epsilon}}_{H^1(\Omega)\cap C(\overline{\Omega})}
		\leq C_1 \epsilon_1 +C_0\epsilon,  \quad \text{ for all  }\; u \in \mathcal{C}_{ad},
	\end{equation}
	where $p_{\epsilon}$, $p_{0}$ are solutions of the left and right equations of \eqref{eq:pair2} respectively.
\end{proposition}
\begin{proof}
	Let $y_\epsilon$ and $y_0$ be  solutions of the learning-informed PDE and the original PDE, respectively.  Recall that the $H^1$ norms of both $y_\epsilon$ and $y_0$ are bounded by $K>0$. Subtracting the two PDEs we get
	\begin{equation}\label{eq:diff_PDE}
		-\Delta (y_0-y_\epsilon) =\mathcal{N}(\cdot,y_\epsilon)-f(\cdot,y_0)\;\text{ in }\;\Omega \quad  \text{ and }  \quad \partial_\nu (y_0-y_\epsilon)=0\;\text{ on }\;\partial\Omega. 
	\end{equation}
	Using the same technique as in the proof of Theorem \ref{thm:continuity_multi_map}, the equation in \eqref{eq:diff_PDE} can be rewritten as
	\begin{equation}\label{eq:diff_PDE2}
	\left(	-\Delta +\kappa_{0} +(\partial_y  f(\cdot,\zeta_\epsilon))^+\right)(y_0 -y_\epsilon)   =\mathcal{N}(\cdot,y_\epsilon)-f(\cdot,y_\epsilon) + (\kappa_{0} -(\partial_y  f(\cdot,\zeta_\epsilon))^-)(y_0 -y_\epsilon)   ,
	\end{equation}	
	where $\zeta_\epsilon$ is a pointwise convex combination of $y_0$ and $y_\epsilon$ that results from a pointwise application of the mean value theorem, and $\kappa_{0}>0$ is a fixed small constant.   
	We have then the estimate
	\[
	\begin{aligned}
	&\frac{\kappa_{0}}{C_h}\norm{y_0-y_\epsilon}_{H^1(\Omega)} +\norm{y_0 -y_\epsilon}_{C(\overline{\Omega})}\\
	\leq&     (\kappa_{0}+C_l) (\norm{ \mathcal{N}(\cdot,y_\epsilon) -f(\cdot,y_\epsilon)}_{L^2(\Omega)}
	+  \norm{ (\kappa_{0}-(\partial_y  f(\cdot,\zeta_\epsilon))^{-})(y_0 -y_\epsilon)}_{L^2(\Omega)}),
	\end{aligned}
	\]
	Rearranging the above inequality, and taking into account the Lipschitz continuity of $\partial_y  f$ and the condition \eqref{eq:smallness} for $\zeta_{\epsilon}$ for which it holds $\|\zeta_{\epsilon}\|_{L^{\infty}(\om)}\le K$, for sufficiently small $\epsilon$ we derive finally  
	\[\norm{y_0 -y_\epsilon}_H  \leq C \epsilon. \]
	For deriving \eqref{eq:ope_der_est} we use a similar approach. Let $p_\epsilon$ and $p_0$ be the solutions of the left and right equations in \eqref{eq:pair2}, respectively. Subtracting these two equations  gives
	\begin{equation}\label{eq:error_pde2}
		\begin{aligned}
			-\Delta (p_\epsilon -p_0) +\partial_y  f(x,y_0)(p_\epsilon-p_0)&=(\partial_y  f(x,y_0) - \partial_y  \mathcal{N}(x,y_\epsilon))p_\epsilon\quad \text{ in }\;\Omega,\\
			\partial_\nu (p_\epsilon-p_0)&=0\quad \text{ on }\;\partial\Omega.
		\end{aligned}
	\end{equation}
	Using again the same trick as above, we rewrite \eqref{eq:error_pde2} as 
	\begin{equation}\label{eq:error_pde3}
		\begin{aligned}
			&	-\Delta (p_\epsilon -p_0) +(\kappa_{1}+(\partial_y  f(x,y_0))^+ (p_\epsilon-p_0)\\
			=&(\partial_y  f(x,y_0) - \partial_y  \mathcal{N}(x,y_\epsilon))p_\epsilon +(\kappa_{1}-(\partial_y  f(x,y_0))^-)(p_\epsilon-p_0),
		\end{aligned}
	\end{equation}
	and then similarly we get
	\begin{equation}\label{est_p_minus_peps}
		\norm{p_\epsilon-p_0}_{H^1(\Omega)}  \leq C \norm{p_\epsilon}_{C(\bar{\Omega})}\norm{\partial_y  f(\cdot,y_0) - \partial_y  \mathcal{N}(\cdot,y_\epsilon)}_{L^2(\Omega)},
	\end{equation}
	for some constant $C$ independent of both $p_0$ and $p_\epsilon$, but depending on the constants $C_h$ and $C_l$. The estimate in \eqref{est_p_minus_peps} holds also for $\norm{p_\epsilon-p_0}_{C(\overline{\Omega})}$ but with a different constant, say $\tilde{C}>0$.
	Focusing on the right-hand side of the  inequality above and using the triangle inequality we have
	\[\begin{aligned}
	\norm{\partial_y  f(\cdot,y_0) - \partial_y  \mathcal{N}(\cdot,y_\epsilon)}_{L^2(\Omega)}
	 &\leq
	\norm{\partial_y  f(\cdot,y_0) - \partial_y  f(\cdot,y_\epsilon)}_{L^2(\Omega)} \\
	+ &	\norm{\partial_y  f(\cdot,y_\epsilon) - \partial_y  \mathcal{N}(\cdot,y_\epsilon)}_{L^2(\Omega)}
	\leq L\norm{y_0-y_\epsilon}_{L^2(\Omega)}+\epsilon_1,
	\end{aligned}
	\]
	where $L$ is the local Lipschitz constant of $\partial_{y}f(\cdot,\cdot)$ for those $y\in H^1(\Omega)\cap C(\overline{\Omega})$ with $\norm{y}_{L^\infty(\Omega)}\leq K$.
	
	Finally we need to estimate $\|p_\epsilon\|_{C(\overline{\Omega})}$ in \eqref{est_p_minus_peps}. For this we note that for sufficiently small $\epsilon_{1}$, the second bound in \eqref{eq:adjoint_bounds} also holds for the solution of PDEs with $\mathcal{N}$. This yields the estimate
	\begin{equation}\label{est_p}
		\|p_\epsilon\|_{C(\overline{\Omega})}\le C_c \|v\|_{{L^{2}(\Omega)}},
	\end{equation}
	with the constant $C_c$ independent of $v$ and $\epsilon$. Finally we conclude
	\begin{align*}
		\|p_{0}-p_{\epsilon}\|_{H^1(\Omega)\cap C(\overline{\Omega})}
		&=\sup_{\|v\|_{{L^{2}(\Omega)}}\le 1} \|p_{0}-p_{\epsilon}\|_{H^{1}(\Omega)\cap C(\overline{\Omega})}\\
		&=\sup_{\|v\|_{{L^{2}(\Omega)}}\le 1} \norm{p_0-p_\epsilon}_{H^1(\Omega)} +\norm{p_0-p_\epsilon}_{C(\overline{\Omega})}\\
		&\le   C_c (C+\tilde{C})(L\epsilon +\epsilon_{1}) 
		\le C_{1}\epsilon_{1} +C_{0}\epsilon,
	\end{align*}
	which ends the proof.
\end{proof}

\begin{remark}\label{rmk:multivalue_map}
	Notice that the condition \eqref{eq:smallness} imposed to all $y$ with $\norm{y}_{L^\infty(\Omega)}\leq K$ in fact enforces a unique solution to the semilinear PDE \eqref{eq:state_eq}, which also satisfies the same constraint. It is possible to treat the multi-solution case using  a similar strategy as Theorem \ref{thm:continuity_multi_map}, by using $\Gamma$--convergence arguments to show the convergence of $y_\epsilon \to y$ in a certain sense, and then apply the condition \eqref{eq:smallness} to $y_0$.
\end{remark}
\begin{remark}\label{rmk:zeroNeumann}
	The results above also hold for more general types of boundary conditions, including homogeneous Dirichlet boundary conditions.
\end{remark}
\subsection{Existence of solutions of the learning-informed optimal control}
After having replaced the unknown $f$ by the neural network based approximation $\mathcal{N}$ we are now interested in the following optimal control problem with a partially learning-informed state equation:
\begin{align}\label{eq:cost_NN}
	&\text{minimize}\quad J(y,u):= \frac{1}{2}\|y-g\|^2_{L^2(\Omega)} +\frac{\alpha}{2} \|u\|_{L^2(\Omega)}^2,\quad\text{over }(y,u)\in H^1(\Omega)\times L^2(\Omega),\\
	&\text{s.t. } \quad
	\label{eq:state_eq_NN}
	-\Delta y  +\mathcal{N}(x,y)=u\quad \text{ in }\;\om,\quad
	\partial_{\nu}y=0\;\;  \text{on }\; \partial \om,\\
	&\phantom{\text{s.t. }}\quad\; u\in\mathcal{C}_{ad}.\label{eq:control_constr_NN}
\end{align}

In what follows we prove the existence of an optimal control for the problem \eqref{eq:cost_NN}--\eqref{eq:control_constr_NN}. Here we consider that the control-to-state operator is single-valued, that is, the learning-informed PDE \eqref{eq:state_eq_NN} has a unique solution for every $u\in\mathcal{C}_{ad}$.
According to Proposition \ref{pro:existence_wsc},  we only need to check that the operator $Q_\mathcal{N}:U \to H$ is weakly sequentially closed. In fact, an even stronger property holds true as we show next.  

\begin{proposition}\label{prop:existence_opt_con}
	Let  $\mathcal{N}\in C^{\infty}(\R^{d}\times \R)$ be a neural network such that any solution of the learning-informed PDE \eqref{eq:state_eq_NN} satisfies a bound as in \eqref{eq:y0_C_estimate}. Then the reduced operator $Q_\mathcal{N}:U=L^2(\Omega)\supset C_{ad}\to  H=L^2(\Omega)$ induced from the control-to-state map of \eqref{eq:state_eq_NN} is weakly-strongly continuous, in the sense that if $u_{n} \rightharpoonup u$ in $U$ and $y_{n}\in \Pi_\mathcal{N}(u_n)$ then, $y_{n}\to y$ in $H$ for some $y\in \Pi(u)$.
\end{proposition}
\begin{proof}
	Let  $u_{n} \rightharpoonup u$ in $U$ and $y_n\in \Pi_\mathcal{N}(u_n)$. Then $(y_n)_{n\in\mathbb{N}} $ is a bounded sequence in $Y=H^1(\Omega)\cap C(\bar\Omega)$ as $(u_n)_{n\in\mathbb{N}} \subset U$ is a bounded set in $L^{\infty}(\om)$. Thus, up to a subsequence, still denoted by $(y_n)$, there is $\bar{y}\in {H^{1}(\om)}$ such that $y_n \rightharpoonup \bar{y}$ in ${H^{1}(\om)}$.
	Since ${H^{1}(\om)}$ embeds compactly into $H$, we can consider that $y_n \to \bar{y}$ strongly in $H$.
	We show that $\bar{y}=\Pi_\mathcal{N}(\bar{u})$, i.e., $\bar{y}$ is a weak solution of the PDE in \eqref{eq:state_eq_NN}. Since $y_n$ is the weak solution of  \eqref{eq:state_eq_NN} with right hand-side $u_{n}$, we have
	\begin{equation}\label{eq:PDE_sequence}
		\int_\Omega \nabla y_n\cdot \nabla v\,dx+\int_\Omega \mathcal{N}(x, y_n) v\,dx=\int_\Omega u_n v\,dx\quad\text{for all }v\in H^1(\Omega).
	\end{equation}
	We only need to show that 
	\begin{equation}\label{eq:nonlinear_error}
		\int_\Omega\left( \mathcal{N}(x, y_n) - \mathcal{N}(x, \bar{y})\right) v\,dx=0,
	\end{equation}
	since the convergence of the other two terms readily follows from weak convergence.
	Taking into account that $\mathcal{N}\in C^{1}(\R^{d}\times \R)$ 
	we have that for every $M>0$, there exists an $L_{M}>0$ such that for every $x\in\Omega$ and $y_{1},y_{2}\in [-M,M]$, we have
	\begin{equation}\label{N_Lip}
		|\mathcal{N}(x,y_{1})-\mathcal{N}(x,y_{2})|\le L_{M} |y_{1}-y_{2}|.
	\end{equation}
	Using the estimate \eqref{eq:y0_C_estimate}, we have that $(y_{n})_{n\in\mathbb{N}}$ and, hence, $\bar{y}$ are uniformly bounded in $L^{\infty}(\Omega)$, say by a constant $M>0$. Thus we have 
	\begin{align*}
		\| \mathcal{N}(\cdot, y_n) - \mathcal{N}(\cdot, \bar{y})\|_{U}
		\le 
		L_{M}\|y_{n}-\bar{y}\|_{H}.
	\end{align*}
	Due to the inequality above and the strong convergence of $y_n\to \bar{y}$ in $H$,  \eqref{eq:nonlinear_error} is verified.
	Passing to the limit $n\to \infty$ in \eqref{eq:PDE_sequence} we get that $\bar{y}$ is a weak solution of \eqref{eq:state_eq_NN}  corresponding to $\bar{u}$. Since any other subsequence of $(y_{n})_{n\in\mathbb{N}}$ will have a further subsequence that converges to $\Pi_{\mathcal{N}}(\bar{u})$ the assertion follows.
\end{proof}

For the error analysis on the optimal controls of \eqref{eq:cost_NN} with \eqref{eq:state_eq_NN} to solutions from \eqref{eq:cost} with \eqref{eq:state_eq}, we can readily apply Theorems \ref{thm:convergence}, \ref{thm:error_bound} and \ref{thm:error_bound2} for the monotone function $f$, in view of the error bounds shown in Proposition \ref{prop:state_error}.
For the nonmonotone case, these results are still applicable up to a selection of subsequences of the solutions.

Finally, we would like to make a remark regarding the approximation of $f:\Omega \times \R \to \R$ in a semilinear PDE, given a set of input-output data.
The input data is a family of sampled points from $\Omega \times [y_{min}, y_{max} ] $, denoted by  $(x_i,y(x_i))_{i\in I}$, and the outputs are the corresponding values $(f(x_i,y(x_i)))_{i\in I}$, which are computed from \eqref{eq:state_eq} via 
\[f(x_i,y(x_i))=u(x_i)+\Delta y(x_i).\]
In real world applications, we assume that we have access to the data points $y(x_{i})$ and thus also to $\Delta y(x_{i})$, while $u$ is a control which is at our disposal to be tuned.
In order to be consistent with the functional analytic setting, one  needs to give pointwise meaning to $\Delta y$, which in general is an object in $H^{-1}(\Omega)$, only. This can be achieved by choosing controls $u\in\mathcal{C}_{ad}$ of sufficient regularity.
Indeed, since both $f$ and $y$ are continuous functions when choosing continuous $u$, equation \eqref{eq:state_eq} implies that $\Delta y$ is continuous, too, and hence admits a pointwise evaluation.

\subsection{Numerical algorithm for the optimal control problems}\label{subsec:Newton}

In this section we briefly describe an algorithm for solving the optimal  control problem \eqref{eq:cost}. Even though it is suitable for rather general problems, we outline it here for the version with the learning-informed state equation.

In order to compute a numerical solution, we first state the Karush-Kuhn-Tucker (KKT) conditions, which are justified by constraint regularity (see \cite{ZowKur79} for a general setting):
\begin{equation}\label{eq:stationary1}
	\begin{aligned}
		- \Delta y +\mathcal{N}(\cdot,y)  -u&=0\;   \text{ in } \Omega ,\quad
		\partial_\nu y=0\;   \text{ on } \partial \Omega ,\\
		-  \Delta p+   \partial_y \mathcal{N}(\cdot,y) p +y&= g\;   \text{ in } \Omega ,\quad
		\partial_\nu p=0\;   \text{ on } \partial \Omega,\\
		-p+\lambda + \alpha u&=0\;  \text{ in } \Omega ,\\
		\lambda - \max (0,\lambda+ c(u-\overline{u})) -\min(0,\lambda + c(u-\underline{u}))&=0\;   \text{ in } \Omega ,
	\end{aligned}  
\end{equation}
where $c>0$ is some constant, which in practice, is useful to be chosen $c=\alpha$. 
The first  equation with its boundary condition is just the learning-informed PDE constraint, while  the next one is the associated adjoint equation. The third equation represents optimality w.r.t. $u$ and, together with the last one, it  incorporates the control constraint $\underline{u} \leq u \leq \overline{u}$. Indeed, notice that the last equation is equivalent to the usual complementarity system as it secures a.e. that
\begin{equation*}
	\lambda=0:  \: \underline{u} < u < \overline{u},\quad
	\lambda \geq 0:  \:  u=\underline{u},\quad
	\lambda \leq 0:  \:  u=\overline{u}.
\end{equation*}
Letting $\phi:=(y,u,p,\lambda)^\top$, \eqref{eq:stationary1} can be compactly rewritten  as the nonsmooth equation
\begin{equation}\label{eq:stationary2}
	M_\mathcal{N}(\phi) -(0,g,0,0)^\top=0.
\end{equation}
For solving \eqref{eq:stationary2}, we employ a semi-smooth Newton method (SSN); see, e.g., \cite{HinItoKun02}. It operates as follows: Given an initial guess $\phi_0$ of a solution to \eqref{eq:stationary2}, compute for all $k=0,1,2,\ldots$
\[ \begin{aligned}
\phi_{k+1}& =\phi_k - (\mathcal{G}_\mathcal{N}(\phi_k) )^{-1}(M_\mathcal{N}(\phi_k)  -(0,g,0,0)^\top).
\end{aligned}
\]
Here, $\mathcal{G}_\mathcal{N}(\phi_k)$ is a Newton derivative of the operator $ M_\mathcal{N}$ at $\phi_k$ given by
\[\mathcal{G}_\mathcal{N}(\phi_k)= 
\left(
\begin{array}{cccc}
- \Delta   +\partial_y \mathcal{N} (\cdot ,y_k) & 0&- \text{ Id}  & 0\\
\partial_{yy} \mathcal{N}(\cdot ,y_k) p_k + \text{ Id} & -  \Delta +   \partial_y \mathcal{N}(\cdot ,y_k) & 0 & 0\\
0 & - \text{ Id} & \alpha \text{ Id} &  \text{ Id}\\
0& 0 & - cG_k &   \text{ Id}-G_k
\end{array}  \right),
\]
where for $x\in\Omega$,
\[
G_k(x):=\left\{  
\begin{aligned}
1, & \quad \text{if }c(\underline{u}(x)-u_k(x))\leq \lambda_k(x) \leq c(\overline{u}(x)-u_k(x)),\\
0, & \quad \text{else},
\end{aligned} \right. 
\]
is a Newton derivative that corresponds to the nonsmooth functions $\max(0,\cdot)$ and $\min(0,\cdot)$ in \eqref{eq:stationary1}. SSN can be shown to converge locally at  a superlinear rate, provided $\phi_0$ is sufficiently close to a solution and the selection of Newton derivatives for $M_\mathcal{N}$ is uniformly bounded and invertible along the iteration sequence; see \cite{HinItoKun02} and \cite{HinUlb04}.
Moreover, under a nondegeneracy assumption the method exhibits a mesh independent convergence upon proper discretization of \eqref{eq:stationary2}; see \cite{Hin07, HinUlb04}. Globalization of the SSN iteration can be achieved, e.g., by employing a path search \cite{DirFer95, Ral94}, which we did not pursue here, however. Rather we intertwined SSN with a sequential quadratic programming (SQP) iteration, with the latter specified below. This combination helped the globally convergent SQP solver to escape from unfavorable local minimizers or stationary points. Obviously, one cannot expect a general theoretical result supporting such a behavior. It, hence, merely reflects a useful numerical observation, in particular in connection with our example with a nonmonotone $f$.

\paragraph{SQP algorithm}
Here we consider the reduced SQP approach which operates on the reduced optimal control problem. Given an estimate $u_k$ of an optimal control, in every iteration it seeks to solve the following quadratic problem:
\begin{equation}\label{eq:SQP}
	\begin{aligned}
		&\text{minimize}\quad \; 
		\langle \mathcal{J}_{\mathcal{N}}^\prime(u_k) + \frac{1}{2} H_k(u_k)\delta_u,\delta_u\rangle_{U^*,U}, \quad\text{over }\delta_u\in U,\\
		&\text{subject to }  \;  \underline{u} \leq u_k+\delta_u \leq  \overline{u}\quad \text{a.e. in }\Omega,
	\end{aligned}
\end{equation}
where $\mathcal{J}_{\mathcal{N}}^\prime(u_k)$ is the Fr\'echet derivative of the reduced functional $\mathcal{J}_{\mathcal{N}}$, and $H_k(u_k)$ is a positive definite approximation of the second-order derivative of $\mathcal{J}_{\mathcal{N}}$ at $u_k$. 
First-order optimality for \eqref{eq:SQP} yields
\begin{equation}\label{eq:SQP_optimal}
	\begin{aligned}
		& \mathcal{J}_{\mathcal{N}}^\prime(u_k) +H_k(u_k)\delta_u +\lambda=0 , \\
		&  \lambda - \max (0,\lambda+ c(u_k+\delta_u -\overline{u})) -\min(0,\lambda + c(u_k +\delta_u -\underline{u}))=0,
	\end{aligned}
\end{equation}
for some fixed $c>0$. This nonsmooth system can be again solved using a semi-smooth Newton method which yields $\delta_{u,k}$ and $\lambda_k$.
Concerning the Hessian approximation, in our implementation we choose $H_k(u_k):=(\mathcal{J}_{\mathcal{N}}^\prime(u_k))^*\mathcal{J}_{\mathcal{N}}^\prime(u_k)$, where '$^*$' denotes the adjoint operator. 

For globalization we use a classical line search with the merit function
\begin{equation}\label{eq:merit}
	\Phi_k(\mu)=\mathcal{J}_{\mathcal{N}}(u_k+\mu\delta_{u,k}) + \beta_k \Psi_k(\mu) \quad \text{  for some  } \beta_k >0,
\end{equation}
where 
\[\Psi_k(\mu):= \norm{(u_k+\mu\delta_{u,k}-\overline{u})^+}_{L^2(\Omega)}+\norm{(u_k+\mu\delta_{u,k}-\underline{u})^-}_{L^2(\Omega)},\]
with $ a^+:=\max\set{a,0}, \text{ and }\; a^-:=\min \set{0,a}$.
We employ a backtracking line search method starting with $\mu:=1$ to decide on the step length. Note that the reduced problem requires to enforce the PDE constraint for every $u_k$. For this purpose a (smooth) Newton iteration was embedded into every SQP update step.
This Newton iteration is terminated when $\|-\Delta_h y_k+\mathcal{N}(\cdot,y_k)-u_k  \|_{H^{-1}(\Omega)}\leq\text{tol}=10^{-16}$ or a maximum of 15 iterations was reached.

To summarize, we utilize the following overall algorithm:
\renewcommand{\thealgorithm}{\arabic{algorithm}}
\setcounter{algorithm}{0}
\begin{algorithm}	
	\begin{itemize}
		\item[$\bullet$] {Initialization:} Choose $\phi_0:=(y_0,\; u_0\;, p_0,\; \lambda_0)$, and compute $\Phi_0(0)$. Fix a lower bound $\epsilon>0$ for the step length, choose $\rho\in (0,1)$, and $\beta_0>0$. Set $k:=0$.
		\item[$\bullet$] {Unless the stopping criteria are satisfied, iterate:} 
		\begin{itemize}
			\item[(1)] Compute an update direction $\delta_{u,k}$ by solving \eqref{eq:SQP_optimal} using SSN. Let $\mu_k^0:=1$, $y_k^{-1}:=y_k$ and set $l:=0$. Iterate:			
			\begin{itemize}
				\item[(a1)]  Compute $y_k^l:=\Pi_{\mathcal{N}}(u_k+\mu_k^l \delta_{u,k})$, where $\Pi_{\mathcal{N}}$ is realized by performing Newton iterations as a nonlinear PDE solver initialized by $y_k^{l-1}$.\\ Setting $y:=y_k^l$ and $u:=u_k+\mu_k^l \delta_{u,k}$ compute the remaining quantities in $\phi_k^l$ according to \eqref{eq:stationary1} with $p=:p_k^l$ and $\lambda=:\lambda_k^l$. This yields $\phi_k^l$.			
				\item[(a2)] Increase $\beta_k$, if necessary, to get $\beta_k^l$. 
				\item[(a3)] Check the Armijo condition \eqref{eq:updating_condition}.\\  If it is satisfied, then set $l_k:=l$ and continue with step $(2)$; otherwise  update $\mu_k^{l+1}:=r \mu_k^l$, $l:=l+1$. \\ If $\mu_k^{l+1}<\epsilon$, then terminate the algorithm;  otherwise return to Step (a1).
			\end{itemize}
			\item[(2)] Set $\phi_{k+1}:={\phi}_k^{l_k}$, and $\beta_{k+1}:=\beta_k^{l_k}$, and $k:=k+1$.
		\end{itemize} 
		\item[$\bullet$] {Output:} The value of $\phi_k$ which contains both the control and state variables.
	\end{itemize}
	\caption{A semi-smooth Newton SQP algorithm for  PDE control problems}	\label{alg:SQP}
\end{algorithm}

In our examples, we choose $\mu_0=1$,  $\epsilon=10^{-5}$, $r=2/3$, and  $\beta_0=\norm{\lambda_0}_{L^2(\Omega)}+1$. 
In order to solve the nonsmooth system in \eqref{eq:SQP_optimal}, we employ a primal-dual active set strategy (pdAS), which was shown to be equivalent to an efficient SSN solver for classes of constrained optimization problems \cite{HinItoKun02}.  For the precise set-up of pdAS and the associated active/inactive set estimation we also refer to \cite{HinItoKun02}. For minimizing quadratic objectives subject to box constraints and utilizing highly accurate linear system solvers, pdAS is typically terminated when two consecutive active and inactive set estimates coincide. We recall here that the active set for \eqref{eq:SQP} at the solution $\delta_{u,k}$ is a subset $\mathcal{A}_k$ of $\Omega$ with $(u_k+\delta_{u,k})(x)\in[\underline{u}(x),\overline{u}(x)]$ for $x\in\mathcal{A}_k$; $\mathcal{I}_k:=\Omega\setminus\mathcal{A}_k$ denotes the associated inactive set.  Alternatively one may stop the iteration once the residual norm of the nonsmooth system at an iterate drops below a user specified tolerance.
In view of \eqref{eq:SQP_optimal} and constraint satisfaction, the function $\Psi_k(\mu)$ in \eqref{eq:merit} appears irrelevant as a penalty for violations of the box constraints. However, it becomes relevant when early stopping is employed in SSN (respectively pdAS). 
In this case we still need to guarantee that $\delta_{u,k}$ is a descent direction for our merit function to obtain sufficient decrease of $\Phi_k$ in our 
line search \eqref{eq:updating_condition}. This is needed for getting convergence of $(u_k)$ (along a subsequence) to a stationary point. For deriving a proper stopping rule for SSN to guarantee sufficient decrease, we multiply the first equation in \eqref{eq:SQP_optimal}
by the solution $\delta_{u}$, use $\lambda(u_k+\delta_{u}-\overline{u})(u_k+\delta_{u}-\underline{u})=0$ a.e. in $\Omega$ and  the feasibility of $u_k+\delta_{u}$, both according to the second line in \eqref{eq:SQP_optimal}.  We further set $\beta_k>\|\lambda\|_{U}$ (upon identifying $U^*\widehat{=}U$) to find
\[
\langle \mathcal{J}_{\mathcal{N}}^\prime(u_k), \delta_{u} \rangle_{U^*,U}+ \beta_k ( \underbrace{\Psi_k(1)}_{=0} - \Psi_k(0)) \leq -\langle H_k(u_k)\delta_{u},\delta_{u}\rangle_{U^*,U}<0,
\]
unless $\delta_{u}=0$, i.e., $u_k$ is stationary for the original reduced problem. Here, $\delta_u$ replaces $\delta_{u,k}$ in $\Psi_k(1)$. This motivates our termination rule for SSN when solving \eqref{eq:SQP_optimal}. In fact, let superscript $l$ denote the iteration index of SSN for the outer iteration $k$, i.e., for given $u_k$. For some initial guess $(\delta_u^0,\lambda^0)$ (typically chosen to be $(\delta_{u,k-1},\lambda_{k-1})$) SSN computes iterates $(\delta_u^{l},\lambda^{l})$, $l\in\mathbb{N}$, and terminates at iteration $l_k$, which is the smallest index with 
\begin{equation}\label{eq:stopping_rule}
\begin{aligned}
	&\langle \mathcal{J}_{\mathcal{N}}^\prime(u_k), \delta_{u}^{l_k}\rangle_{U^*,U} + \beta_k(\Psi_k(1)-\Psi_k(0))
	\leq - \xi  \langle H_k(u_k)\delta_{u}^{l_k},\delta_{u}^{l_k}\rangle_{U^*,U}\\\
	  &\text{and}\;\quad { \Psi_k(1)} \leq (1- \xi) \Psi_k(0) 
\end{aligned}	
\end{equation}
for some $\xi\in(0,1)$, with $\beta_k>\|\lambda^{l_k}\|_U$, and where $\delta_{u}^{l_k}$ is used in $\Psi_k(1)$.
In our tests, we choose $\xi=0.9$, and terminate SSN iterations whenever  \eqref{eq:stopping_rule} is satisfied or two consecutive active set estimates are identical. Then we set $\delta_{u,k}:=\delta^{l_k}_u$, $\lambda_k:=\lambda^{l_k}$, and determine a suitable step size $\mu_k$.

For the latter we use a backtracking line search based on the Armijo condition \cite{Pow76}. Indeed, given $u_k$, $\delta_{u,k}$, and $\lambda_k$, let $l$ now denote the running index of the line search iteration. Then $l_k\in\mathbb{N}$ is the smallest index such that 
\begin{equation}\label{eq:updating_condition}
	\Phi_k(\mu_k^{l_k})-\Phi_k(0)\leq \kappa \mu _k^{l_k}\left (\langle\mathcal{J}_{\mathcal{N}}^\prime(u_k), \delta_{u,k}\rangle_{U^*,U} + \beta_k(\Psi_k(1)-\Psi_k(0)) \right),
\end{equation}
for some parameter $0<\kappa<1$, and $\beta_k=\max\{\beta_{k-1}, \zeta\|\lambda_k\|_{U}\}>\|\lambda_k\|_{U}$, for some $\zeta>1$ in (a2). In our implementation we use $\kappa=10^{-3}$ and $\zeta=2$. 

Regarding the stopping criteria for the SQP iterations, we set a tolerance for the norm of the residual of \eqref{eq:stationary1} along with a maximal number of iterations. We note here that \eqref{eq:stationary1} matches \eqref{eq:SQP_optimal} upon introducing the adjoint state for efficiently computing $\mathcal{J}_{\mathcal{N}}'(u_k)$ to the latter. 

In our implementation we simplified the Newton derivative of the first-order system \eqref{eq:stationary1} by dropping the second-order derivatives $\partial_{yy}\mathcal{N}(\cdot, y_k)p_k$ from $\mathcal{G}_{\mathcal{N}}(\phi_k)$. The corresponding approximation reads
\[
\left(
\begin{array}{llll}
- \Delta   +\partial_y \mathcal{N} (\cdot ,y_k) & 0 &- \text{Id}  & 0\\
\text{Id} & -  \Delta +   \partial_y \mathcal{N}(\cdot ,y_k) & 0 & 0\\
0 & - \text{Id} & \alpha  \text{Id} &  \text{Id}\\
0& 0 & - cG_k &   \text{Id}-G_k
\end{array}  \right)\simeq\mathcal{G}_\mathcal{N}(\phi_k) .
\]
This helped to stabilize the SSN iterations, while maintaining almost the same convergence rates as for the exact Newton derivative in our tests.

\subsection{Numerical results on distributed optimal control of semilinear elliptic PDEs }
\label{subsec:monotone_example}
Our first test problem is given by
\begin{equation} \label{eq:example_op_pde}
	\left.\begin{aligned}
		&\text{minimize}\quad \frac{1}{2}\norm{y-g}_{L^{2}(\om)}^2+\frac{\alpha}{2} \norm{u}_{L^{2}(\om)}^2,\text{ over }(y,u)\in H^1(\Omega)\times L^2(\Omega),\\
		&\text{subject to}\quad -\Delta y+ f(x,y)=u \;\text{ in } \Omega :=(0,2)\times (0,2) ,\quad \partial_\nu y=0  \;\text{ on } \partial \Omega,\\
		&\phantom{\text{subject to}\quad}-20\leq u \leq 20.
	\end{aligned}\right\}  
\end{equation}
with exact underlying nonlinearity $ f(x,z) = z+ 5\cos^2(\pi x_1x_2) z^3$ and $x=(x_1,x_2)\in\mathbb{R}^2$, $z\in\mathbb{R}$.

\subsubsection{Training  of artificial neural networks}
For learning the function $f$
we use neural networks that are built from standard (multi-layer) feed-forward networks. 
Their respective architecture together with the loss function as well as the training data and method are specified next.

\paragraph{Loss function and training method}
Let $\Theta=(W,b)$ denote the parameters associated with an ANN $\mathcal{N}=:\mathcal{N}_\Theta$ that needs to be trained by solving an associated minimization problem; compare \eqref{mh.ann.min}. 
We use here the mean squared error 
\[\mathfrak{d}(\f x,\f f) = \frac{1}{n_D} \sum_{j=1}^{n_D}\abs{\mathcal{N}_{\Theta}(\f x_j) -\f f_j}^2, \]
as a loss function, no regularization, i.e, $\mathfrak{r}\equiv 0$, and $\mathcal{F}_{\text{ad}}$ is the full space. In this context, $(\f x_j, \f f_j )_{j=1}^{n_D}$ are the input-output training pairs. For simplicity of presentation we assume that $n_D$ is larger than the number of unknowns in $\Theta$.

For solving \eqref{mh.ann.min}, we adopt a Bayesian regularization method \cite{Mac92} which is based on a Levenberg-Marquardt (LM) algorithm, 
and is available in MATLAB packages. We initialized the LM algorithm by unitary random vectors using the Nguyen-Widrow method \cite{NguWid90}, and terminated it as soon as the Euclidean norm of the gradient of the loss function dropped below $10^{-7}$ or a maximum of $1000$ iterations was reached. For other methods that are suitable for this task we refer to the overview in \cite{BotCurNoc18}.

\paragraph{Architecture of the network}
In order to have a representative study of the influence of ANN architectures on our computational results, we used networks with a total number of hidden layers (HL) equal to 1, 3 or 5. In each choice, we further varied the number of neurons per layer such that the final number of unknowns in $\Theta$ (degree(s) of freedom; DoF) remained in essence the same. Such tests were performed for three different DoF (small, medium, large) resulting in a total of nine different architectures; cf.\ Table \ref{tab:net_arc_pde}. All underlying networks operate with input layer size of three neurons and one neuron in the output layer.
In all tests for this example, the log-sigmoid transfer function (\verb+logsig+ in MATLAB) was chosen as the activation function at all the hidden layers.

\begin{table}[h!]
	\begin{center}
		\renewcommand{\arraystretch}{1.0}
		\begin{tabular}{|l|l|l|l|l|l|l|}\hline
			& HL 1 & HL 2 & HL 3 &HL 4 & HL 5  & Total DoF  \\
			\hline
			&\multicolumn{6}{c|}{Small DoF}\\
			\hline
			No. of neurons & 30 &  - & - & - & - &  151 	\\
			No. of  neurons& 6 & 10 & 5& - & - &   155 	\\
			No. of  neurons & 3 & 5 & 10& 5 & 1 &   155  \\
			\hline
			&\multicolumn{6}{c|}{Medium DoF}\\
			\hline
			No. of  neurons & 60 &  - & - & - & - &   301  	\\
			No. of  neurons&10 & 12 & 10 & - & - &  313  	\\
			No. of  neurons & 5 & 8 & 10 & 8 & 6 &  307	\\
			\hline
			&\multicolumn{6}{c|}{Large DoF}\\
			\hline
			No. of  neurons & 120 &  - & - & - & - &   601   	\\
			No. of  neurons&15 & 18 & 13& - & - &   609  	\\
			No. of  neurons & 10 & 10 & 15 & 10 & 10 &  596 	\\
			\hline
		\end{tabular}\\[8pt]
		\caption{\label{tab:net_arc_pde} Architecture of networks. 
			HL $i$: $i$ hidden layers; DoF: degrees of freedom in $\Theta$.}	
	\end{center}
\end{table}

\paragraph{Training and validation data}

The training data rest on chosen control actions $(u_j)_{j=1}^{n_D}\subset \mathcal{C}_{ad}$ with
\begin{equation*}\begin{aligned}
		u^j= &-2d_{j}\pi^2\cos(\pi x_1)\cos(\pi x_2)\\
		&-d_{j}\cos(\pi x_1)\cos(\pi x_2)-5d_{j}^3\cos^2(\pi x_1x_2) \cos^3(\pi x_1)\cos^3(\pi x_2), 
	\end{aligned}
\end{equation*}
and $(d_j)=\set{[0.01:0.4:2.01]}$ (in MATLAB notation).
The procedure for generating the training data is as follows: First, numerical solutions are computed on a uniform discrete mesh $\om_{h}=\{x^{k}\}_{k=1}^{\bar N_h}$ (represented here by the associated mesh nodes including those on $\partial\Omega$) with mesh width $h=\frac{1}{50}$, and $\bar N_h=(n_h+1)^2$, $n_h=1/h$. 
The Laplace operator is discretized by the standard five-point finite difference stencil respecting the homogeneous Neumann boundary conditions. This yields the $N_h\times N_h$-matrix $\Delta_h$ related to nodes $x_k$ in $\Omega$ with $N_h=(n_h-1)^2$. The nonlinearity as well as the controls are evaluated at such mesh points $x^k$, and the resulting discrete nonlinear PDE \eqref{eq:example_op_pde} is solved by Newton's method. The Newton iteration is terminated once the PDE residual in the discrete $H^{-1}(\Omega)$-norm drops below $10^{-16}$, or a maximum of $30$ iterations is reached. Thus for each $u^j$, $j=1,\ldots, n_D$, we obtain numerical values $y_{h}^{j}=(y_{h,1}^j,\ldots,y_{h,N_h}^j)^\top$ associated with the (interior) mesh nodes $x^{k}$ and approximating $y^j(x^k)=-d_{j}\cos(\pi x_{1}^k)\cos(\pi x_{2}^k)$, the analytical PDE solution. Using these data we compute the output values of $f$ denoted by $f^j_h\in\mathbb{R}^{N_h}$ according to the PDE by
\[f(x^{k},y^{j}(x^{k}))\approx u^{j}(x^{k})+(\Delta_h y_{h}^{j})_k=:f_{h,k}^j,\quad k=1,\ldots, N_h, \quad j=1,\ldots, n_D.\]

These input-output pairs both are prepossessed using \verb+mapminmax+ function in MATLAB without change of notation here.
The training data are then obtained through subsampling $f_{h,k}^j$ by restriction to a coarse mesh $\Omega_H$, with $H>h$. For this purpose we use $H\in\{0.2,0.1,0.08\}$ giving rise to a small, medium and large training set, respectively. The corresponding reduction rates are 1/10, 1/5, and 1/4 with respect to the data for $h=1/50$.

This subsampled data set is then split into a training data set, a validation data set and a testing data set at the ratio of $8:1:1$. In our tests, such a data partitioning is done randomly by using MATLAB's \verb+randperm+ function.

\subsubsection{Numerical results}
We start by comparing the exact, numerical and learning-based solutions, respectively. 
The exact reference solution is chosen to be 
\[y^*=1.5 \cos(\pi x_1)\cos(\pi x_2),\]
and the numerical approximation $y_h$ resulted from a mesh with $h=2^{-7}$ and the use of the exact nonlinearity $f$. The same grid is used for obtaining the numerical approximation of $y_{\mathcal{N}}$. Note, however, that the grid for data generation is different from the grid for numerical computation. 

Our report on the experiments involves several discrete norms. In fact, for $z_h\in\mathbb{R}^{N_h}$ we have \[\abs{z_h}^2_{1}:=h^2(\Delta_hz_h)^\top z_h,\quad \norm{z_h}_{0}^2:=h^2z_h^\top z_h,\]
where $\abs{\cdot}_1$ and $\norm{\cdot}_0$  correspond to the $H^1$-seminorm and $L^2$-norm, respectively. 

\begin{table}[!ht]
	\begin{center}
		\resizebox{\textwidth}{!}{
			\begin{tabular}{ |c|ll|ll|ll|ll|}
				\hline
				&  $\abs{y_\mathcal{N}-y^*_h}_1$ &   $\abs{y_\mathcal{N}-y^*_h}_1$ &
				$\abs{y_\mathcal{N}-y^*}_1$  &   $\abs{y_\mathcal{N}-y^*}_1$  				&  $\norm{y_\mathcal{N}-y^*_h}_0$ &   $\norm{y_\mathcal{N}-y^*_h}_0$ &$\norm{y_\mathcal{N}-y^*}_0$  &   $\norm{y_\mathcal{N}-y^*}_0$  \\ \hline
				&  min &   max &   min &   max &   min &   max & min &   max   \\  \hline 				
				1-L & $0.2506 $ & $0.6532   $  &  $ 0.2868   $ & $ 0.6713  $ & $0.0752  $  &  $ 0.2422  $  & $0.0808   $  &  $0.2435   $   \\ \hline
				3-L & $ 0.2575  $ & $0.7537   $  &  $ 0.2391   $ & $0.7777   $ & $ 0.0817  $ & $0.2524       $ & $0.0791   $  &  $0.2565   $  \\ \hline
				5-L & $0.2157  $ & $36.2640 $  &  $ 0.2235 $ & $ 36.2731  $ & $0.0539   $  &  $29.4926   $   & $ 0.0544 $  &  $ 29.4936  $  \\ \hline	
				&  mean&   deviation &  mean &   deviation	&  mean&   deviation &  mean&   deviation  \\  \hline 
				1-L & $0.4276  $ & $0.1099  $  &  $ 0.4496  $ & $ 0.1075 $ & $0.1472 $  &  $ 0.0484 $  & $0.1506  $  &  $0.0485  $   \\ \hline
				3-L & $ 0.3853 $ & $0.1350  $  &  $ 0.4003  $ & $0.1687  $ & $ 0.1425 $ & $0.0462      $ & $0.1268  $  &  $0.0482  $  \\ \hline
				5-L & $3.0242 $ & $ 8.9087 $  &  $3.0287  $ & $8.9103  $ & $ 2.1309 $  &  $7.3143   $   & $ 2.1299 $  &  $ 7.3149 $  \\ \hline
		\end{tabular}}
		{\small \caption{\label{tab:layer_comparison}Statistics on learning-informed PDEs with different layers in neural networks using small size training data, small DoF in $\Theta$, and 15 samples in total.}}
	\end{center}
\end{table} 

Table \ref{tab:layer_comparison} depicts the approximation results for different ANN architectures with small DoF as described in \ref{tab:net_arc_pde} and in all cases the small training data set.

We find that the $1$-layer network is robust in terms of the statistical quantities shown, and the $3$-layer network has the smallest errors on average, but exhibits a larger deviation than the $1$-layer network. The $5$-layer network yields the smallest error, but also the largest ones with a very big deviation. 
This behavior may be attributed to the fact that deeper networks give rise to increasingly more nonlinear compositions entering the loss function. This may be stabilized by tuned initializations, additional regularization, or sufficient training data. A study along these lines, however, is not within the scope of the present work as noted earlier.

\begin{table}[!ht]
	\begin{center}
		\resizebox{\textwidth}{!}{
			\begin{tabular}{ |c|ll|ll|ll|ll|}
				\hline
				&  $\abs{y_\mathcal{N}-y^*_h}_1$ &   $\abs{y_\mathcal{N}-y^*_h}_1$ &
				$\abs{y_\mathcal{N}-y^*}_1$  &   $\abs{y_\mathcal{N}-y^*}_1$  				&  $\norm{y_\mathcal{N}-y^*_h}_0$ &   $\norm{y_\mathcal{N}-y^*_h}_0$ &$\norm{y_\mathcal{N}-y^*}_0$  &   $\norm{y_\mathcal{N}-y^*}_0$  \\ \hline
				&  min &   max &   min &   max &   min &   max & min &   max   \\  \hline 				
				3-L S & $0.0546 $ & $0.1658   $  &  $ 0.0889   $ & $ 0.2211  $ & $0.0086  $  &  $ 0.0546  $  & $0.0207   $  &  $0.0515   $   \\ \hline
				3-L M & $ 0.0090  $ & $0.1508   $  &  $ 0.0876   $ & $0.2039   $ & $ 0.0026  $ & $0.0492       $ & $0.0168   $  &  $0.0591   $  \\ \hline
				3-L L & $0.0155  $ & $0.2815 $  &  $ 0.0833 $ & $ 0.3306  $ & $0.0036   $  &  $0.0901   $   & $ 0.0161 $  &  $ 0.0996  $  \\ \hline	
				&  mean&   deviation &  mean &   deviation	&  mean&   deviation &  mean&   deviation  \\  \hline 
				3-L S& $0.1103  $ & $0.0357  $  &  $ 0.1464  $ & $ 0.0329 $ & $0.0266 $  &  $ 0.0125 $  & $0.0339  $  &  $0.0095  $   \\ \hline
				3-L M& $ 0.0631 $ & $0.0407  $  &  $ 0.1113  $ & $0.0367  $ & $ 0.0170 $ & $0.0120      $ & $0.0250  $  &  $0.0117  $  \\ \hline
				3-L L & $ 0.0559 $ & $ 0.0626 $  &  $0.1115  $ & $0.0609  $ & $ 0.0149 $  &  $0.0205   $   & $ 0.0250 $  &  $ 0.0204 $  \\ \hline
		\end{tabular}}
		\caption{\label{tab:width_comparison}Statistics on learning-informed PDEs with different numbers of neurons in networks using medium size training data of 15 samples in total.}
		
	\end{center}
\end{table} 

In Table \ref{tab:width_comparison}, we provide statistics on the influence of the number of neurons for fixed layers. We use $3$-layer networks and medium sized training data for this set of experiments. All three levels of DoF for the networks as given in Table \ref{tab:net_arc_pde} are studied. The results in terms of 'mean' and 'deviation' indicate that a large number of neurons gives typically better approximations when compared to the smaller size of DoFs.
However, we also observe that the deviation and the maximum error increases with the number of DoF. 
This can be attributed to an increase in training error for increasing DoFs.

Next we present some computational results where we  use the learning-informed PDE as constraint when numerically solving the optimal control  problem \eqref{eq:cost_NN}. 
Here we consider a target function $g=y^*+\delta$ where $\delta$ is a variable denoting zero-mean Gaussian noise of standard deviation $\hat\sigma$, for different values of $\hat\sigma$.
For convenience of comparison, we take $y^*$ to be the solution from the last set of experiments.
We denote by $u_\mathcal{N}$ and $\bar{u}$   the optimal controls
with respect to the learning-informed PDE constraint and the original PDE constraint, respectively, both computed by the semi-smooth Newton algorithm as described in Section \ref{subsec:Newton} with a fixed number of $30$ iterations which turns out to be sufficient for this example, 
as the sum of all residual norms of the first-order system \eqref{eq:stationary1} is less than $10^{-10}$.
As before,  $y_\mathcal{N}$ and $\bar{y}$ are the states corresponding to  $u_\mathcal{N}$ and $\bar{u}$, respectively.

\begin{table}[!ht]
	\begin{center}
		\resizebox{\textwidth}{!}{
			\begin{tabular}{ |c|lll|lll|lll|} \hline
				&\multicolumn{3}{c|}{Small DoF} &\multicolumn{3}{c|}{Medium  DoF}  
				&\multicolumn{3}{c|}{Large DoF}  \\  \hline 
				&  $\norm{u_\mathcal{N}-\bar{u}}_0$ & $\norm{y_\mathcal{N}- \bar{y}}_0$ &  $ \abs{y_\mathcal{N}- \bar{y}}_1$ &  $\norm{u_\mathcal{N}-\bar{u}}_0$ & $\norm{y_\mathcal{N}- \bar{y}}_0$ &  $ \abs{y_\mathcal{N}- \bar{y}}_1$ &  $\norm{u_\mathcal{N}-\bar{u}}_0$ & $\norm{y_\mathcal{N}- \bar{y}}_0$ &  $ \abs{y_\mathcal{N}- \bar{y}}_1$   \\ \hline
				& \multicolumn{9}{c|}{Small size of training data }    \\ \hline
				1-L & $0.5578 $ & $0.0330 $  &  $0.1609  $ & $0.3055 $ & $0.0283 $  &  $0.1423 $ &  $ 0.2548$ & $0.0194 $  &  $ 0.1143 $    \\ \hline
				3-L  & $ 0.3426$ & $ 0.0274$  &  $0.1246  $ & $ 0.3597$ & $0.0343 $  &  $0.1777  $ & ${\bf 0.3932 }$ & ${\bf 0.0354 }$  &  $ {\bf0.1722 } $  \\ \hline
				5-L  & $ 0.3888 $ & $ 0.0183 $  &  $  0.1041$ & $ 0.1771$ & $0.0117 $  &  $ 0.0666 $ & $ 0.3986$ & $0.0359 $  &  $ 0.1698$  \\ \hline
				& \multicolumn{9}{c|}{Medium size of training data }    \\ \hline				
				1-L  & $0.2145 $ & $0.0071 $  &  $  0.0413$ & $ 0.1153$ & $0.0072 $  &  $0.0587 $ & $ 0.0655$ & $0.0029 $  &  $  0.0244$  \\ \hline
				3-L  & $0.1647 $ & $ 0.0069$  &  $0.0419  $ & $ 0.0985$ & $0.0082 $  &  $ 0.0423 $ & ${\bf 0.0623} $ & ${\bf 0.0046 }$  &  ${\bf0.0287 }  $   \\ \hline
				5-L  & $0.2971 $ & $0.0271 $  &  $ 0.1223 $ & $0.0325$ & $ 0.0014$  &  $0.0081  $ & $ 0.0736$ & $0.0064 $  &  $0.0414  $   \\ \hline
				& \multicolumn{9}{c|}{Large size of training data }    \\ \hline				
				1-L  & $0.1417 $ & $0.0089 $  &  $ 0.0481 $ & $ 0.0920$ & $0.0040 $  &  $0.0266  $ & $0.0447 $ & $0.0009 $  &  $ 0.0055 $   \\ \hline
				3-L  & $0.0566 $ & $0.0020 $  &  $ 0.0126 $ & $ 0.0467$ & $0.0024 $  &  $0.0122  $ & ${\bf 0.0076 }$ & ${\bf 0.0004} $  &  ${\bf  0.0020 }$  \\ \hline
				5-L  & $ 0.1239$ & $0.0070 $  &  $ 0.0435$ & $0.2135$ & $0.0098 $  &  $ 0.0645 $ & $0.0192 $ & $0.0018 $  &  $ 0.0115 $   \\ \hline\addlinespace[5pt] 
				\multicolumn{10}{c}{ Using the same noisy data $g$ (Gaussian noise of mean zero and deviation $0.1$) with $\alpha=0.001$ in all the tests}  \\  
		\end{tabular}}
		\caption{Optimal control with learning-informed PDEs using different layers, different size of networks, and a variety of training data.}		\label{tab:result_op_control}
	\end{center}
\end{table} 

In general, we observe in Table \ref{tab:result_op_control} that most combinations give similar results. This shows the robustness of our proposed method with respect to a wide range of network architectures.
Here, the presented errors are just computed from one specific initialization. 

Note that when using $3$-hidden-layer networks with large DoF, we observe a clear increase in the levels of accuracy for both the control and state variables as the training data increase from small to large size. These are highlighted with bold font numbers in Table \ref{tab:result_op_control}.
A similar behavior occurs for $1$-hidden-layer and $5$-hidden-layer networks.
By fixing the  $3$-hidden-layer networks, and for each case of DoFs provided in Table \ref{tab:result_op_control}, we are next interested in exploring how the noise level $\sigma$ and the cost parameter $\alpha$ further influence the optimal control approximation.
\begin{table}[!ht]
	
	\begin{center}
		\resizebox{\textwidth}{!}{
			\begin{tabular}{ |c|lll|lll|lll|}
				\hline 
				&\multicolumn{3}{c|}{Noise free} &\multicolumn{3}{c|}{Mild noise $\hat\sigma=0.05$}  
				&\multicolumn{3}{c|}{Larger noise $\hat\sigma=0.5$}  \\  \hline 
				&  $\norm{u_\mathcal{N}-\bar{u}}_0$ & $\norm{y_\mathcal{N}- \bar{y}}_0$ &  $ \abs{y_\mathcal{N}- \bar{y}}_1$ &  $\norm{u_\mathcal{N}-\bar{u}}_0$ & $\norm{y_\mathcal{N}- \bar{y}}_0$ &  $ \abs{y_\mathcal{N}- \bar{y}}_1$ &  $\norm{u_\mathcal{N}-\bar{u}}_0$ & $\norm{y_\mathcal{N}- \bar{y}}_0$ &  $ \abs{y_\mathcal{N}- \bar{y}}_1$   \\ \hline
				& \multicolumn{9}{c|}{$\alpha =0.00001$}    \\ \hline
				3-L-S NN & $1.9523 $ & $0.0210 $  &  $0.2041  $ & $1.9518$ & $0.0210 $  &  $0.2043$ &  $ 2.1480$ & $0.0213 $  &  $ 0.2085 $    \\ \hline
				3-L-M  NN & $ 0.1187$ & $ 0.0018$  &  $0.0253  $ & $ 0.1190$ & $0.0018 $  &  $0.0253  $ & $0.1264 $ & $0.0018 $  &  $ 0.0254 $  \\ \hline
				3-L-L  NN & $ 0.0213$ & $0.0004 $  &  $0.0046 $ & $ 0.0215$ & $0.0004 $  &  $ 0.0046 $ & $ 0.0258$ & $0.0004 $  &  $ 0.0047$  \\ \hline
				& \multicolumn{9}{c|}{$\alpha=0.0001$}    \\ \hline				
				3-L-S  NN & $1.3489 $ & $0.0395 $  &  $  0.2695$ & $ 1.3560$ & $0.0397 $  &  $0.2705 $ & $ 1.4181$ & $0.0410 $  &  $  0.2796$  \\ \hline
				3-L-M  NN& $0.1361 $ & $ 0.0032$  &  $0.0314  $ & $ 0.1357$ & $0.0032 $  &  $ 0.0314 $ & $0.1384 $ & $ 0.0032$  &  $0.0315 $   \\ \hline
				3-L-L  NN  & $0.0137$ & $0.0005  $  &  $ 0.0039  $ & $0.0136  $ & $ 0.0005 $  &  $0.0039   $ & $ 0.0136 $ & $0.0005  $  &  $0.0039   $   \\ \hline
				& \multicolumn{9}{c|}{$\alpha=0.001$ }    \\ \hline				
				3-L-S  NN & $0.3903  $ & $0.0350  $  &  $ 0.1706  $ & $ 0.3917 $ & $0.0352  $  &  $0.1714   $ & $0.4067  $ & $0.0371  $  &  $ 0.1792  $   \\ \hline
				3-L-M  NN & $0.0628 $ & $0.0046  $  &  $ 0.0286  $ & $ 0.0630$ & $0.0046 $  &  $0.0286  $ & $0.0671  $ & $0.0046  $  &  $ 0.0293  $  \\ \hline
				3-L-L  NN & $ 0.0076$ & $0.0004 $  &  $ 0.0020$ & $0.0076$ & $0.0004 $  &  $ 0.0020 $ & $0.0080 $ & $0.0004 $  &  $ 0.0021 $   \\ \hline
				& \multicolumn{9}{c|}{$\alpha=0.01$ }    \\ \hline				
				3-L-S  NN & $0.0570  $ & $0.0066  $  &  $ 0.0209 $ & $ 0.0572$ & $0.0066 $  &  $0.0210   $ & $0.0592  $ & $0.0069  $  &  $ 0.0217  $   \\ \hline
				3-L-M  NN & $0.0271  $ & $0.0020  $  &  $ 0.0080  $ & $ 0.0271$ & $0.0021 $  &  $0.0081  $ & $0.0277 $ & $0.0022 $  &  $ 0.0083 $  \\ \hline
				3-L-L  NN & $ 0.0035$ & $0.0003 $  &  $ 0.0008$ & $0.0035$ & $0.0003 $  &  $ 0.0008 $ & $0.0035 $ & $0.0003 $  &  $ 0.0008 $   \\ \hline\addlinespace[5pt] 
				\multicolumn{10}{c}{ Variant level of noise in $g$ with  respect to different $\alpha$ and coarser to finer neural networks}  \\
		\end{tabular}}
		\caption{Optimal control on learning-informed PDEs with networks by 3 layers networks, but different sizes on the neurons (DoF), and a variant amount of training data.}	
		\label{tab:result_op_control_2}	
	\end{center}
\end{table} 

From Table \ref{tab:result_op_control_2} we draw several interesting conclusions. In both, the noisy and noise free case, we have that the error $\norm{u_\mathcal{N}-\bar{u}}$ is proportional to the accuracy of the neural network approximation, and inverse proportional to $\sqrt{\alpha}$.
This verifies the results of Theorem \ref{thm:error_bound} and Theorem \ref{thm:error_bound2}, respectively.
The dependence on $\alpha$ could only be proved for the noise-free case in Theorem \ref{thm:error_bound}.
Therefore the convergence rates provided by our tests here seem to indicate that better convergence rates or more relaxed assumptions appear plausible.

\subsection{Numerical results on optimal control of stationary Allen-Cahn equation}
Next we study the optimal control of the Allen-Cahn equation, which involves a nonmonotone $f$ and reads
\begin{equation} \label{eq:Allen_Cahn}
	-\Delta y+ \frac{1}{\eta}(y^3-y) =u\quad \text{ in }\; \Omega ,\quad \partial_\nu y=0 \quad \text{ on }\; \partial \Omega,
\end{equation}
with $\eta>0$. 
In our numerical tests, we set $\eta=0.004$, use $\Omega=(0,2)^2$, and $h:=2^{-7}$. 

We focus on  $3$-hidden-layer neural networks with $10$, $12$ and $10$ neurons per layer yielding DoF$=293$. In each hidden layer we use log-sigmoid transfer functions.
Note also that since the input data here does not depend explicitly on the spatial variable $x$, i.e., $f=f(y)$, both the input and output layers have only one neuron, respectively. This is different to the previous test examples.

In our tests, we obtained the training data by solving the PDE in \eqref{eq:Allen_Cahn} with 
\[ u=u^d:=\left\{\begin{aligned}
1000,  & \quad x\in \Omega^l:=(0,2)\times (0,1),\\
-1000, & \quad x\in \Omega/\Omega^l.
\end{aligned} \right. \]
In order to train  the neural networks described above, the solution of the PDE is subsampled uniformly at a rate of $0.25$, that is $H=0.08$.
As $f$ has an one dimensional image space, it suffices that the data $u^d$ correspond to a PDE solution that has a relatively wide range of values. Indeed, using our choice of $u^{d}$, the value of the corresponding solution $y$ varies between $-2.5$ and $2.5$ which turns out to be sufficient for learning $f$.

\begin{figure}[h!]
	\begin{minipage}[t]{0.32\textwidth}
		\centering
		\resizebox{0.95\textwidth}{!}{
%
%
\begin{tikzpicture}

\begin{axis}[%
width=3.014in,
height=3.509in,
at={(0.506in,0.474in)},
scale only axis,
xmin=-10,
xmax=10,
xtick={-10, -5, 0, 5, 10},
xticklabels={$-10$, $-5$, $0$, $5$, $10$},
xticklabel style={font=\large},  
ymin=-100000,
ymax=700000,
yticklabel style={font=\large},  
axis background/.style={fill=white},
legend style={at={(0.5,1)}, anchor=north, legend cell align=left, align=left, draw=white!15!black}
]
\addplot [color=red,  line width=2.0pt]
  table[row sep=crcr]{%
-10	60815.0538626492\\
-9.9	59969.1384306213\\
-9.8	59123.2363046311\\
-9.7	58277.3486469942\\
-9.6	57431.4767280262\\
-9.5	56585.6219369774\\
-9.4	55739.7857942025\\
-9.3	54893.9699647212\\
-9.2	54048.1762733497\\
-9.1	53202.4067216058\\
-9	52356.6635066249\\
-8.9	51510.9490423509\\
-8.8	50665.2659833144\\
-8.7	49819.6172513482\\
-8.6	48974.0060656489\\
-8.5	48128.4359766532\\
-8.4	47282.9109042685\\
-8.3	46437.4351810828\\
-8.2	45592.0136012759\\
-8.1	44746.6514760656\\
-8	43901.3546966597\\
-7.9	43056.1298058361\\
-7.8	42210.9840794565\\
-7.7	41365.9256194349\\
-7.6	40520.9634599287\\
-7.5	39676.1076888183\\
-7.4	38831.3695868855\\
-7.3	37986.7617875116\\
-7.2	37142.298460198\\
-7.1	36297.9955217808\\
-7	35453.8708798914\\
-6.9	34609.9447140112\\
-6.8	33766.2398004259\\
-6.7	32922.781888514\\
-6.6	32079.6001371609\\
-6.5	31236.7276216998\\
-6.4	30394.201923719\\
-6.3	29552.0658183876\\
-6.2	28710.3680767425\\
-6.1	27869.164403726\\
-6	27028.5185368106\\
-5.9	26188.5035349284\\
-5.8	25349.2032933274\\
-5.7	24510.7143271346\\
-5.6	23673.1478750695\\
-5.5	22836.6323852609\\
-5.4	22001.3164578397\\
-5.3	21167.3723343553\\
-5.2	20335.0000425639\\
-5.1	19504.4323272698\\
-5	18675.9405240951\\
-4.9	17849.84156359\\
-4.8	17026.5063278082\\
-4.7	16206.3696194111\\
-4.6	15389.9420420367\\
-4.5	14577.8241249773\\
-4.4	13770.723045565\\
-4.3	12969.4722923076\\
-4.2	12175.0545426697\\
-4.1	11388.6278566916\\
-4	10611.5549435078\\
-3.9	9845.434646216\\
-3.8	9092.13378906791\\
-3.7	8353.81601167698\\
-3.6	7632.96210495571\\
-3.5	6932.37377195627\\
-3.4	6255.15016198004\\
-3.3	5604.62509205839\\
-3.2	4984.25439614118\\
-3.1	4397.44937825128\\
-3	3847.36482556128\\
-2.9	3336.66635233565\\
-2.8	2867.31585976671\\
-2.7	2440.41751724179\\
-2.6	2056.15513197939\\
-2.5	1713.82834241481\\
-2.4	1411.97013178942\\
-2.3	1148.51245158886\\
-2.2	920.964627577319\\
-2.1	726.57756044398\\
-2	562.478947373949\\
-1.9	425.775415104477\\
-1.8	313.624234454112\\
-1.7	223.280342748477\\
-1.6	152.12496129906\\
-1.5	97.6813974673215\\
-1.4	57.6224975558198\\
-1.3	29.7730936109491\\
-1.2	12.1098234374516\\
-1.1	2.75992566427831\\
-1	-0\\
-0.9	2.25524973415308\\
-0.799999999999999	8.09936950577194\\
-0.699999999999999	16.2549953171682\\
-0.6	25.5944888851179\\
-0.5	35.140776627324\\
-0.399999999999999	44.0679911912547\\
-0.299999999999999	51.7017491376024\\
-0.199999999999999	57.519014306861\\
-0.0999999999999996	61.1476119464469\\
0	62.3655462164512\\
0.0999999999999996	61.1003141871449\\
0.199999999999999	57.4283963853072\\
0.299999999999999	51.5750443259968\\
0.399999999999999	43.9143974470667\\
0.5	34.9698698548795\\
0.6	25.414675720769\\
0.699999999999999	16.0723295574924\\
0.799999999999999	7.91697231649416\\
0.9	2.07343279259055\\
1	-0.182978144361531\\
1.1	2.57315240579525\\
1.2	11.9169546570685\\
1.3	29.5730877983845\\
1.4	57.4159093225666\\
1.5	97.4700252781052\\
1.6	151.910975157978\\
1.7	223.065357489606\\
1.8	313.409010337217\\
1.9	425.560841979336\\
2	562.268488062866\\
2.1	726.380165582999\\
2.2	920.795154065395\\
2.3	1148.38398805874\\
2.4	1411.87011566248\\
2.5	1713.66950072255\\
2.6	2055.69520237169\\
2.7	2439.14994387314\\
2.8	2864.34626540127\\
2.9	3330.60178469007\\
3	3836.24734758569\\
3.1	4378.7575708525\\
3.2	4954.97776373639\\
3.3	5561.39585515328\\
3.4	6194.40392107654\\
3.5	6850.50910332515\\
3.6	7526.4767946373\\
3.7	8219.40869100988\\
3.8	8926.76940949113\\
3.9	9646.3783574337\\
4	10376.3815644243\\
4.1	11115.2143245931\\
4.2	11861.5616588597\\
4.3	12614.3205981361\\
4.4	13372.5662359719\\
4.5	14135.5222366186\\
4.6	14902.5357797381\\
4.7	15673.0565743085\\
4.8	16446.6194358013\\
4.9	17222.829897206\\
5	18001.3523581403\\
5.1	18781.9003339782\\
5.2	19564.2284308126\\
5.3	20348.1257332424\\
5.4	21133.410346562\\
5.5	21919.924881748\\
5.6	22707.5327108304\\
5.7	23496.1148525478\\
5.8	24285.5673745826\\
5.9	25075.7992200934\\
6	25866.7303835964\\
6.1	26658.2903752415\\
6.2	27450.4169238262\\
6.3	28243.0548780041\\
6.4	29036.155272515\\
6.5	29829.6745322216\\
6.6	30623.5737915701\\
6.7	31417.8183110143\\
6.8	32212.3769751367\\
6.9	33007.2218598059\\
7	33802.3278578415\\
7.1	34597.6723544018\\
7.2	35393.2349447519\\
7.3	36188.997188249\\
7.4	36984.942393365\\
7.5	37781.0554293755\\
7.6	38577.3225610203\\
7.7	39373.731303001\\
7.8	40170.2702916541\\
7.9	40966.9291715271\\
8	41763.6984949203\\
8.1	42560.5696327307\\
8.2	43357.5346951741\\
8.3	44154.5864611558\\
8.4	44951.7183152331\\
8.5	45748.9241912526\\
8.6	46546.1985218708\\
8.7	47343.5361932679\\
8.8	48140.9325044576\\
8.9	48938.3831306689\\
9	49735.8840903439\\
9.1	50533.4317153533\\
9.2	51331.0226240769\\
9.3	52128.6536970432\\
9.4	52926.322054856\\
9.5	53724.0250381687\\
9.6	54521.7601894956\\
9.7	55319.5252366729\\
9.8	56117.3180778024\\
9.9	56915.1367675327\\
10	57712.9795045452\\
};
\addlegendentry{\large{NN approximation of $F$}}

\addplot [color=blue,  line width=2.0pt]
  table[row sep=crcr]{%
-10	612562.5\\
-9.9	588183.75625\\
-9.8	564537.6\\
-9.7	541609.25625\\
-9.6	519384.1\\
-9.5	497847.65625\\
-9.4	476985.6\\
-9.3	456783.75625\\
-9.2	437228.1\\
-9.1	418304.75625\\
-9	400000\\
-8.9	382300.25625\\
-8.8	365192.1\\
-8.7	348662.25625\\
-8.6	332697.6\\
-8.5	317285.15625\\
-8.4	302412.1\\
-8.3	288065.75625\\
-8.2	274233.6\\
-8.1	260903.25625\\
-8	248062.5\\
-7.9	235699.25625\\
-7.8	223801.6\\
-7.7	212357.75625\\
-7.6	201356.1\\
-7.5	190785.15625\\
-7.4	180633.6\\
-7.3	170890.25625\\
-7.2	161544.1\\
-7.1	152584.25625\\
-7	144000\\
-6.9	135780.75625\\
-6.8	127916.1\\
-6.7	120395.75625\\
-6.6	113209.6\\
-6.5	106347.65625\\
-6.4	99800.1\\
-6.3	93557.25625\\
-6.2	87609.5999999999\\
-6.1	81947.75625\\
-6	76562.5\\
-5.9	71444.75625\\
-5.8	66585.6\\
-5.7	61976.25625\\
-5.6	57608.1\\
-5.5	53472.65625\\
-5.4	49561.6\\
-5.3	45866.75625\\
-5.2	42380.1\\
-5.1	39093.75625\\
-5	36000\\
-4.9	33091.25625\\
-4.8	30360.1\\
-4.7	27799.25625\\
-4.6	25401.6\\
-4.5	23160.15625\\
-4.4	21068.1\\
-4.3	19118.75625\\
-4.2	17305.6\\
-4.1	15622.25625\\
-4	14062.5\\
-3.9	12620.25625\\
-3.8	11289.6\\
-3.7	10064.75625\\
-3.6	8940.09999999999\\
-3.5	7910.15625\\
-3.4	6969.59999999999\\
-3.3	6113.25625\\
-3.2	5336.09999999999\\
-3.1	4633.25625\\
-3	4000\\
-2.9	3431.75625\\
-2.8	2924.1\\
-2.7	2472.75625\\
-2.6	2073.6\\
-2.5	1722.65625\\
-2.4	1416.1\\
-2.3	1150.25625\\
-2.2	921.599999999999\\
-2.1	726.756249999999\\
-2	562.5\\
-1.9	425.75625\\
-1.8	313.599999999999\\
-1.7	223.256249999999\\
-1.6	152.1\\
-1.5	97.65625\\
-1.4	57.6000000000001\\
-1.3	29.7562499999998\\
-1.2	12.0999999999999\\
-1.1	2.75624999999998\\
-1	0\\
-0.9	2.25624999999998\\
-0.799999999999999	8.10000000000008\\
-0.699999999999999	16.2562500000001\\
-0.6	25.6\\
-0.5	35.15625\\
-0.399999999999999	44.1000000000001\\
-0.299999999999999	51.7562500000001\\
-0.199999999999999	57.6\\
-0.0999999999999996	61.25625\\
0	62.5\\
0.0999999999999996	61.25625\\
0.199999999999999	57.6\\
0.299999999999999	51.7562500000001\\
0.399999999999999	44.1000000000001\\
0.5	35.15625\\
0.6	25.6\\
0.699999999999999	16.2562500000001\\
0.799999999999999	8.10000000000008\\
0.9	2.25624999999998\\
1	-4.32986979603811e-15\\
1.1	2.75624999999998\\
1.2	12.0999999999999\\
1.3	29.7562499999998\\
1.4	57.6000000000001\\
1.5	97.65625\\
1.6	152.1\\
1.7	223.256249999999\\
1.8	313.599999999999\\
1.9	425.75625\\
2	562.5\\
2.1	726.756249999999\\
2.2	921.599999999998\\
2.3	1150.25625\\
2.4	1416.1\\
2.5	1722.65625\\
2.6	2073.6\\
2.7	2472.75625\\
2.8	2924.1\\
2.9	3431.75625\\
3	4000\\
3.1	4633.25625\\
3.2	5336.09999999999\\
3.3	6113.25625\\
3.4	6969.6\\
3.5	7910.15625\\
3.6	8940.1\\
3.7	10064.75625\\
3.8	11289.6\\
3.9	12620.25625\\
4	14062.5\\
4.1	15622.25625\\
4.2	17305.6\\
4.3	19118.75625\\
4.4	21068.1\\
4.5	23160.15625\\
4.6	25401.6\\
4.7	27799.25625\\
4.8	30360.1\\
4.9	33091.25625\\
5	36000\\
5.1	39093.75625\\
5.2	42380.1\\
5.3	45866.75625\\
5.4	49561.6\\
5.5	53472.65625\\
5.6	57608.1\\
5.7	61976.25625\\
5.8	66585.6\\
5.9	71444.75625\\
6	76562.5\\
6.1	81947.75625\\
6.2	87609.6\\
6.3	93557.25625\\
6.4	99800.1\\
6.5	106347.65625\\
6.6	113209.6\\
6.7	120395.75625\\
6.8	127916.1\\
6.9	135780.75625\\
7	144000\\
7.1	152584.25625\\
7.2	161544.1\\
7.3	170890.25625\\
7.4	180633.6\\
7.5	190785.15625\\
7.6	201356.1\\
7.7	212357.75625\\
7.8	223801.6\\
7.9	235699.25625\\
8	248062.5\\
8.1	260903.25625\\
8.2	274233.6\\
8.3	288065.75625\\
8.4	302412.1\\
8.5	317285.15625\\
8.6	332697.6\\
8.7	348662.25625\\
8.8	365192.1\\
8.9	382300.25625\\
9	400000\\
9.1	418304.75625\\
9.2	437228.1\\
9.3	456783.75625\\
9.4	476985.6\\
9.5	497847.65625\\
9.6	519384.1\\
9.7	541609.25625\\
9.8	564537.6\\
9.9	588183.75625\\
10	612562.5\\
};
\addlegendentry{\large{original $F$}}

\end{axis}
\end{tikzpicture}%
}
	\end{minipage}
	\begin{minipage}[t]{0.32\textwidth}
		\centering
		\resizebox{0.95\textwidth}{!}{
%
%
\begin{tikzpicture}

\begin{axis}[%
width=3.014in,
height=3.509in,
at={(0.506in,0.464in)},
scale only axis,
xmin=-10,
xmax=10,
xtick={-10, -5, 0, 5, 10},
xticklabels={$-10$, $-5$, $0$, $5$, $10$},
xticklabel style={font=\large},  
ymin=-250000,
ymax=250000,
ytick={-300000, -200000, -100000, 0 , 100000, 200000, 300000},
yticklabel style={font=\large},  
axis background/.style={fill=white},
legend style={at={(0.5,1)}, anchor=north,  legend cell align=left, align=left, draw=white!15!black}
]
\addplot [color=red,  line width=2.0pt]
  table[row sep=crcr]{%
-10	-8459.21722606949\\
-9.9	-8459.08964237354\\
-9.8	-8458.95094187516\\
-9.7	-8458.80009548894\\
-9.6	-8458.6359704784\\
-9.5	-8458.45731882896\\
-9.4	-8458.26276415338\\
-9.3	-8458.05078692356\\
-9.2	-8457.81970779059\\
-9.1	-8457.56766871867\\
-9	-8457.29261161517\\
-8.9	-8456.99225408986\\
-8.8	-8456.66406191775\\
-8.7	-8456.30521771224\\
-8.6	-8455.91258523598\\
-8.5	-8455.48266868367\\
-8.4	-8455.01156616259\\
-8.3	-8454.49491646854\\
-8.2	-8453.92783810544\\
-8.1	-8453.30485932015\\
-8	-8452.61983771713\\
-7.9	-8451.86586777251\\
-7.8	-8451.03517427916\\
-7.7	-8450.11898941285\\
-7.6	-8449.10741070604\\
-7.5	-8447.98923673766\\
-7.4	-8446.75177677852\\
-7.3	-8445.38062995706\\
-7.2	-8443.85942870471\\
-7.1	-8442.16954028007\\
-7	-8440.28971902263\\
-6.9	-8438.19570061146\\
-6.8	-8435.85972795324\\
-6.7	-8433.24999633959\\
-6.6	-8430.33000312139\\
-6.5	-8427.05778426089\\
-6.4	-8423.38501662968\\
-6.3	-8419.25596068812\\
-6.2	-8414.60621304422\\
-6.1	-8409.36123214364\\
-6	-8403.43459273994\\
-5.9	-8396.72591553422\\
-5.8	-8389.11840709376\\
-5.7	-8380.4759314314\\
-5.6	-8370.63951795557\\
-5.5	-8359.4231903311\\
-5.4	-8346.60897654708\\
-5.3	-8331.94093163713\\
-5.2	-8315.11797069945\\
-5.1	-8295.78527122266\\
-5	-8273.5239612632\\
-4.9	-8247.83876647347\\
-4.8	-8218.1432501411\\
-4.7	-8183.74225732302\\
-4.6	-8143.81118668568\\
-4.5	-8097.37179588913\\
-4.4	-8043.26445473869\\
-4.3	-7980.11718498214\\
-4.2	-7906.31260427062\\
-4.1	-7819.95522298469\\
-4	-7718.84368796191\\
-3.9	-7600.45581824151\\
-3.8	-7461.95883992452\\
-3.7	-7300.26294074148\\
-3.6	-7112.14207681958\\
-3.5	-6894.44914322981\\
-3.4	-6644.44804112427\\
-3.3	-6360.2656935213\\
-3.2	-6041.42679005095\\
-3.1	-5689.37584600601\\
-3	-5307.83568410626\\
-2.9	-4902.83748699056\\
-2.8	-4482.32146043075\\
-2.7	-4055.34553594006\\
-2.6	-3631.0898750893\\
-2.5	-3217.91943315134\\
-2.4	-2822.72191784278\\
-2.3	-2450.61051965263\\
-2.2	-2104.95041022724\\
-2.1	-1787.5950469436\\
-2	-1499.211309497\\
-1.9	-1239.60490059309\\
-1.8	-1007.99845198704\\
-1.7	-803.246540800441\\
-1.6	-623.990072523214\\
-1.5	-468.760145525154\\
-1.4	-336.043065149903\\
-1.3	-224.317060218823\\
-1.2	-132.069393733549\\
-1.1	-57.8007658431135\\
-1	-0.0223923004984954\\
-0.9	42.7501305794619\\
-0.799999999999999	72.0031871509598\\
-0.699999999999999	89.2297858692643\\
-0.6	95.9315185140597\\
-0.5	93.6177432638003\\
-0.399999999999999	83.8026000552718\\
-0.299999999999999	68.0009167479182\\
-0.199999999999999	47.7242139526475\\
-0.0999999999999996	24.4778645866106\\
0	-0.239929508873502\\
0.0999999999999996	-24.9372073336101\\
0.199999999999999	-48.1257505049325\\
0.299999999999999	-68.3177926943068\\
0.399999999999999	-84.0228797024746\\
0.5	-93.7457356329016\\
0.6	-95.9856794887697\\
0.699999999999999	-89.2376858493715\\
0.799999999999999	-71.9947077320634\\
0.9	-42.750497070056\\
1	-0.00196431354152082\\
1.1	57.7498279559986\\
1.2	132.00062189536\\
1.3	224.245873204802\\
1.4	335.984770976462\\
1.5	468.723394661255\\
1.6	623.973514780597\\
1.7	803.24146750081\\
1.8	1007.9986938879\\
1.9	1239.62184065409\\
2	1499.28614274327\\
2.1	1787.79262905111\\
2.2	2105.31097992702\\
2.3	2451.03165556712\\
2.4	2822.75061452232\\
2.5	3216.46106699756\\
2.6	3626.08868735166\\
2.7	4043.54548321185\\
2.8	4459.24254568037\\
2.9	4863.06349425789\\
3	5245.5996519786\\
3.1	5599.30027127164\\
3.2	5919.20555885739\\
3.3	6203.10752055545\\
3.4	6451.20860417287\\
3.5	6665.49138004576\\
3.6	6849.02791988298\\
3.7	7005.38608120622\\
3.8	7138.19891935524\\
3.9	7250.89647821759\\
4	7346.56649938685\\
4.1	7427.90340155018\\
4.2	7497.210674012\\
4.3	7556.43151427169\\
4.4	7607.19151586555\\
4.5	7650.84394837812\\
4.6	7688.51265329796\\
4.7	7721.13031683763\\
4.8	7749.47143270039\\
4.9	7774.18007767295\\
5	7795.79299671772\\
5.1	7814.75862610915\\
5.2	7831.45268953848\\
5.3	7846.19094963971\\
5.4	7859.23962250503\\
5.5	7870.82388430213\\
5.6	7881.13482602327\\
5.7	7890.3351482686\\
5.8	7898.56383360426\\
5.9	7905.93998891928\\
6	7912.56601326247\\
6.1	7918.53021665014\\
6.2	7923.90899111848\\
6.3	7928.76861580369\\
6.4	7933.16676217133\\
6.5	7937.15375293912\\
6.6	7940.77361813494\\
6.7	7944.06498361096\\
6.8	7947.06182079792\\
6.9	7949.79408121413\\
7	7952.28823498788\\
7.1	7954.56772920701\\
7.2	7956.65337911604\\
7.3	7958.56370290944\\
7.4	7960.31520901877\\
7.5	7961.92264327892\\
7.6	7963.39920211984\\
7.7	7964.75671691309\\
7.8	7966.00581376484\\
7.9	7967.15605235543\\
8	7968.21604685386\\
8.1	7969.19357146008\\
8.2	7970.09565273367\\
8.3	7970.92865053755\\
8.4	7971.69832915014\\
8.5	7972.40991986853\\
8.6	7973.06817623122\\
8.7	7973.67742282561\\
8.8	7974.24159850754\\
8.9	7974.76429474361\\
9	7975.24878968787\\
9.1	7975.69807852064\\
9.2	7976.11490050549\\
9.3	7976.50176315935\\
9.4	7976.86096387852\\
9.5	7977.19460931846\\
9.6	7977.50463278714\\
9.7	7977.79280987815\\
9.8	7978.06077254195\\
9.9	7978.31002176839\\
10	7978.54193903292\\
};
\addlegendentry{\large{NN approximation of $f$}}

\addplot [color=blue,  line width=2.0pt]
  table[row sep=crcr]{%
-10	-247500\\
-9.9	-240099.75\\
-9.8	-232848\\
-9.7	-225743.25\\
-9.6	-218784\\
-9.5	-211968.75\\
-9.4	-205296\\
-9.3	-198764.25\\
-9.2	-192372\\
-9.1	-186117.75\\
-9	-180000\\
-8.9	-174017.25\\
-8.8	-168168\\
-8.7	-162450.75\\
-8.6	-156864\\
-8.5	-151406.25\\
-8.4	-146076\\
-8.3	-140871.75\\
-8.2	-135792\\
-8.1	-130835.25\\
-8	-126000\\
-7.9	-121284.75\\
-7.8	-116688\\
-7.7	-112208.25\\
-7.6	-107844\\
-7.5	-103593.75\\
-7.4	-99456\\
-7.3	-95429.25\\
-7.2	-91512\\
-7.1	-87702.75\\
-7	-84000\\
-6.9	-80402.25\\
-6.8	-76908\\
-6.7	-73515.75\\
-6.6	-70224\\
-6.5	-67031.25\\
-6.4	-63936\\
-6.3	-60936.75\\
-6.2	-58032\\
-6.1	-55220.25\\
-6	-52500\\
-5.9	-49869.75\\
-5.8	-47328\\
-5.7	-44873.25\\
-5.6	-42504\\
-5.5	-40218.75\\
-5.4	-38016\\
-5.3	-35894.25\\
-5.2	-33852\\
-5.1	-31887.75\\
-5	-30000\\
-4.9	-28187.25\\
-4.8	-26448\\
-4.7	-24780.75\\
-4.6	-23184\\
-4.5	-21656.25\\
-4.4	-20196\\
-4.3	-18801.75\\
-4.2	-17472\\
-4.1	-16205.25\\
-4	-15000\\
-3.9	-13854.75\\
-3.8	-12768\\
-3.7	-11738.25\\
-3.6	-10764\\
-3.5	-9843.75\\
-3.4	-8976\\
-3.3	-8159.25\\
-3.2	-7392\\
-3.1	-6672.75\\
-3	-6000\\
-2.9	-5372.25\\
-2.8	-4788\\
-2.7	-4245.75\\
-2.6	-3744\\
-2.5	-3281.25\\
-2.4	-2856\\
-2.3	-2466.75\\
-2.2	-2112\\
-2.1	-1790.25\\
-2	-1500\\
-1.9	-1239.75\\
-1.8	-1008\\
-1.7	-803.249999999999\\
-1.6	-623.999999999999\\
-1.5	-468.75\\
-1.4	-336\\
-1.3	-224.249999999999\\
-1.2	-131.999999999999\\
-1.1	-57.7499999999997\\
-1	0\\
-0.9	42.7499999999999\\
-0.799999999999999	72.0000000000003\\
-0.699999999999999	89.2500000000001\\
-0.6	96\\
-0.5	93.75\\
-0.399999999999999	83.9999999999998\\
-0.299999999999999	68.2499999999998\\
-0.199999999999999	47.9999999999998\\
-0.0999999999999996	24.7499999999999\\
0	0\\
0.0999999999999996	-24.7499999999999\\
0.199999999999999	-47.9999999999998\\
0.299999999999999	-68.2499999999998\\
0.399999999999999	-83.9999999999998\\
0.5	-93.75\\
0.6	-96\\
0.699999999999999	-89.2500000000001\\
0.799999999999999	-72.0000000000003\\
0.9	-42.7499999999999\\
1	0\\
1.1	57.7499999999997\\
1.2	131.999999999999\\
1.3	224.249999999999\\
1.4	336\\
1.5	468.75\\
1.6	623.999999999999\\
1.7	803.249999999999\\
1.8	1008\\
1.9	1239.75\\
2	1500\\
2.1	1790.25\\
2.2	2112\\
2.3	2466.75\\
2.4	2856\\
2.5	3281.25\\
2.6	3744\\
2.7	4245.75\\
2.8	4788\\
2.9	5372.25\\
3	6000\\
3.1	6672.75\\
3.2	7392\\
3.3	8159.25\\
3.4	8976\\
3.5	9843.75\\
3.6	10764\\
3.7	11738.25\\
3.8	12768\\
3.9	13854.75\\
4	15000\\
4.1	16205.25\\
4.2	17472\\
4.3	18801.75\\
4.4	20196\\
4.5	21656.25\\
4.6	23184\\
4.7	24780.75\\
4.8	26448\\
4.9	28187.25\\
5	30000\\
5.1	31887.75\\
5.2	33852\\
5.3	35894.25\\
5.4	38016\\
5.5	40218.75\\
5.6	42504\\
5.7	44873.25\\
5.8	47328\\
5.9	49869.75\\
6	52500\\
6.1	55220.25\\
6.2	58032\\
6.3	60936.75\\
6.4	63936\\
6.5	67031.25\\
6.6	70224\\
6.7	73515.75\\
6.8	76908\\
6.9	80402.25\\
7	84000\\
7.1	87702.75\\
7.2	91512\\
7.3	95429.25\\
7.4	99456\\
7.5	103593.75\\
7.6	107844\\
7.7	112208.25\\
7.8	116688\\
7.9	121284.75\\
8	126000\\
8.1	130835.25\\
8.2	135792\\
8.3	140871.75\\
8.4	146076\\
8.5	151406.25\\
8.6	156864\\
8.7	162450.75\\
8.8	168168\\
8.9	174017.25\\
9	180000\\
9.1	186117.75\\
9.2	192372\\
9.3	198764.25\\
9.4	205296\\
9.5	211968.75\\
9.6	218784\\
9.7	225743.25\\
9.8	232848\\
9.9	240099.75\\
10	247500\\
};
\addlegendentry{\large{original $f$}}

\end{axis}
\end{tikzpicture}%
}
	\end{minipage}
	\begin{minipage}[t]{0.32\textwidth}
		\centering
		\resizebox{0.95\textwidth}{!}{
%
%
\begin{tikzpicture}

\begin{axis}[%
width=3.014in,
height=3.509in,
at={(0.506in,0.474in)},
scale only axis,
xmin=-10,
xmax=10,
xtick={-10, -5, 0, 5, 10},
xticklabels={$-10$, $-5$, $0$, $5$, $10$},
xticklabel style={font=\large},  
ymin=-10000,
ymax=80000,
yticklabel style={font=\large},  
axis background/.style={fill=white},
legend style={at={(0.5,1)}, anchor=north, legend cell align=left, align=left, draw=white!15!black}
]
\addplot [color=red,  line width=2.0pt]
  table[row sep=crcr]{%
-10	1.2234451731344\\
-9.9	1.32978683666778\\
-9.8	1.44593655385553\\
-9.7	1.5728771311359\\
-9.6	1.71170092636094\\
-9.5	1.86362358545595\\
-9.4	2.02999969951514\\
-9.3	2.21234067504816\\
-9.2	2.41233515765301\\
-9.1	2.63187240506622\\
-9	2.87306907080115\\
-8.9	3.13829993612631\\
-8.8	3.43023321803465\\
-8.7	3.75187118654169\\
-8.6	4.10659694903122\\
-8.5	4.49822840594987\\
-8.4	4.93108055507105\\
-8.3	5.41003752584978\\
-8.2	5.94063596703392\\
-8.1	6.52916169696352\\
-8	7.18276186558831\\
-7.9	7.90957528071029\\
-7.8	8.71888403112685\\
-7.7	9.62129011170659\\
-7.6	10.6289214388786\\
-7.5	11.7556724625868\\
-7.4	13.0174855605993\\
-7.3	14.4326805776275\\
-7.2	16.0223412873284\\
-7.1	17.8107692618992\\
-7	19.8260176957299\\
-6.9	22.1005202254457\\
-6.8	24.6718328162352\\
-6.7	27.5835104642453\\
-6.6	30.8861449460679\\
-6.5	34.6385953129952\\
-6.4	38.9094495066522\\
-6.3	43.7787636422938\\
-6.2	49.3401355064634\\
-6.1	55.7031810578724\\
-6	62.9964976943536\\
-5.9	71.3712163260834\\
-5.8	81.0052665195991\\
-5.7	92.1085058315262\\
-5.6	104.928896581617\\
-5.5	119.759951177306\\
-5.4	136.949710672124\\
-5.3	156.911569493308\\
-5.2	180.137309276239\\
-5.1	207.212750103824\\
-5	238.836455729246\\
-4.9	275.84191783085\\
-4.8	319.223553175702\\
-4.7	370.1666091721\\
-4.6	430.080576291024\\
-4.5	500.634772132513\\
-4.4	583.793116852897\\
-4.3	681.842360043952\\
-4.2	797.403596147027\\
-4.1	933.410167230246\\
-4	1093.02544411817\\
-3.9	1279.46158746334\\
-3.8	1495.64712601029\\
-3.7	1743.68286334987\\
-3.6	2024.03477346575\\
-3.5	2334.46020025406\\
-3.4	2668.77502704694\\
-3.3	3015.75487454552\\
-3.2	3358.67968107256\\
-3.1	3676.14225550475\\
-3	3944.54163682165\\
-2.9	4142.05193713371\\
-2.8	4252.9993367121\\
-2.7	4271.05710361169\\
-2.6	4200.00206537027\\
-2.5	4051.87896474274\\
-2.4	3843.55677563165\\
-2.3	3593.10498992256\\
-2.2	3317.06442329341\\
-2.1	3028.98669469759\\
-2	2739.05855162855\\
-1.9	2454.40097508428\\
-1.8	2179.66250167018\\
-1.7	1917.6579031227\\
-1.6	1669.93021407084\\
-1.5	1437.1970634612\\
-1.4	1219.6835181472\\
-1.3	1017.35841545119\\
-1.2	830.093163097587\\
-1.1	657.759608432808\\
-1	500.280762292009\\
-0.9	357.646206032537\\
-0.799999999999999	229.902944783265\\
-0.699999999999999	117.131834202504\\
-0.6	19.4189395321652\\
-0.5	-63.1701656707436\\
-0.399999999999999	-130.607293395617\\
-0.299999999999999	-182.906682012008\\
-0.199999999999999	-220.118245236307\\
-0.0999999999999996	-242.311903789763\\
0	-249.557812219415\\
0.0999999999999996	-241.907952242638\\
0.199999999999999	-219.383845252335\\
0.299999999999999	-181.973277377438\\
0.399999999999999	-129.636435856025\\
0.5	-62.3193615000002\\
0.6	20.0292662095357\\
0.699999999999999	117.442845461844\\
0.799999999999999	229.927546226164\\
0.9	357.46199682379\\
1	500.006728897462\\
1.1	657.521516370709\\
1.2	829.986073083809\\
1.3	1017.41699442646\\
1.4	1219.87116377174\\
1.5	1437.4225851618\\
1.6	1670.09500212373\\
1.7	1917.72604714685\\
1.8	2179.73018238964\\
1.9	2454.71980412504\\
2	2739.94391796101\\
2.1	3030.52704766569\\
2.2	3318.56166074338\\
2.3	3592.24981402292\\
2.4	3835.49998893711\\
2.5	4028.58001318128\\
2.6	4150.40730806133\\
2.7	4182.57990660266\\
2.8	4114.28585318896\\
2.9	3946.25162490904\\
3	3691.76089182357\\
3.1	3373.96043253805\\
3.2	3020.52787920396\\
3.3	2657.95962558112\\
3.4	2307.48523039751\\
3.5	1983.39081330709\\
3.6	1693.35801026359\\
3.7	1439.90411066251\\
3.8	1222.09159132921\\
3.9	1037.01348349326\\
4	880.876644502672\\
4.1	749.695065665577\\
4.2	639.681715483317\\
4.3	547.436754035122\\
4.4	470.011529137903\\
4.5	404.903797939822\\
4.6	350.01928825877\\
4.7	303.62018883853\\
4.8	264.271757641083\\
4.9	230.79256155472\\
5	202.210627725472\\
5.1	177.726053093252\\
5.2	156.679760989511\\
5.3	138.527725192075\\
5.4	122.819873512762\\
5.5	109.182903904508\\
5.6	97.3063240077893\\
5.7	86.931121790572\\
5.8	77.8405713966971\\
5.9	69.8527659096307\\
6	62.8145444047408\\
6.1	56.5965441245774\\
6.2	51.0891608627876\\
6.3	46.1992431566805\\
6.4	41.8473802174255\\
6.5	37.9656711128145\\
6.6	34.4958848207577\\
6.7	31.3879384547706\\
6.8	28.5986351017142\\
6.9	26.0906140193313\\
7	23.8314749901457\\
7.1	21.7930458783438\\
7.2	19.9507682532742\\
7.3	18.2831806184106\\
7.4	16.7714825492122\\
7.5	15.3991660808625\\
7.6	14.1517031431699\\
7.7	13.0162798306313\\
7.8	11.9815699129183\\
7.9	11.0375413080459\\
8	10.1752903156339\\
8.1	9.38689928755018\\
8.2	8.66531413514426\\
8.3	8.00423866606106\\
8.4	7.39804323327256\\
8.5	6.84168558371237\\
8.6	6.33064212935761\\
8.7	5.86084814226379\\
8.8	5.4286456071348\\
8.9	5.03073765872989\\
9	4.66414869351651\\
9.1	4.32618938091111\\
9.2	4.01442591370972\\
9.3	3.72665293357152\\
9.4	3.46086964866233\\
9.5	3.21525872931144\\
9.6	2.98816762578578\\
9.7	2.77809200179777\\
9.8	2.58366101945315\\
9.9	2.40362424727502\\
10	2.23683999360404\\
};
\addlegendentry{\large{NN approximation of $f'$}}

\addplot [color=blue,  line width=2.0pt]
  table[row sep=crcr]{%
-10	74750\\
-9.9	73257.5\\
-9.8	71780\\
-9.7	70317.5\\
-9.6	68870\\
-9.5	67437.5\\
-9.4	66020\\
-9.3	64617.5\\
-9.2	63230\\
-9.1	61857.5\\
-9	60500\\
-8.9	59157.5\\
-8.8	57830\\
-8.7	56517.5\\
-8.6	55220\\
-8.5	53937.5\\
-8.4	52670\\
-8.3	51417.5\\
-8.2	50180\\
-8.1	48957.5\\
-8	47750\\
-7.9	46557.5\\
-7.8	45380\\
-7.7	44217.5\\
-7.6	43070\\
-7.5	41937.5\\
-7.4	40820\\
-7.3	39717.5\\
-7.2	38630\\
-7.1	37557.5\\
-7	36500\\
-6.9	35457.5\\
-6.8	34430\\
-6.7	33417.5\\
-6.6	32420\\
-6.5	31437.5\\
-6.4	30470\\
-6.3	29517.5\\
-6.2	28580\\
-6.1	27657.5\\
-6	26750\\
-5.9	25857.5\\
-5.8	24980\\
-5.7	24117.5\\
-5.6	23270\\
-5.5	22437.5\\
-5.4	21620\\
-5.3	20817.5\\
-5.2	20030\\
-5.1	19257.5\\
-5	18500\\
-4.9	17757.5\\
-4.8	17030\\
-4.7	16317.5\\
-4.6	15620\\
-4.5	14937.5\\
-4.4	14270\\
-4.3	13617.5\\
-4.2	12980\\
-4.1	12357.5\\
-4	11750\\
-3.9	11157.5\\
-3.8	10580\\
-3.7	10017.5\\
-3.6	9470\\
-3.5	8937.5\\
-3.4	8420\\
-3.3	7917.5\\
-3.2	7430\\
-3.1	6957.5\\
-3	6500\\
-2.9	6057.5\\
-2.8	5630\\
-2.7	5217.5\\
-2.6	4820\\
-2.5	4437.5\\
-2.4	4070\\
-2.3	3717.5\\
-2.2	3380\\
-2.1	3057.5\\
-2	2750\\
-1.9	2457.5\\
-1.8	2180\\
-1.7	1917.5\\
-1.6	1670\\
-1.5	1437.5\\
-1.4	1220\\
-1.3	1017.5\\
-1.2	829.999999999999\\
-1.1	657.5\\
-1	500\\
-0.9	357.500000000001\\
-0.799999999999999	229.999999999999\\
-0.699999999999999	117.499999999999\\
-0.6	19.9999999999997\\
-0.5	-62.5\\
-0.399999999999999	-130.000000000001\\
-0.299999999999999	-182.5\\
-0.199999999999999	-220\\
-0.0999999999999996	-242.5\\
0	-250\\
0.0999999999999996	-242.5\\
0.199999999999999	-220\\
0.299999999999999	-182.5\\
0.399999999999999	-130.000000000001\\
0.5	-62.5\\
0.6	19.9999999999997\\
0.699999999999999	117.499999999999\\
0.799999999999999	229.999999999999\\
0.9	357.500000000001\\
1	500\\
1.1	657.5\\
1.2	829.999999999999\\
1.3	1017.5\\
1.4	1220\\
1.5	1437.5\\
1.6	1670\\
1.7	1917.5\\
1.8	2180\\
1.9	2457.5\\
2	2750\\
2.1	3057.5\\
2.2	3380\\
2.3	3717.5\\
2.4	4070\\
2.5	4437.5\\
2.6	4820\\
2.7	5217.5\\
2.8	5630\\
2.9	6057.5\\
3	6500\\
3.1	6957.5\\
3.2	7430\\
3.3	7917.5\\
3.4	8420\\
3.5	8937.5\\
3.6	9470\\
3.7	10017.5\\
3.8	10580\\
3.9	11157.5\\
4	11750\\
4.1	12357.5\\
4.2	12980\\
4.3	13617.5\\
4.4	14270\\
4.5	14937.5\\
4.6	15620\\
4.7	16317.5\\
4.8	17030\\
4.9	17757.5\\
5	18500\\
5.1	19257.5\\
5.2	20030\\
5.3	20817.5\\
5.4	21620\\
5.5	22437.5\\
5.6	23270\\
5.7	24117.5\\
5.8	24980\\
5.9	25857.5\\
6	26750\\
6.1	27657.5\\
6.2	28580\\
6.3	29517.5\\
6.4	30470\\
6.5	31437.5\\
6.6	32420\\
6.7	33417.5\\
6.8	34430\\
6.9	35457.5\\
7	36500\\
7.1	37557.5\\
7.2	38630\\
7.3	39717.5\\
7.4	40820\\
7.5	41937.5\\
7.6	43070\\
7.7	44217.5\\
7.8	45380\\
7.9	46557.5\\
8	47750\\
8.1	48957.5\\
8.2	50180\\
8.3	51417.5\\
8.4	52670\\
8.5	53937.5\\
8.6	55220\\
8.7	56517.5\\
8.8	57830\\
8.9	59157.5\\
9	60500\\
9.1	61857.5\\
9.2	63230\\
9.3	64617.5\\
9.4	66020\\
9.5	67437.5\\
9.6	68870\\
9.7	70317.5\\
9.8	71780\\
9.9	73257.5\\
10	74750\\
};
\addlegendentry{\large{original $f'$}}

\end{axis}
\end{tikzpicture}%
}
	\end{minipage}

	\begin{minipage}[t]{0.32\textwidth}
		\centering
		\resizebox{0.95\textwidth}{!}{
%
%
\begin{tikzpicture}

\begin{axis}[%
width=3.746in,
height=3.09in,
at={(0.628in,0.417in)},
scale only axis,
xmin=-2,
xmax=2,
xtick={-2, -1, 0, 1, 2},
xticklabels={$-2$, $-1$, $0$, $1$, $2$},
xticklabel style={font=\large},  
ymin=-100,
ymax=600,
yticklabel style={font=\large},  
axis background/.style={fill=white},
legend style={at={(0.5,1)}, anchor=north, legend cell align=left, align=left, draw=white!15!black}
]
\addplot [color=red, mark=asterisk, mark options={solid, red}, mark size=3.5pt, line width=1.3pt]
  table[row sep=crcr]{%
-2	562.5\\
-1.9	425.75625\\
-1.8	313.6\\
-1.7	223.25625\\
-1.6	152.1\\
-1.5	97.65625\\
-1.4	57.6\\
-1.3	29.75625\\
-1.2	12.1\\
-1.1	2.75625\\
-1	0\\
-0.9	2.25625\\
-0.8	8.10000000000001\\
-0.7	16.25625\\
-0.6	25.6\\
-0.5	35.15625\\
-0.4	44.1\\
-0.3	51.75625\\
-0.2	57.6\\
-0.0999999999999999	61.25625\\
0	62.5\\
0.0999999999999999	61.25625\\
0.2	57.6\\
0.3	51.75625\\
0.4	44.1\\
0.5	35.15625\\
0.6	25.6\\
0.7	16.25625\\
0.8	8.10000000000001\\
0.9	2.25625000000001\\
1	-4.32986979603811e-15\\
1.1	2.75625\\
1.2	12.1\\
1.3	29.75625\\
1.4	57.6\\
1.5	97.65625\\
1.6	152.1\\
1.7	223.25625\\
1.8	313.6\\
1.9	425.75625\\
2	562.5\\
};
\addlegendentry{\large{NN approximation of $F$}}

\addplot [color=blue, line width=2.0pt]
  table[row sep=crcr]{%
-2	562.478947373949\\
-1.9	425.775415104476\\
-1.8	313.624234454113\\
-1.7	223.280342748477\\
-1.6	152.12496129906\\
-1.5	97.6813974673215\\
-1.4	57.6224975558196\\
-1.3	29.7730936109493\\
-1.2	12.1098234374516\\
-1.1	2.75992566427835\\
-1	-0\\
-0.9	2.25524973415309\\
-0.8	8.09936950577189\\
-0.7	16.2549953171681\\
-0.6	25.5944888851178\\
-0.5	35.140776627324\\
-0.4	44.0679911912546\\
-0.3	51.7017491376023\\
-0.2	57.5190143068611\\
-0.0999999999999999	61.147611946447\\
0	62.3655462164512\\
0.0999999999999999	61.1003141871449\\
0.2	57.4283963853071\\
0.3	51.5750443259968\\
0.4	43.9143974470664\\
0.5	34.9698698548795\\
0.6	25.414675720769\\
0.7	16.0723295574923\\
0.8	7.91697231649404\\
0.9	2.07343279259019\\
1	-0.182978144361531\\
1.1	2.57315240579511\\
1.2	11.9169546570685\\
1.3	29.573087798385\\
1.4	57.4159093225662\\
1.5	97.4700252781052\\
1.6	151.910975157978\\
1.7	223.065357489606\\
1.8	313.409010337218\\
1.9	425.560841979336\\
2	562.268488062866\\
};
\addlegendentry{\large{original $F$}}

\end{axis}
\end{tikzpicture}%
}
	\end{minipage}
	\begin{minipage}[t]{0.32\textwidth}
		\centering
		\resizebox{0.95\textwidth}{!}{
%
%
\begin{tikzpicture}

\begin{axis}[%
width=3.767in,
height=3.22in,
at={(0.628in,0.417in)},
scale only axis,
xmin=-2,
xmax=2,
xtick={-2, -1, 0, 1, 2},
xticklabels={$-2$, $-1$, $0$, $1$, $2$},
xticklabel style={font=\large},  
ymin=-1500,
ymax=1500,
yticklabel style={font=\large},  
axis background/.style={fill=white},
legend style={at={(0.5,1)}, anchor=north, legend cell align=left, align=left, draw=white!15!black}
]
\addplot [color=red, mark=asterisk, mark options={solid, red}, mark size=3.5pt, line width=1.3pt]
  table[row sep=crcr]{%
-2	-1499.211309497\\
-1.9	-1239.60490059309\\
-1.8	-1007.99845198704\\
-1.7	-803.246540800446\\
-1.6	-623.990072523217\\
-1.5	-468.760145525154\\
-1.4	-336.043065149904\\
-1.3	-224.317060218824\\
-1.2	-132.06939373355\\
-1.1	-57.8007658431139\\
-1	-0.0223923004984954\\
-0.9	42.750130579461\\
-0.8	72.0031871509584\\
-0.7	89.2297858692639\\
-0.6	95.9315185140579\\
-0.5	93.6177432638003\\
-0.4	83.8026000552745\\
-0.3	68.0009167479182\\
-0.2	47.7242139526475\\
-0.0999999999999999	24.4778645866106\\
0	-0.239929508873502\\
0.0999999999999999	-24.9372073336101\\
0.2	-48.1257505049329\\
0.3	-68.3177926943064\\
0.4	-84.0228797024711\\
0.5	-93.7457356329016\\
0.6	-95.9856794887697\\
0.7	-89.2376858493715\\
0.8	-71.9947077320634\\
0.9	-42.7504970700555\\
1	-0.00196431354152082\\
1.1	57.7498279559991\\
1.2	132.000621895359\\
1.3	224.245873204804\\
1.4	335.984770976462\\
1.5	468.723394661255\\
1.6	623.973514780597\\
1.7	803.241467500811\\
1.8	1007.9986938879\\
1.9	1239.62184065409\\
2	1499.28614274326\\
};
\addlegendentry{\large{NN approximation of $f$}}

\addplot [color=blue,  line width=2.0pt]
  table[row sep=crcr]{%
-2	-1500\\
-1.9	-1239.75\\
-1.8	-1008\\
-1.7	-803.25\\
-1.6	-624\\
-1.5	-468.75\\
-1.4	-336\\
-1.3	-224.25\\
-1.2	-132\\
-1.1	-57.7500000000001\\
-1	0\\
-0.9	42.75\\
-0.8	72\\
-0.7	89.25\\
-0.6	96\\
-0.5	93.75\\
-0.4	84\\
-0.3	68.25\\
-0.2	48\\
-0.0999999999999999	24.75\\
0	0\\
0.0999999999999999	-24.75\\
0.2	-48\\
0.3	-68.25\\
0.4	-84\\
0.5	-93.75\\
0.6	-96\\
0.7	-89.25\\
0.8	-72\\
0.9	-42.75\\
1	0\\
1.1	57.7500000000001\\
1.2	132\\
1.3	224.25\\
1.4	336\\
1.5	468.75\\
1.6	624\\
1.7	803.25\\
1.8	1008\\
1.9	1239.75\\
2	1500\\
};
\addlegendentry{\large{original $f$}}

\end{axis}
\end{tikzpicture}%
}
	\end{minipage}
	\begin{minipage}[t]{0.32\textwidth}
		\centering
		\resizebox{0.95\textwidth}{!}{
%
%
\begin{tikzpicture}

\begin{axis}[%
width=3.767in,
height=3.09in,
at={(0.632in,0.417in)},
scale only axis,
xmin=-2,
xmax=2,
xtick={-2, -1, 0, 1, 2},
xticklabels={$-2$, $-1$, $0$, $1$, $2$},
xticklabel style={font=\large},  
ymin=-500,
ymax=3000,
yticklabel style={font=\large},  
axis background/.style={fill=white},
legend style={at={(0.5,1)}, anchor=north, legend cell align=left, align=left, draw=white!15!black}
]
\addplot [color=red, mark=asterisk, mark options={solid, red}, mark size=3.5pt, line width=1.3pt]
  table[row sep=crcr]{%
-2	2739.05855162855\\
-1.9	2454.40097508428\\
-1.8	2179.66250167019\\
-1.7	1917.6579031227\\
-1.6	1669.93021407084\\
-1.5	1437.1970634612\\
-1.4	1219.6835181472\\
-1.3	1017.35841545119\\
-1.2	830.093163097588\\
-1.1	657.759608432809\\
-1	500.280762292009\\
-0.9	357.646206032536\\
-0.8	229.902944783266\\
-0.7	117.131834202504\\
-0.6	19.4189395321655\\
-0.5	-63.1701656707436\\
-0.4	-130.607293395616\\
-0.3	-182.906682012007\\
-0.2	-220.118245236307\\
-0.0999999999999999	-242.311903789764\\
0	-249.557812219415\\
0.0999999999999999	-241.907952242638\\
0.2	-219.383845252335\\
0.3	-181.973277377437\\
0.4	-129.636435856023\\
0.5	-62.3193615000002\\
0.6	20.0292662095356\\
0.7	117.442845461845\\
0.8	229.927546226165\\
0.9	357.461996823789\\
1	500.006728897462\\
1.1	657.521516370709\\
1.2	829.98607308381\\
1.3	1017.41699442646\\
1.4	1219.87116377174\\
1.5	1437.4225851618\\
1.6	1670.09500212374\\
1.7	1917.72604714685\\
1.8	2179.73018238964\\
1.9	2454.71980412504\\
2	2739.94391796101\\
};
\addlegendentry{\large{NN approximation of $f'$}}

\addplot [color=blue,  line width=2.0pt]
  table[row sep=crcr]{%
-2	2750\\
-1.9	2457.5\\
-1.8	2180\\
-1.7	1917.5\\
-1.6	1670\\
-1.5	1437.5\\
-1.4	1220\\
-1.3	1017.5\\
-1.2	830\\
-1.1	657.5\\
-1	500\\
-0.9	357.5\\
-0.8	230\\
-0.7	117.5\\
-0.6	19.9999999999999\\
-0.5	-62.5\\
-0.4	-130\\
-0.3	-182.5\\
-0.2	-220\\
-0.0999999999999999	-242.5\\
0	-250\\
0.0999999999999999	-242.5\\
0.2	-220\\
0.3	-182.5\\
0.4	-130\\
0.5	-62.5\\
0.6	19.9999999999999\\
0.7	117.5\\
0.8	230\\
0.9	357.5\\
1	500\\
1.1	657.5\\
1.2	830\\
1.3	1017.5\\
1.4	1220\\
1.5	1437.5\\
1.6	1670\\
1.7	1917.5\\
1.8	2180\\
1.9	2457.5\\
2	2750\\
};
\addlegendentry{\large{original $f'$}}

\end{axis}
\end{tikzpicture}%
}
	\end{minipage}

	\caption{Functions $F$,  $f$ and its first order derivative $f'$ along with the corresponding approximations learned from a neural network. We note that the range of the learning-informed function is influenced by the training data. The second row of images shows that the functions are well-approximated by their neural network counterparts in the ranges where the training data cover well, which here is around the interval $[-2,2].$}
	\label{fig:non_monotone_f}
\end{figure}
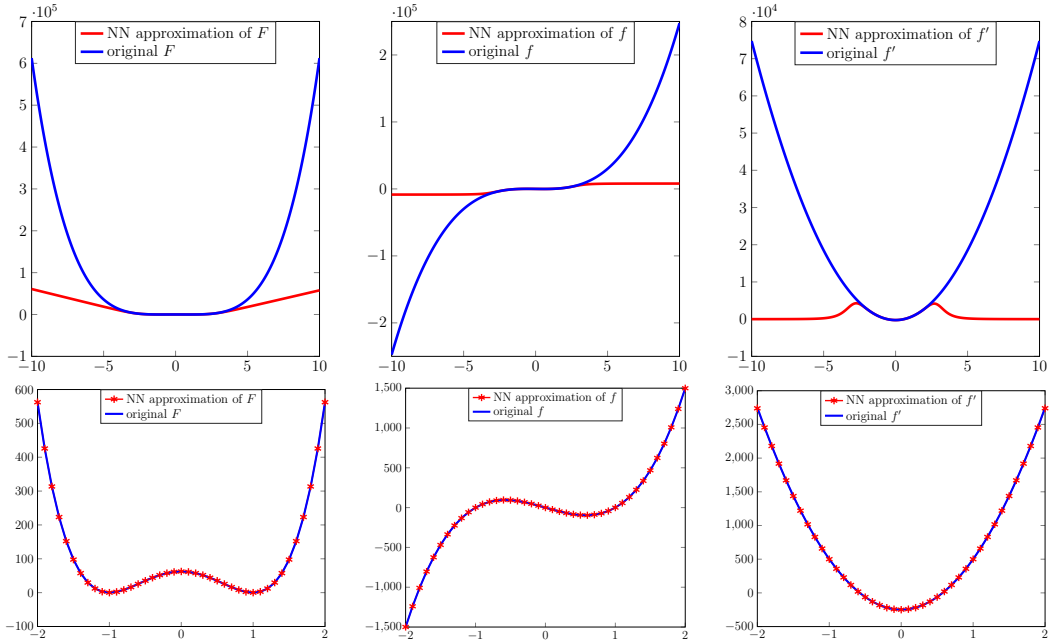
In Figure \ref{fig:non_monotone_f}, we provide the plots of  $F(z)=\int_{-1}^{z} f(t)\,dt$, the function $f$ and  its derivative $f^\prime$ on $[-K,K]\subset \R$, ($K=10$ and $K=2$, respectively) as well as their learned counterparts.

We observe that all the learning-informed versions preserve the key features of their exact counterparts very well. This is due to the fact that the training data cover exactly those ranges where important features are located.

As a next step, we consider the corresponding optimal control problem when the function $f$ is replaced by its learned version.
Notice that both the original and the learning-informed PDE admit no unique solution. Therefore the initial guess for the Newton iteration is crucial for the convergence to the final solutions.
The algorithm for solving the optimal control problem for both PDEs is a combination of the semi-smooth Newton algorithm for \eqref{eq:stationary1} (with  $0$ as the initial guess) and the SQP algorithm.
The switch between the solvers operates as follows: Consider the summed up residual of
the four equations in  \eqref{eq:stationary1} with respect to their norms in the spaces $H^{-1}(\Omega)$, $H^{-1}(\Omega)$, $L^{2}(\Omega)$ and $L^{2}(\Omega)$, respectively. Then
we start our algorithm by calling the semi-smooth Newton iterations, and when the residual drops below a threshold value (e.g., $5$ in our tests), then we switch to the SQP algorithm. The iteration is stopped if the residual is smaller than $10^{-10}$, or a maximum of $30$ iterations is reached.
We fix $\alpha= 10^{-5}$ and 
$\mathcal{C}_{ad}:=\set{u: -50\leq u \leq 50}$. Next consider $g$ to be some polarized data  preferring the values  $-1$ and $1$ and representing two distinct material states, e.g., a binary alloy; see Figure 
\ref{fig:Neumann_Allen_Cahn_optimal_control}.

\begin{figure}[h!]
	\begin{minipage}[t]{0.48\textwidth}
		\centering
		\resizebox{0.95\textwidth}{!}{
%
%
\begin{tikzpicture}

\begin{axis}[%
width=6.706in,
height=3.283in,
at={(1.125in,0.443in)},
scale only axis,
xmin=1,
xmax=10,
xtick={0, 2, 4, 6, 8, 10},
xticklabels={$0$, $2$, $4$, $6$, $8$, $10$},
xticklabel style={font=\large},  
ymode=log,
ymin=0.183895834150301,
ymax=0.976965617748144,
ytick={0.2, 0.3, 0.4, 0.5, 0.6, 0.7, 0.8, 0.9},
yticklabels={$0.2$, $0.3$, $0.4$, $0.5$, $0.6$, $0.7$, $0.8$, $0.9$},
yticklabel style={font=\large},  
yminorticks=true,
axis background/.style={fill=white},
title style={font=\bfseries},
title={\Large{Merit function value at iterations}},
legend style={at={(0.7,0.84)}, anchor=south west, legend cell align=left, align=left, draw=white!15!black}
]
\addplot [color=red, mark=asterisk, mark options={solid, red}, mark size=3.5pt, line width=1.3pt]
  table[row sep=crcr]{%
1	0.733040927951206\\
2	0.653087182155955\\
3	0.972495183495124\\
4	0.187814853440284\\
5	0.187913925448807\\
6	0.187095358472964\\
7	0.187172055812577\\
8	0.18642533521474\\
9	0.184251175090561\\
10	0.183941527969864\\
11	0\\
12	0\\
13	0\\
14	0\\
15	0\\
16	0\\
};
\addlegendentry{Learning-informed PDE}

\addplot [color=blue, line width=2.0pt]
  table[row sep=crcr]{%
1	0.734254067611782\\
2	0.656049916893515\\
3	0.976965617748144\\
4	0.187788043030654\\
5	0.187866357638422\\
6	0.187051649180194\\
7	0.187118694233224\\
8	0.186375515557492\\
9	0.184174051273507\\
10	0.183895834150301\\
11	0\\
12	0\\
13	0\\
14	0\\
15	0\\
16	0\\
};
\addlegendentry{Original PDE}

\end{axis}
\end{tikzpicture}%
}
	\end{minipage}
	\begin{minipage}[t]{0.48\textwidth}
		\centering
		\resizebox{0.95\textwidth}{!}{
%
%
\begin{tikzpicture}

\begin{axis}[%
width=6.706in,
height=3.283in,
at={(1.125in,0.443in)},
scale only axis,
xmin=1,
xmax=10,
xtick={0, 2, 4, 6, 8, 10},
xticklabels={$0$, $2$, $4$, $6$, $8$, $10$},
xticklabel style={font=\large},  
ymode=log,
ymin=1e-15,
ymax=100000,
ytick={1e-15, 1e-12, 1e-9, 1e-6, 1e-3, 1e-0, 1e+3},
yticklabels={$10^{-15}$, $10^{-12}$, $10^{-9}$, $10^{-6}$, $10^{-3}$, $10^0$, $10^3$},
yticklabel style={font=\large},  
yminorticks=true,
axis background/.style={fill=white},
title style={font=\bfseries},
title={\Large{The norm summation of the residual of all the equations}},
legend style={at={(0.7,0.84)}, anchor=south west, legend cell align=left, align=left, draw=white!15!black}
]
\addplot [color=red, mark=asterisk, mark options={solid, red}, mark size=3.5pt, line width=1.3pt]
  table[row sep=crcr]{%
1	366.753602150525\\
2	108.91515692657\\
3	168.492086793302\\
4	74.9764867007364\\
5	18.3629521011138\\
6	7.04741115665682\\
7	6.15109849280729\\
8	4.19694429011999\\
9	0.254870609301803\\
10	7.64393360966758e-14\\
11	0\\
12	0\\
13	0\\
14	0\\
15	0\\
16	0\\
};
\addlegendentry{Learning-informed PDE}

\addplot [color=blue,  line width=2.0pt]
  table[row sep=crcr]{%
1	315.936084140842\\
2	164.845092461874\\
3	150.903197077436\\
4	22.9950755697483\\
5	10.750394152272\\
6	6.72431029976343\\
7	5.97336394692387\\
8	4.06028344510164\\
9	0.232143500811591\\
10	3.36506633661072e-14\\
11	0\\
12	0\\
13	0\\
14	0\\
15	0\\
16	0\\
};
\addlegendentry{Original PDE}

\end{axis}
\end{tikzpicture}%
}
	\end{minipage}	
	\caption{Merit function (left) and residual norm (right).}
	\label{fig:Neumann_merit_function}
\end{figure}
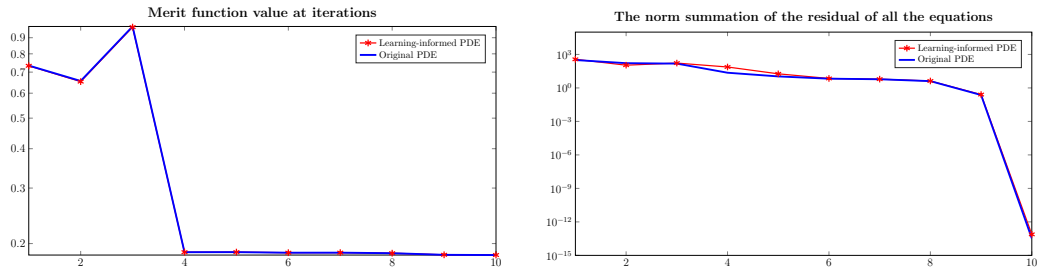
In Figure \ref{fig:Neumann_merit_function} 
we show the plots of the merit function values and also the overall residual of the first-order system in \eqref{eq:stationary1}.
The increasing part in the first few steps in the left plot (merit function) is due to the initilization of SSN while full step length is accepted. We notice that the threshold is reached by $10$ overall iterations including also the SSN initialization steps.

\begin{figure}[h!]
	\includegraphics[width=0.32\textwidth]{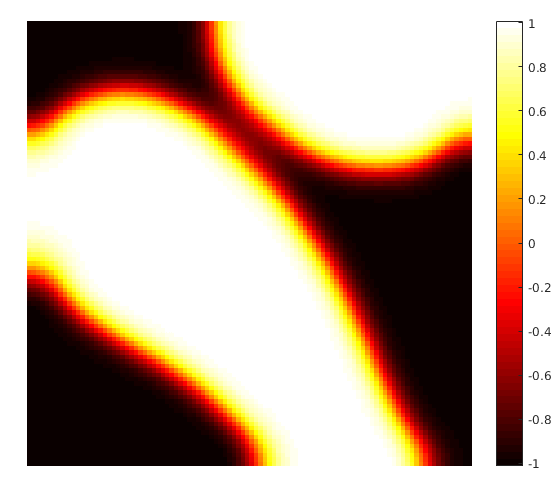}
	\includegraphics[width=0.32\textwidth]{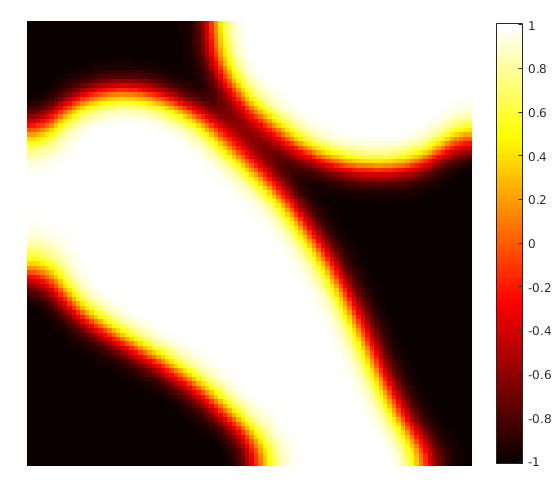}
	\includegraphics[width=0.32\textwidth]{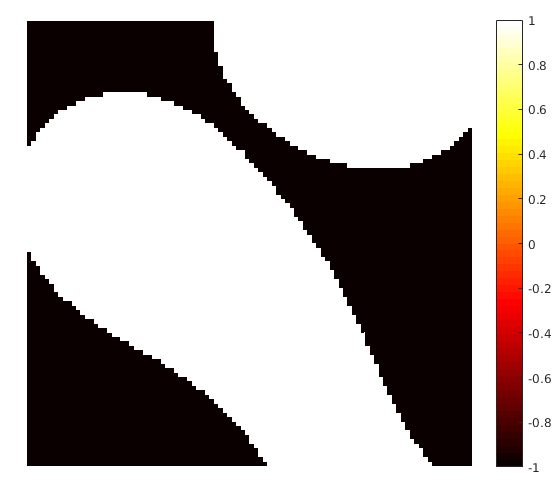}\\	\includegraphics[width=0.32\textwidth]{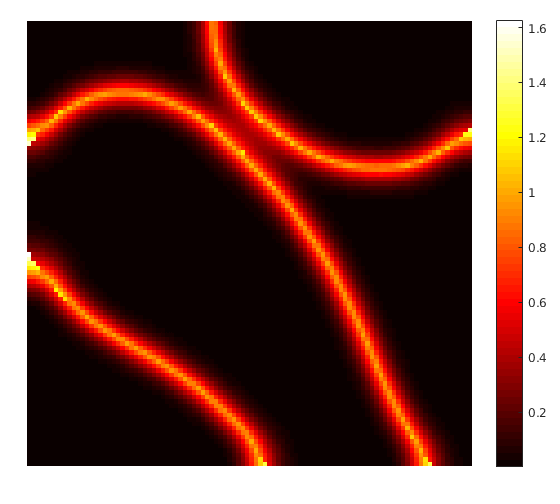}
	\includegraphics[width=0.32\textwidth]{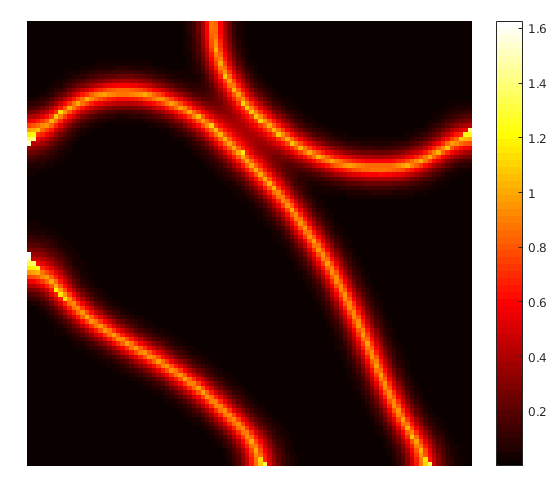}
	\includegraphics[width=0.32\textwidth]{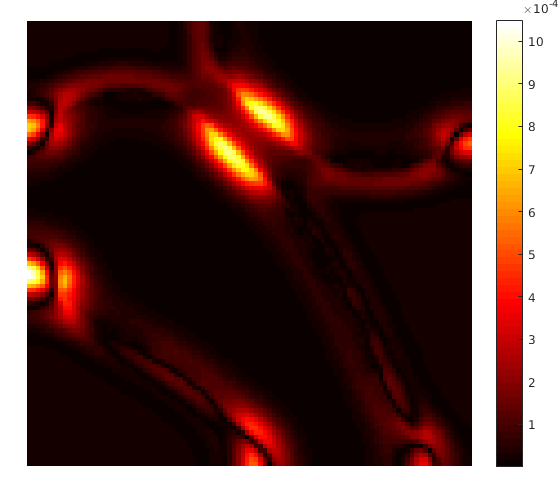}\\
	\includegraphics[width=0.32\textwidth]{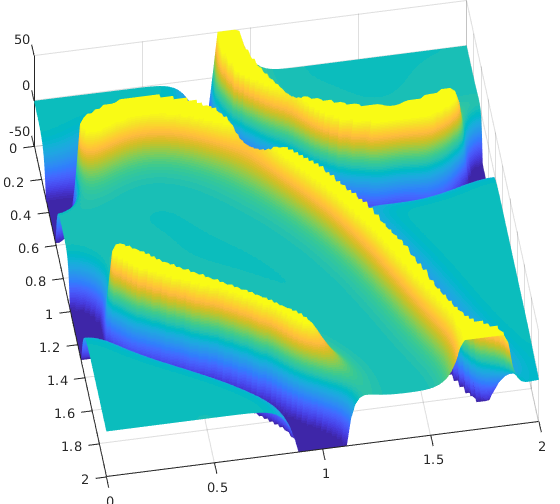}
	\includegraphics[width=0.32\textwidth]{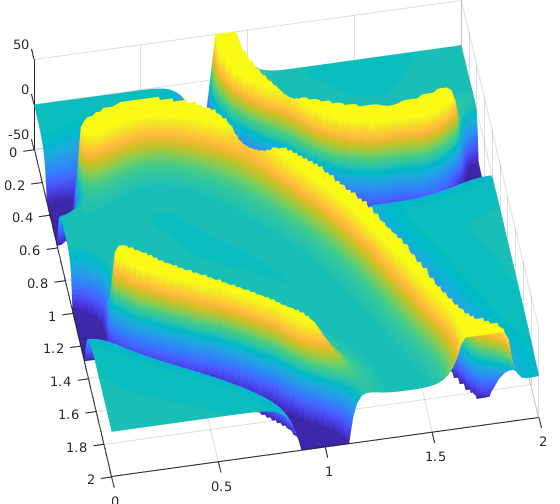}
	\includegraphics[width=0.32\textwidth]{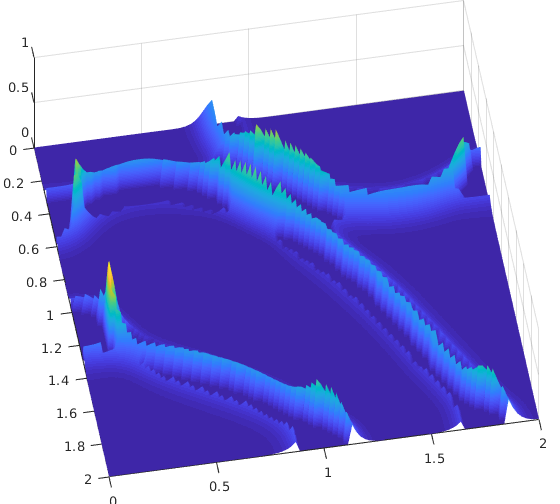}
	\caption{Optimal control of the stationary Allen-Cahn equation. First row: states (right: target data $g$; left and the middle: optimal states of learning-informed and original PDE, respectively);
		second row: difference images of states (left and the middle: differences (in absolute values) of optimal states to target state $g$; right: actual difference between the two optimal states $\abs{y_{\mathcal{N}}-\bar{y}}$ in the first row; third row: left and middle the optimal controls corresponding to the learning-informed and original PDE respectively, as well as their difference $\abs{u_{\mathcal{N}}-\bar{u}}$ on the right.}
	\label{fig:Neumann_Allen_Cahn_optimal_control}
\end{figure}

Since neither the optimal control problem nor the PDE  admit unique solutions, many local minima make the semi-smooth Newton algorithm rather sensitive to the initial guess. 
Concerning SQP we note here that enforcing the PDE and the box constraints too strongly in the early iterations, might result to the SQP algorithm getting trapped at some unfavorable stationary point. This has been numerically observed, e.g., when initializing the SQP algorithm by zero.
In our tests, the combination of the semi-smooth Newton algorithm with the SQP algorithm, however, turns out to be robust against the aforementioned adverse effects.
From  Figure \ref{fig:Neumann_Allen_Cahn_optimal_control} (right plot) 
we  observe a high accuracy approximation of the solutions of the learning-informed control to the solutions of the original control problem. Both, the PDE constraint and also the box constraint are satisfied with high accuracy.

\section{Application: Quantitative magnetic resonance imaging (qMRI)}
\label{sec:appl_2}
According to \cite{DonHinPap19}, we consider the following optimization task in qMRI:
\begin{equation}
	\label{eq:qMRI_optimal_control}
	\begin{aligned}
		&\text{minimize}\quad   \frac{1}{2}\norm{P\mathcal{F}(y)-g^\delta}_{H}^{2} + \frac{\alpha}{2}\norm{u}^2_{U},\quad\text{over }(y,u:=(T_1,T_2,\rho) ^\top)\in Y\times U, \\  
		&\text{s.t.}\quad \frac{\partial y}{\partial t}(t) =  y(t) \times \gamma B(t) - \left ( \frac{y_{1}(t)}{T_{2}}, \frac{y_{2}(t)}{T_{2}}, \frac{y_{3}(t)- \rho m_{e}}{T_{1}} \right ), \quad t=t_{1},\ldots, t_{L},\\
		&\phantom{\text{s.t.}}\quad y(0)= \rho m_{0},\\
		& \phantom{\text{s.t.}}\quad u\in \mathcal{C}_{ad}. 
	\end{aligned}
\end{equation}
where $0< t_1<\ldots<t_L$, $L\in\mathbb{N}$, $u\in U:=[H^1(\Omega)]^3$ and $Y:=[L^2(\Omega)^{3}]^{L}$ with $\Omega\subset\mathbb{R}^2$ the image domain, $H=\left [L^2(\mathbb{K})^{2}\right]^{L}$ with $\mathbb{K}$ the Fourier space. By $\mathcal{F}: Y\to H$ we denote the component-wise Fourier transform acting on
 $(y_{1}, y_{2})$, i.e., the first two coordinates of $y$, and $P:H\to H$ is a subsampling operator.

Further, $g^\delta=(g_{l}^\delta)_{l=1}^L\in H$ are (noisy) data, and $\mathcal{C}_{ad}$ is an nonempty, closed, convex, and bounded subset of 
$ [L_{\epsilon}^\infty(\Omega)^{+}]^{3}$ with $L_{\epsilon}^\infty(\Omega)^{+}:=\{f\in L^{\infty}(\om):\; \mathrm{ess}\,\mathrm{inf} f>\epsilon\}, $ for some $\epsilon>0$, 
which takes care of practical properties of physical quantities. 
The system of ordinary differential equations in \eqref{eq:qMRI_optimal_control} with initial value $\rho m_0$ represents the renowned Bloch equations (BE), which  model the evolution of nuclear magnetization in MRI \cite{Blo46} with the parameters $\gamma$ and $m_e$ being fixed constants. In our context, the external magnetic field $B$ is assumed to be a uniformly bounded function in time.
To accommodate different scaling, we consider $\frac{\mathbf{\alpha}}{2}\norm{u}^2_{U }:=\frac{\alpha_0}{2}\norm{u}^2_{[L^2(\Omega)]^3}+\frac{1}{2}\abs{u}^2_{[H^1(\Omega)]^3},$ and
\[  \abs{u}^2_{[H^1(\Omega)]^3 }:=\int_\Omega  \left(\alpha_{1,1} \abs{\nabla T_1}^2 +\alpha_{1,2}\abs{\nabla T_2}^2 + \alpha_{1,3}\abs{\nabla \rho}^2\right)dx,  \]
with $\alpha_0>0$ and $\alpha_{1,j}>0$ for $j=1,2,3$.
For the ease of presentation, below we omit these scaling parameters.
\begin{remark}\label{rem:Bloch_bounds}
One readily checks that the solutions to the BE are bounded uniformly  as long as $T_1,T_2$ are positive values and the magnetic field $B(t)$ is bounded. This property persists if either of the two terms on the right hand side of the equation is missing.
\end{remark}

Fixing the external magnetic field $B$ according to an excitation protocol with a specific sequence of frequency pulses (cf., e.g., \cite{DonHinPap19}) and associated echo times $\{t_i\}_{i=1}^L$ we have $u\mapsto \{y(t_i)\}_{i=1}^L$ yielding the solution map $\Pi:\mathcal{C}_{ad}\to  [(L^{\infty}(\Omega))^3]^L$. Using this notation we have $Q(\cdot)=P\mathcal{F}(\Pi(\cdot))$. Noting that $\Pi(T_{1}, T_{2}, \rho)=\rho \Pi(T_{1}, T_{2},1)$ we show first continuity and differentiability results for $\tilde{\Pi}(\theta):=\Pi(T_{1}, T_{2},1)$ where $\theta:=(T_1,T_2)^\top$. Even though for simplicity we do that for $\theta\in  [L_{\epsilon}^{\infty}(\Omega)^{+}]^{2}$, with $\epsilon>0$, we note  that the map $\tilde{\Pi}$ can be continuously extended also for $T_{1}=0$ and/or $T_{2}=0$.

\begin{proposition}\label{prop:continuity_Bloch}
	The operator $\tilde{\Pi}: [ L_{\epsilon}^{\infty}(\Omega)^{+}]^2\to  [(L^{\infty}(\Omega))^3]^L$ is locally Lipschitz continuous, and Fr\'echet differentiable with locally Lipschitz  derivative.
\end{proposition}
\begin{proof}
	Let $\theta,\theta^a\in [L_{\epsilon}^\infty(\Omega)^+]^2$ be given with associated solutions $y, y^a$ of the BE, respectively. Suppressing $x\in\Omega$ in our notation, subtracting the BE for both $\theta$ values, and letting $r^a:=y-y^a$ as well as $R(\theta) :=\operatorname{diag}(\frac{1}{T_2},\frac{1}{T_2},\frac{1}{T_1})$, we get
	\begin{equation}\label{eq:diff_Bloch}
		\frac{\partial r^a}{\partial t}(t) -  r^a(t) \times \gamma B(t) +R(\theta) r^a =\left(R(\theta^a)-R(\theta)\right) (y^a(t)-(0,0,y_e))^\top ,\;
		r^a(0) =  0.
	\end{equation}
	This equation and its homogeneous counterpart (i.e., with zero right hand side) admit unique solutions, respectively, cf. \cite{Tes12}, for instance. 
	According to \cite[Theorem 3.12]{Tes12} the solution to \eqref{eq:diff_Bloch} is
	\begin{equation}\label{eq:solution_diff_Bloch}
		r^a(t)=  \int_0^t \Phi(t,s) \left(R(\theta^a)-R(\theta)\right) (y^a(s)-(0,0,y_e)^\top )ds,
	\end{equation}
	where $ \Phi(t,s) $ is the principal matrix consisting of the three independent  solutions of the homogeneous counterpart of \eqref{eq:diff_Bloch}
	resulting from the initial data $h(s)=e_{i}$, $i=1,2,3$, with $\{e_1,e_2,e_3\}$ the canonical orthonormal basis in $\R^{3}$. Note that it is easy to check that any such solution is uniformly bounded both in $t\ge 0$ and $\theta\ge0$ almost everywhere.
	Since $R(\cdot)$ restricted to $[\epsilon,\infty)$ is  Lipschitz (modulus $L>0$), 
	\eqref{eq:solution_diff_Bloch} can be further estimated as follows
	\[ \abs{r^a(t)}\leq  L\int_0^t |\Phi(t,s) (y^a(s)-(0,0,y_e)^\top )|ds \abs{\theta^a-\theta} \leq \tilde{L}(t) \abs{\theta^a-\theta},\]
	for all $\theta^a,\theta \in  [L_{\epsilon}^\infty(\Omega)^+]^2$. Note that the above estimate and in particular $\tilde{L}(t)$ can be considered independent of the spatial variable $x$ due to the uniform bound on the solution of BE for every element of $\mathcal{C}_{ad}$ (cf. Remark \ref{rem:Bloch_bounds}). 
	Therefore we have for some $L_{\Pi}>0$ that
	\begin{equation*}
		\|y^{a}(\cdot,t)-y(\cdot, t)\|_{[L^{q}(\om)]^{3}} \leq L_{\Pi}  \| \theta^a-\theta\|_{[L^{q}(\om)]^{2}} \; \text{ for all } 1\le q\le \infty.  
	\end{equation*}
	By considering the above estimate at $\{t_{i}\}_{i=1}^L$ we get  the asserted local Lipschitz continuity of $\tilde{\Pi}$.

	We now proceed to Fr\'echet differentiability. 
	Let $\theta\in  [L_{\epsilon}^\infty(\Omega)^+]^2$, $v\in [L^\infty(\Omega)]^2$ be an arbitrary  vector, and let $\theta^a= \theta+a v$ where $a >0$ is such that $\theta^a\in [L_{\epsilon}^\infty(\Omega)^{+}]^2$.
	Dividing \eqref{eq:diff_Bloch} by $a$ and letting $p_\theta^a:=\frac{r^a}{a}$,  we get:
	\begin{equation}\label{eq:adjoint_Bloch}
		\frac{\partial p_\theta^a}{\partial t}(t) -  p_\theta^a(t) \times \gamma B(t) +R(\theta) p_\theta^a =\frac{\left(R(\theta^a)-R(\theta)\right)}{a} (y^a(t)-(0,0,y_e))^\top,\;\;
		p_\theta^a(0) =  0.
	\end{equation}
	Existence, uniqueness and representation of a solution again follows from \cite[Theorem 3.12]{Tes12}:
	\[ p_\theta^a(t)=  \int_0^t \Phi(t,s) \frac{\left(R(\theta+ a v)-R(\theta)\right)}{a} (y^a(s)-(0,0,y_e)^\top )ds.\]
	Recall that $R(\cdot)$ is continuously differentiable for $\theta >0$ and time independent. For $a\downarrow 0$ and $p_\theta:=\lim_{a  \to 0} p_\theta^a$, we have
	\[ p_\theta(t)= \int_0^t \Phi(t,s)  R'(\theta;v) (y(s)-(0,0,y_e)^\top )ds ,  \]
	where $R'(\theta;v)$ denotes the directional derivative of $R$ at $\theta$ in direction $v$. 
	By considering again the uniform boundedness with respect to the spatial variable and pointwise evaluation at $\{t_{i}\}_{i=1}^L$, we get that $p_\theta=\tilde{\Pi}'(\theta;v) $ is bounded, and  also linear with respect to the direction $v\in [L^\infty(\Omega)]^2$.
	Thus, $\tilde{\Pi}$ is Gateaux differentiable.
	Notice further that, due to $R'(\cdot;v) $ being locally Lipschitz, we have also the local Lipschitz continuity (modulus $L_{p_\theta}>0$) of the directional derivative:
	\begin{equation}\label{eq:p_norm_estimate}
		\abs{p_{\theta^a}-p_\theta}^q \leq L^q_{p_\theta} \abs{\theta^a-\theta}^q\|v\|_{[L^\infty(\Omega)]^2}   \;\text{ for  all }\; \theta^a, \theta  \in  [L_{\epsilon}^\infty(\Omega)^+]^2, \text{ and  }  1 \leq q \leq \infty, 
	\end{equation}
	with the above estimate again independent of the spatial variable.
	This together with the linearity of the Gateaux derivative implies the  Fr\'echet differentiability of $\tilde{\Pi}$.
	Finally we also conclude the Lipschitz continuity of the Fr\'echet derivative:
	\begin{equation} \label{eq:B_prime_Lip}
		\norm{(\tilde{\Pi}^\prime(\theta^a) -\tilde{\Pi}^\prime(\theta))v}_{[L^{\infty}(\om)]^{3L}} \leq L_{p_\theta} \norm{\theta^a -\theta}_{[L^{\infty}(\om)]^{2}}\norm{v}_{[L^{\infty}(\om)]^{2}}.
	\end{equation}
This ends the proof.
\end{proof}
Note that the continuity and differentiability of $\Pi=\rho\tilde{\Pi}$ for $u\in \mathcal{C}_{ad}$ follows readily as $\rho\in L^\infty(\Omega)$. As a consequence, existence of a solution to \eqref{eq:qMRI_optimal_control} can be shown similarly to Proposition \ref{pro:existence_wsc}.

\begin{remark}
	The estimate \eqref{eq:p_norm_estimate} indicates that for every $u=(\theta^\top,\rho)^\top\in \mathcal{C}_{ad}$, and $h \in [L^\infty(\Omega)]^2$ sufficiently small, we even have
	\[ \norm{\tilde{\Pi}(\theta+h)-\tilde{\Pi}(\theta)- \tilde{\Pi}^\prime(\theta)h}_{[L^q(\Omega)]^{3L}}= \mathcal{O}(\norm{h}^2_{[L^q(\Omega)]^2} ) \quad \text{ for all } 1\leq q\leq \infty. \]
	We also note that due to properties of the Bloch operator, we have that both $\tilde{\Pi}^\prime(\theta): [L^2(\Omega)]^2\to [L^2(\Omega)]^{3L}$ and $Q^\prime(u): [L^2(\Omega)]^3 \to  [(L^2(\mathbb{K}))^{2}]^{L}$ are bounded linear operators, respectively, as soon as $u=(\theta^\top,\rho)^\top\in \mathcal{C}_{ad}$. In this sense, we consider in the following $\tilde{\Pi}^\prime(\theta)$ and $Q^\prime(u)$ to be elements in $\mathcal{L}([L^2(\Omega)]^2,Y)$ and $\mathcal{L}(U,H)$, respectively.
\end{remark}
We are now interested in finding a data-driven approximation $\Pi_{\mathcal{N}}(u) := \rho\mathcal{N}(T_1,T_2)$ of $\Pi$ and in solving the reduced problem
\begin{equation}
	\label{eq:qMRI_nn}
	\begin{aligned}
		&\text{minimize}\quad   \frac{1}{2}\norm{Q_{\mathcal{N}}(u)-g^\delta}_{H}^{2} + \frac{\alpha}{2}\norm{u}^2_{U},\quad\text{over }u\in U,\\
		&\text{s.t.  } \quad  u=(T_{1}, T_{2}, \rho) ^\top\in \mathcal{C}_{ad},
	\end{aligned}
\end{equation}
with $Q_\mathcal{N}(u)=P\mathcal{F}(\Pi_{\mathcal{N}}(T_1,T_2,\rho))$. Existence of a solution to \eqref{eq:qMRI_nn} can again be argued similarly to Proposition \ref{pro:existence_wsc}.

We finish this section with the corresponding approximation result.
\begin{proposition}
	Let $\theta=(T_1,T_2)^\top$, $u=(\theta^\top,\rho)^\top \in  \mathcal{C}_{ad}$. Assume the following error bounds in the neural network approximations
	\[\norm{\mathcal{N}(\theta)-\tilde{\Pi}(\theta)}_{[L^{\infty}(\Omega)^{3}]^L}\leq \epsilon \quad  \text{ 
		and } \quad \norm{\mathcal{N}^\prime(\theta)-\tilde{\Pi}^\prime(\theta)}_{\mathcal{L}([L^2(\Omega)]^2,[L^{\infty}(\Omega)^{3}]^L)}\leq \epsilon_1, \]
	Then we have
	\begin{align}
		\norm{Q(u)-Q_{\mathcal{N}}(u)}_H&\leq C\epsilon,\\
		\norm{Q^\prime (u)-Q^\prime_{\mathcal{N}}(u)}_{\mathcal{L}(U,H)}&\leq   C_1 \epsilon+C_2 \epsilon_{1},
	\end{align}
	for some positive constants $C$, $C_1$ and $C_2$ which are all independent of $\epsilon$ and $\epsilon_{1}$.
\end{proposition}
Before we commence with the proof, note that the above assumptions are plausible in view of $u\in \mathcal{C}_{ad}\subset [(L_\epsilon^\infty(\Omega))^+]^3$ and Theorems \ref{thm:function_app} and \ref{thm:deriv_app}.
\begin{proof}
	The first  estimate is straightforward from the definition of $Q$ 
	\begin{equation}
		\norm{Q(u)-Q_{\mathcal{N}}(u)}_H =\norm{P\mathcal{F}(\rho (\mathcal{N}(\theta)-\tilde{\Pi}(\theta)))}_H\leq \norm{\rho (\mathcal{N}(\theta)-\tilde{\Pi}(\theta))}_{[L^2(\Omega)^{3}]^L} \leq C\epsilon,
	\end{equation}
	since $\mathcal{C}_{ad}\subset [L^\infty(\Omega)]^3$ is a bounded set.
	
	To see the second estimate, notice that for every $v:=(v_{1}, v_{2}, v_{3})^\top \in [L^2(\Omega)]^3$,
	\begin{equation}\label{eq:mri_derivative}
		Q^\prime(u)v=P\mathcal{F}(v_1 \tilde{\Pi}(\theta)) + P\mathcal{F}(\rho \tilde{\Pi}^\prime(\theta) (v_2,v_3)^\top), 
	\end{equation}
	and similarly for $Q_{\mathcal{N}}'$. Thus,
	\begin{align*}
		\norm{(Q^\prime (u)-Q^\prime_{\mathcal{N}}(u))v}_H
		\leq & C_1\norm{\mathcal{N}(\theta)-\tilde{\Pi}(\theta)}_{ [L^{\infty}(\om)^{3}]^{L}} \|v_{1}\|_{L^{2}(\om)}\\ & +C_2\norm{\mathcal{N}^\prime(\theta)-\tilde{\Pi}^\prime(\theta)}_{\mathcal{L}([L^2(\Omega)]^2,[L^{\infty}(\om)^{3}]^{L})} \|(v_{2}, v_{3})\|_{[L^2(\Omega)]^2},
	\end{align*}
	which ends the proof.
\end{proof}
Finally, we show the Lipschitz continuity of $Q$ and $Q'$. For the learning-informed versions this is done similarly. Using the isometric property of the Fourier transform and the triangle inequality, we get for every $u_a,u_b\in \mathcal{C}_{ad}$ and some $C\geq 1$:
\[\norm{Q(u_a) -Q(u_b)}_H
\leq C\left(\norm{\rho_a-\rho_b}_{L^2(\Omega)}+ \norm{\theta_a-\theta_b}_{[L^2(\Omega)]^2}\right). \]

Similarly, we estimate $\norm{(Q^\prime(u_a)- Q^\prime(u_b))v}_H$ assuming that $v$ is unitary:
\[
\begin{aligned}
& \norm{(Q^\prime(u_a)- Q^\prime(u_b))v}_H\\
\leq& \norm{ P\mathcal{F}(v_1 (\tilde{\Pi}(\theta_1) - \tilde{\Pi}(\theta_2))) }_H 
+\norm{P\mathcal{F}\left((\rho_1 \tilde{\Pi}^\prime(\theta_1) -\rho_2 \tilde{\Pi}^\prime(\theta_2))[v_2,v_3]  \right)}_H\\
\leq& L_{\tilde{\Pi}}\norm{\theta_1-\theta_2}_{[L^2(\Omega)]^2} + \norm{\rho_1-\rho_2}_{L^\infty(\Omega)} +L_{p_\theta}\norm{\rho_2}_{L^\infty(\Omega)} \norm{\theta_1-\theta_2}_{[L^2(\Omega)]^2}.
\end{aligned}
\]
Here, we use the fact that $\mathcal{F}$ is a unitary operator,  $\| \tilde{\Pi}(\theta)\|_{[L^\infty(\om)^{3}]^{L}}$ is uniformly bounded,  and $L_{\tilde{\Pi}}$ and $L_{p_\theta}$ are the Lipschitz constants of $\tilde{\Pi}(\theta)$ and $\tilde{\Pi}^\prime(\theta)$, respectively.

\subsection{Numerical algorithm}
For the numerical solution of the reduced optimization problem associated with the present qMRI problem, we adopt the SQP method, i.e., Algorithm \ref{alg:SQP}, from the previous application to the qMRI setting. The only difference is that we do not need the Newton iterations in Step $(a1)$ there.
Recall that now we have $u=(T_1,T_2,\rho)^\top$. In comparison to the previous PDE examples, the sensitivity of the reduced objective functional in \eqref{eq:qMRI_nn} is directly available as
\begin{equation}\label{eq:stationary_mri}
	\mathcal{J}'_{\mathcal{N}}(u)=(\rho (\mathcal{N}^\prime(T_1,T_2))^\ast,\mathcal{N}(T_1,T_2))^\top \mathcal{F}^\ast(\mathcal{F} (\rho \mathcal{N}(T_1,T_2)) -g)+\alpha(\text{Id}-\Delta) ( T_1, T_2, \rho)^\top.
\end{equation}
Further, in every QP-step one is confronted with solving
\begin{equation}\label{eq:SQP_MRI}
	\begin{aligned}
		\text{minimize}& \quad  \langle \mathcal{J}_{\mathcal{N}}'(u_k),h\rangle_{U^{\ast}, U} + \frac{1}{2}\langle H_k(u_k)h, h \rangle_{U^{\ast}, U}\quad\text{over }h\in U\\
		\text{s.t.} & \quad  u_k+h \in \mathcal{C}_{ad},
	\end{aligned}  
\end{equation}
where now $H_k(u_k)$ is the following symmetrized version of the Hessian of $  \mathcal{J}_{\mathcal{N}}$ at $u_k \in \mathcal{C}_{ad}$:
\[ 
\begin{array}{ll}
(\rho (\mathcal{N}^\prime(T_1,T_2))^\ast,\mathcal{N}(T_1,T_2))^\top \mathcal{F}^\ast \mathcal{F} (\rho (\mathcal{N}^\prime(T_1,T_2)),\mathcal{N}(T_1,T_2)) +\alpha( \text{Id}-\Delta)  .
\end{array}
\]

In the following tests, we choose $\mu_0=1$, $\epsilon=10^{-5}$, $r=0.618$, $\kappa=10^{-3}$, and $\xi =0.5$. We stop the SQP iteration when the norm of the residuals of the first-order optimality system drops below a user-specified threshold value of $10^{-3}$ or a maximum of $40$ iterations is reached. The regularization parameter is $\alpha_0=[1,1,1]\times 10^{-10}$ for the $L^2$ part in the regularization functional  in \eqref{eq:qMRI_nn}, and $\alpha_1 = [1,20,2]\times10^{-9} $ for the $H^1$  seminorm part in \eqref{eq:qMRI_nn}, with respect to $T_1, \; T_2,\; \rho$,  respectively. The parameter $c$ in the complementary constraint is chosen to be $10^{9}\alpha_1$ in the numerical tests, which is different to the previous examples. The values of all remaining parameters in Algorithm \ref{alg:SQP} not explicitly mentioned here, are kept the same as in the previous tests. We notice here that due to the analytical structure of the problem, the primal-dual active set algorithm for this example is equivalent to a SSN approach only in the discretized setting. We refer to \cite{HiKu-PathII} for a path-following SSN solver which works in function space upon Moreau-Yosida regularization of the indicator function of the constraint set.

\subsection{Numerical results on qMRI}
For the generation of the training data, we use the explicit Bloch dynamics of \cite{DavPuyVanWia14} where a specific pulse sequence with acronym IR-bSSFP (short for {\it Inversion Recovery balanced Steady State Free Precession}) is considered.
Let $(M_l)_{l=1}^L$ denote the pertinent explicit solution. This yields $\Pi(u)=\rho (M_l(T_1,T_2))_{l=1}^{L}$, with $u=(T_1,T_2,\rho)^\top$. 
The MRI tests are implemented based on an anatomical brain phantom, publicly available from the Brain Web Simulated Brain Database \cite{Brainweb,Col_etal_98}. 
We use a slice with $217\times 181$ pixels from this database and cut some of the zero fill-in pixels so that we finally arrive at a $181\times 181$-pixel image. The selected range for $u$ reflects natural values encountered in the human body. This gives rise to the box constraint $\mathcal{C}_{ad}:=\{u= (T_1,T_2,\rho)^\top: T_1\in (0, 5000), T_2\in (0,1800), \rho\in(0,6000)\}$. 
In Figure \ref{fig:ideal_solutions}, we show the images from the brain phantom for ideal parameter maps $T_1$, $T_2$ and $\rho$.
\begin{figure}[!ht]
	\centering
	\includegraphics[width=\textwidth]{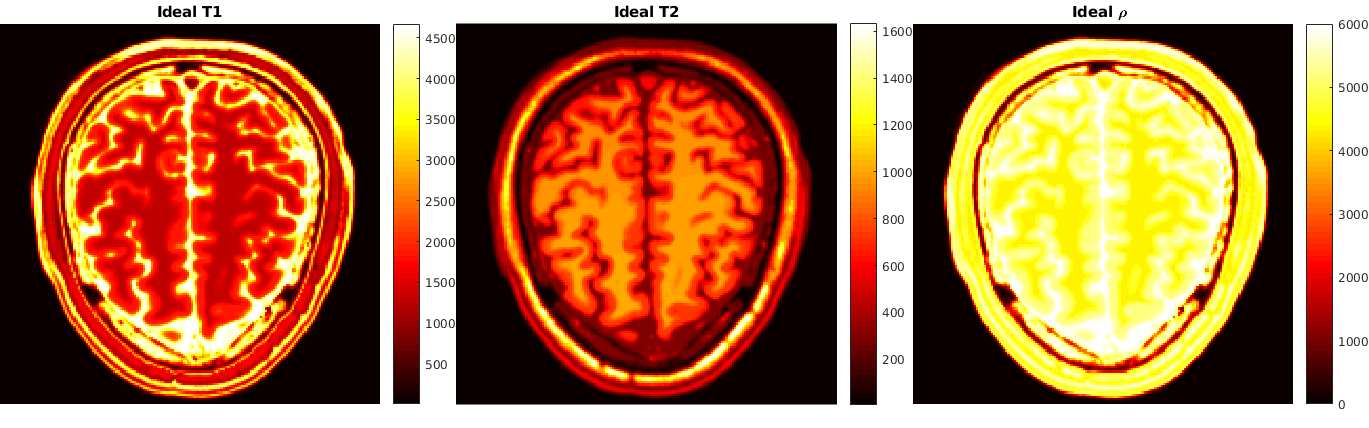}
	\caption{ Simulated ideal tissue parameters of a brain phantom.}
	\label{fig:ideal_solutions}
\end{figure}
\paragraph{Loss function and training method} For each residual of two neighbored images in the time series, we use the mean squared error as the loss function and the Bayesian regularization  algorithm based on the Levenberg-Marquardt method for the training of the residual neural networks DRNN described below. The learning algorithm and the setting are the same as the previous examples.

\paragraph{Architecture of the network}
In order to approximate the Bloch solution map, we use  Direct Residual Neural Networks (DRNNs). Here the solution map at a given time  is approximated by a neural network depending  only on the initial condition $M_0$. To explain this in detail, let $\hat{M}$ be the learned approximation of $M$, i.e. $\hat{M}_{l}(T_{1}, T_{2})\simeq M_{l}(T_{1}, T_{2})$, $l=1,\ldots, L$. The DRNN framework then reads:
\begin{equation}
	\label{eq:dir_res_nn}
	\hat{M}_{l}(T_{1}, T_{2})=\hat{M}_{0}(T_{1}, T_{2}) + \mathcal{N}_{\Theta_l}(T_{1}, T_{2}),\quad \text{} l=1,\ldots, L,\quad
	\hat{M}_{0}(T_{1}, T_{2})=M_0,
\end{equation}
with sub-networks $\{\mathcal{N}_{\Theta_{l}}\}_{l=1}^{L}$.
The  map $(M_l)_{l=1}^L$ is then simply approximated by the map $(M_0 + \mathcal{N}_{\Theta_l})_{l=1}^{L}$.

We use sub-networks with a total number of hidden layers equal to 1, 2, or 3. In each case, we design the architecture at every layer so that  the total degrees of freedom in $\Theta$ are essentially the same.
The detailed description is summarized in Table \ref{tab:net_arc_MRI}. 
In total, we test $9$ different architectures. For every network, we use the 'softmax' activation function in the layer next to the output layer, and the 'logsigmoid' function in all other hidden layers.
The difference to the previous optimal control examples is that the architecture applies to every sub-network which is of residual type, as described above.
\begin{table}[h!]
	\begin{center}
		\renewcommand{\arraystretch}{1.0}
		\resizebox{\textwidth}{!}{
			\begin{tabular}{|l|llll|llll|llll|}\hline
				& HL 1 & HL 2 & HL 3 & DoF 	& HL 1 & HL 2 & HL 3 & DoF	& HL 1 & HL 2 & HL 3 & DoF \\
				\hline
				&\multicolumn{4}{c|}{Small DoF}&\multicolumn{4}{c|}{Medium DoF}&\multicolumn{4}{c|}{Large DoF}\\
				\hline
				1-L-NN & 24 &  - & - &  122& 75 &  - & - &    377& 130 &  - & - &    652 	\\
				2-L-NN  & 7 &  10 & - &   123&17 & 16 & - &   373 &23 & 22 & -&    643		\\
				3-L-NN  &5 & 8 & 5&   120&10 & 15 & 10 &    377&15 & 18 & 15&    650 	\\
				\hline
		\end{tabular}}\\
		\vspace{5pt}
		\caption{The architecture of every sub-network. Both input and output layers have two neurons.}\label{tab:net_arc_MRI}
	\end{center}
\end{table}

\paragraph{Training and validation data}

The training including also the validation data are generated from the dictionary which has been used in methods for magnetic resonance fingerprinting (MRF), e.g., \cite{DavPuyVanWia14, Ma_etal13}.
These are time series resulting from the dynamics, such as e.g. IR-bSSFP, which was introduced in \cite{Sche99}, given the initial value $M_0=(0,0,-1)$. We fix the length of the pulse sequence to be $L=20$.
Of course, other numerical simulations of the Bloch equations can also be proper options as input-output training data.
We test each of the networks with architectures according to Table \ref{tab:net_arc_MRI} using three levels of training data, which we term 'small', 'medium' and 'large'.
For the small size training data, we generate parameter values for $(T_1,T_2)$ from $D_1:=(0:400:5000)$ and $D_2:=(0:100:1800)$ (in MATLAB notation) which contribute $247 $ entries of time series; for the medium size training data from $D_1:=(0:200:5000)$ and $D_2:=(0:50:1800)$ with a total of $962$ entries; and
for the large size data  $D_1:=(0:50:5000)$ and $D_2:=(0:20:1800)$ resulting in total in $9191$ entries.
The input data of the neural networks consist of elements of the set $D_1 \times D_2$. 
Note here that we include $0$ for both $T_1$ and $T_2$, respectively, to take care of the marginal area in the imaging domain. 
The output data will be the Bloch dynamics corresponding to each pair of elements in  $D_1 \times D_2$. 
Both input and output data are normalized to pairs whose elements take values in the range $[-1,1]$. This is done by \verb+mapminmax+ function in MATLAB.

For the SQP we consider the image domain to be $[0,1]\times [0,1]$, thus the spatial discretization size is $h=1/180$.
We compare the results of the learning-based method with results from the algorithm proposed in our previous work \cite{DonHinPap19}.  
The initialization to the SQP algorithm and also the algorithm in \cite{DonHinPap19} is done by using the so-called BLIP algorithm of \cite{DavPuyVanWia14} with a dictionary resulting from the small size $D_1\times D_2$. The parameters are tuned as in \cite{DonHinPap19}.
Concerning the degradation of our image data we consider here Gaussian noise of mean $0$ and standard deviation $30$.

\begin{table}[!ht]
	\begin{center}
		\resizebox{\textwidth}{!}{
			\begin{tabular}{ |c|llll|llll|llll|}
				\hline
				&\multicolumn{4}{c|}{Small DoF} &\multicolumn{4}{c|}{Medium  DoF}  
				&\multicolumn{4}{c|}{Large DoF}  \\  \hline 
				&  $T_1$ & $T_2$ &  $\rho $ & $M(\theta)$   & $T_1$ & $T_2$ &   $\rho$ & $M(\theta)$    & $T_1$ & $T_2$ &   $\rho$ & $M(\theta)$      \\ \hline
				& \multicolumn{12}{c|}{Small training data }    \\ \hline
				1 Layer NN & $0.084$ & $0.056$  &  $0.004 $ & $0.016$ & $ -$ &  $  -$ & $   -$ & $  -$ & $  -$ &  $ -$ & $  - $ & $-  $   \\ \hline
				2 Layer NN & $0.093$ & $0.054$  &  $0.005$ & $  0.013 $ & $ -$ &  $  -$ & $   -$ & $  -$ & $ -$ &  $  -$ & $   -$ & $  -$  \\ \hline
				3 Layer NN & $0.087$ & $0.052$  &  $0.009$ & $  0.012 $ &$ -$ &  $  -$ & $   -$ & $  -$ & $ -$ &  $  -$ & $   -$ & $  -$  \\ \hline
				& \multicolumn{12}{c|}{Medium training data}    \\ \hline				
				1 Layer NN & $0.084$ & $0.058$  &  $0.003 $ & $0.004$ &$ 0.089$ & $0.052  $  &  $ 0.002 $ & $0.005  $  & $ -$ &  $  -$ & $   -$ & $  -$ \\ \hline
				2 Layer NN & $0.143 $ & $0.060  $  &  $0.006  $ & $  0.004 $ & $0.090$ & $0.052 $  &  $0.005 $ & $  0.003$ & $ -$ &  $  -$ & $   -$ & $  -$  \\ \hline
				3 Layer NN & $0.086$ & $0.051$  &  $0.003$ & $  0.004 $ & $0.087$ & $0.051 $  &  $ 0.004$ & $0.002  $ & $ -$ &  $  -$ & $   -$ & $  -$ \\ \hline
				& \multicolumn{12}{c|}{Large training data}    \\ \hline				
				1 Layer NN & $0.120$ & $0.078$  &  $0.005$ & $0.002$ & $0.120  $ & $ 0.081  $  &  $0.004  $ & $ 0.0014 $ & $ 0.090  $ &  $0.050  $ & $  0.004  $ & $ 0.0009 $   \\ \hline
				2 Layer NN & $0.094 $ & $0.057  $  &  $0.006  $ & $  0.001 $ & $0.094 $ & $0.043$  &  $ 0.002 $ & $0.002 $ & $0.089  $ &  $ 0.056 $ & $0.004    $ & $ 0.0012  $  \\ \hline
				3 Layer NN & $0.096$ & $0.059$  &  $0.005$ & $  0.0007 $ & $0.087  $ & $0.051  $  &  $0.004  $ & $0.0004$ & $ 0.087 $ &  $0.051  $ & $ 0.004 $ & $  0.0006$  \\ \hline
				Method \cite{DonHinPap19} & $0.102$ & $0.094$  &  $ 0.004$ & $- $ &  \multicolumn{4}{r|}{proposed Algorithm using exact Bloch}  & $0.084$ & $0.051$ & $0.003$ & $-$ \\	\hline\addlinespace[5pt] 
				\multicolumn{13}{c}{For $25\%$ Cartesian subsampled k-space data with Gaussian noise of mean $0$ and standard deviation $30$.}  \\  				\multicolumn{13}{c}{\vspace{10pt} Relative error computed from $\frac{\norm{x-x^*}}{\norm{x^*}}$ for $x=T_1,\; T_2,\; \rho,\; M$ where $\norm{\cdot}$ is the discrete $2$-norm.}\\ 
		\end{tabular}}
		\caption{Error comparison for qMRI: Using Bloch maps  by networks with different layers, different size of neurons, and a variant of training data}		\label{tab:MRIcomparison}
	\end{center}
\end{table}

Concerning the results reported in Table \ref{tab:MRIcomparison}, the columns of $M(\theta)$ reflect the approximation accuracy to the discrete dynamical Bloch sequences using various neural networks. A smaller value refers to a smaller error, or in other words to higher accuracy in the approximation. However, higher accuracy in the Bloch solution operator approximation does not necessarily result in a better estimation of the $T_1$, $T_2$ parameters.
For this purpose, note that differently to the previous example, here the error is evaluated against the ideal solutions.
The dashes in Table \ref{tab:MRIcomparison} belong to cases where the training data are not sufficient to guarantee well enough learning under the current setting our paper.
We observe that the results are varying slightly under different network architectures and also when using different volumes of training data. In particular, we have the observations: (i) When the training data is sufficiently rich, with the same number of hidden layers, then the larger the number of neurons the better becomes the approximation to the Bloch mapping. However, this does not mean necessarily better to the estimated parameters in terms of the error rates provided.  (ii) We find that the small DoF networks with small volume training data  achieve already almost the same accuracy as the ones using medium and large DoF networks. The results are almost as good as using SQP with the exact Bloch solution formula. We have also observed that the SQP method with learning-based operators can be computationally more efficient than the one with the exact Bloch operators. This is due to the fact that evaluating the learning-based operator can be much cheaper than solving the exact physical model, although a learning process has to be performed before-hand.

In Figures \ref{fig:solutions} and \ref{fig:errors}, we provide visual comparison of results from different methods for quantitative MRI. Particularly, we compare to the method proposed by the authors in \cite{DonHinPap19} assuming knowledge of the exact Bloch solution map and also the BLIP algorithm in \cite{DavPuyVanWia14} in which the fine dictionary (i.e., a large size data set) is used. 

\begin{figure}[!ht]
	\centering
	\includegraphics[width=\textwidth]{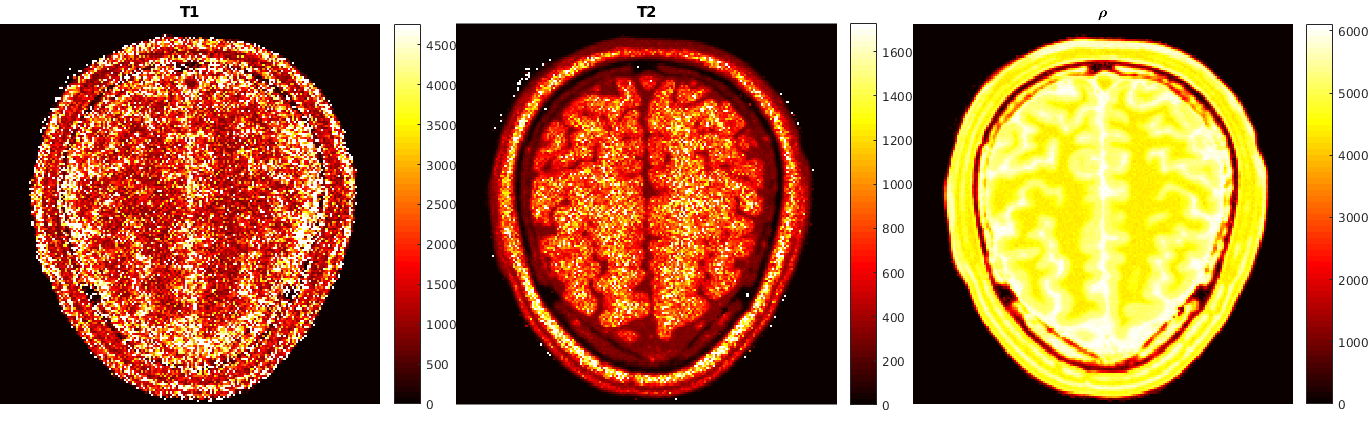}\\
	\vspace*{1pt}
	\includegraphics[width=\textwidth]{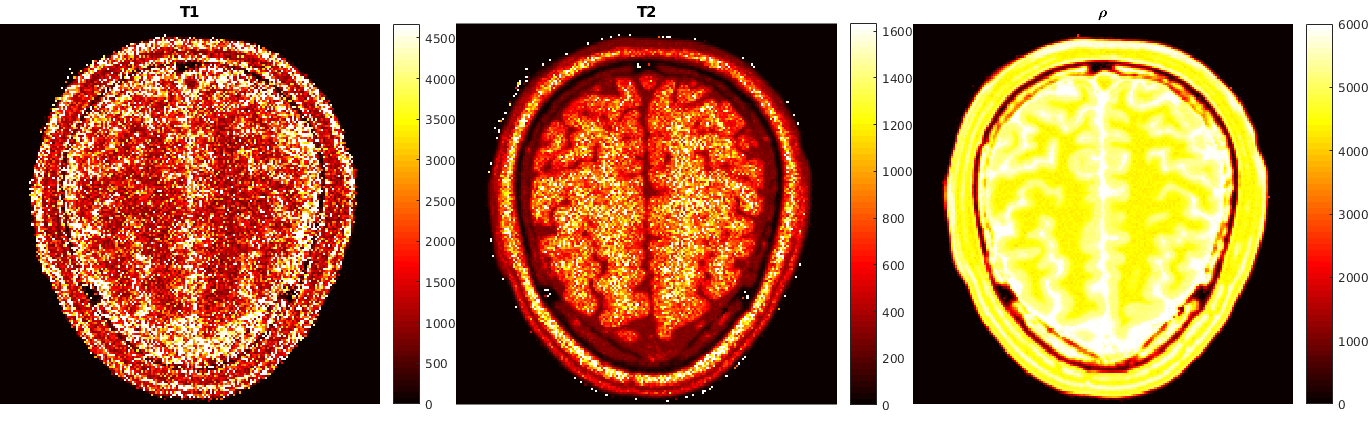}\\
	\vspace*{1pt}
	\includegraphics[width=\textwidth]{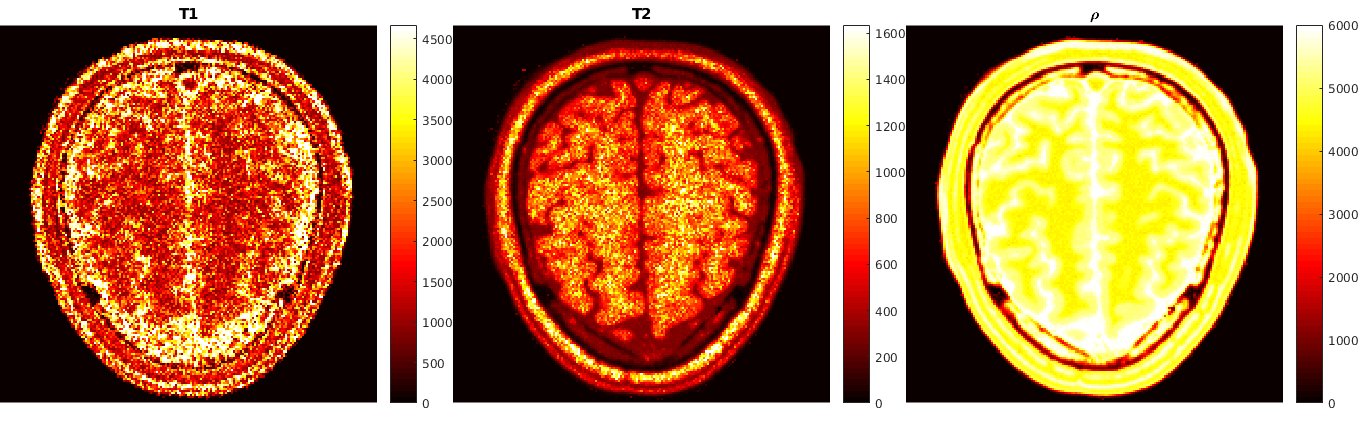}\\
	\vspace*{1pt}
	\includegraphics[width=\textwidth]{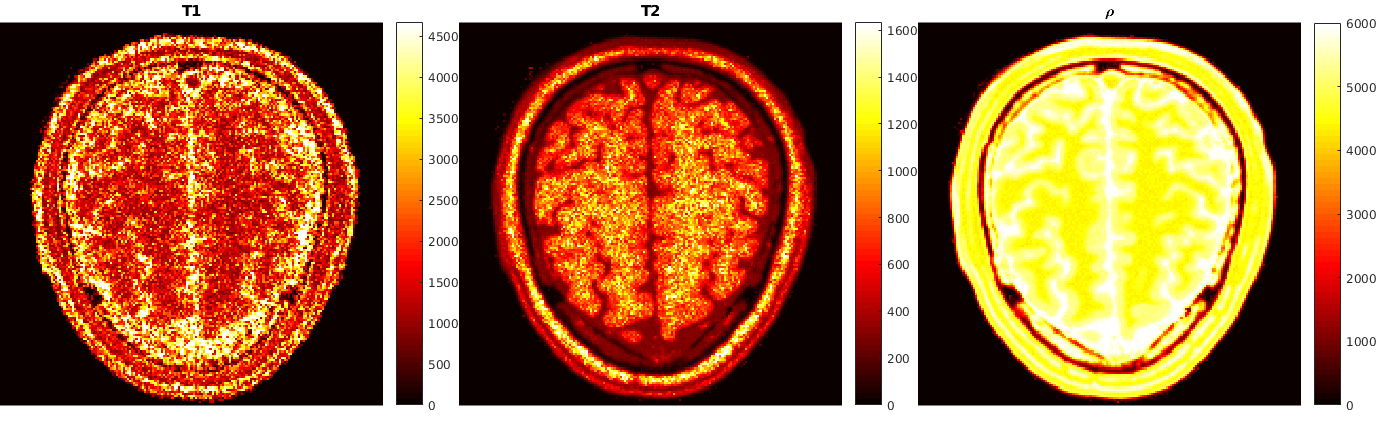}
	\caption{Estimated tissue parameters from subsampled and noisy measurements.
		First row: Solution using the BLIP method in \cite{DavPuyVanWia14} using a fine dictionary; Second row: Solution using method in \cite{DonHinPap19}; Third row: Our SQP solution with learning-informed model small size DoF, 1-hidden-layer residual networks and trained with medium size data.
		Forth row: Our SQP solution using the analytical formula for the Bloch solution map.}
	\label{fig:solutions}
\end{figure}

\begin{figure}[!ht]
	\centering
	\includegraphics[width=\textwidth]{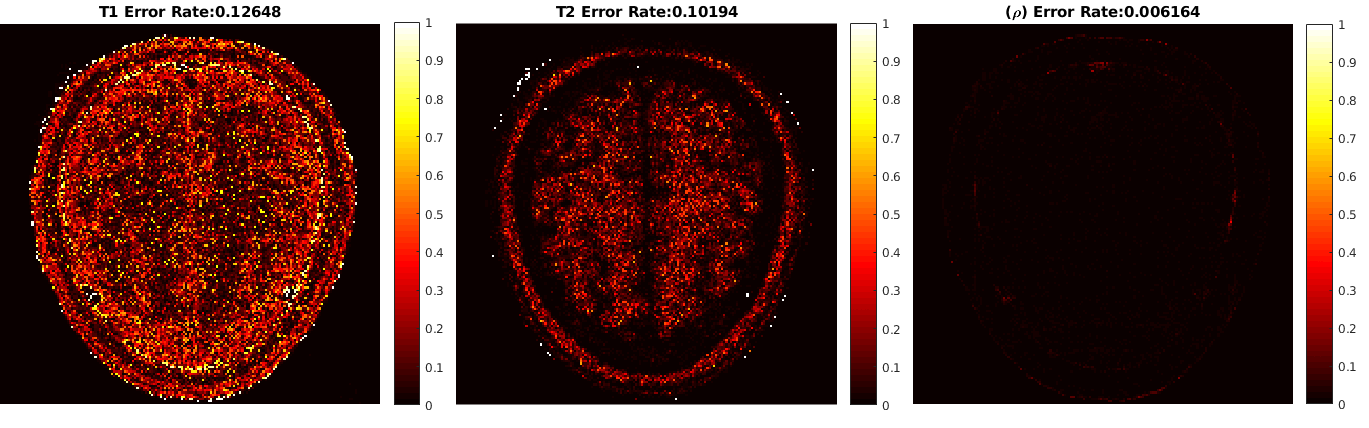}\\
	\vspace*{1pt}
	\includegraphics[width=\textwidth]{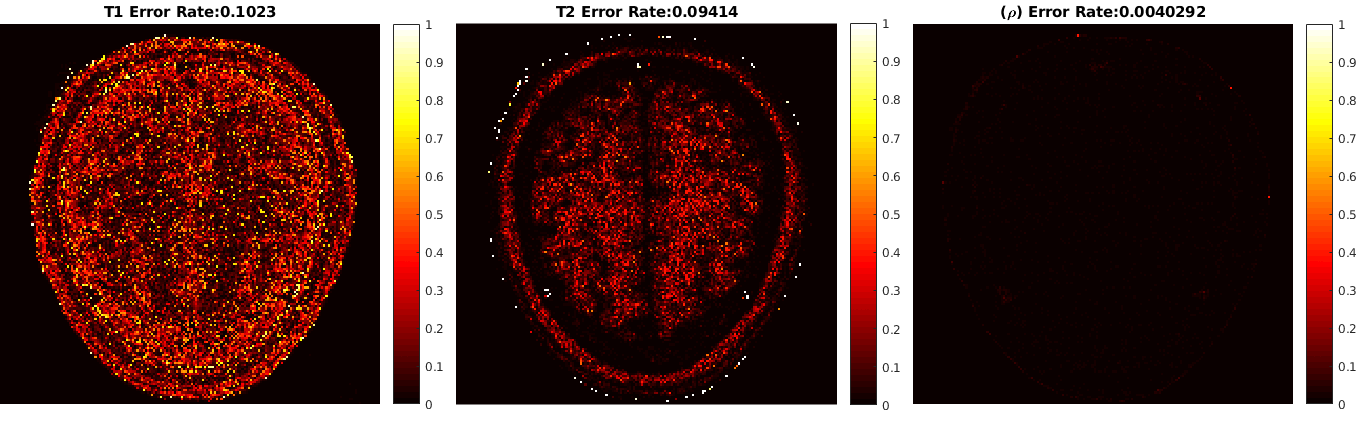}\\
	\vspace*{1pt}
	\includegraphics[width=\textwidth]{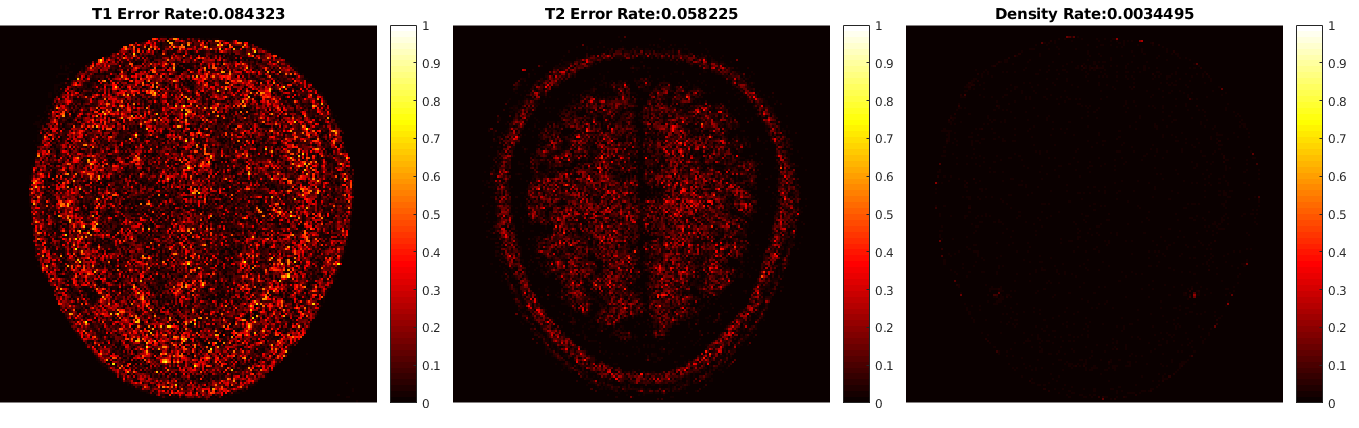}\\
	\vspace*{1pt}
	\includegraphics[width=\textwidth]{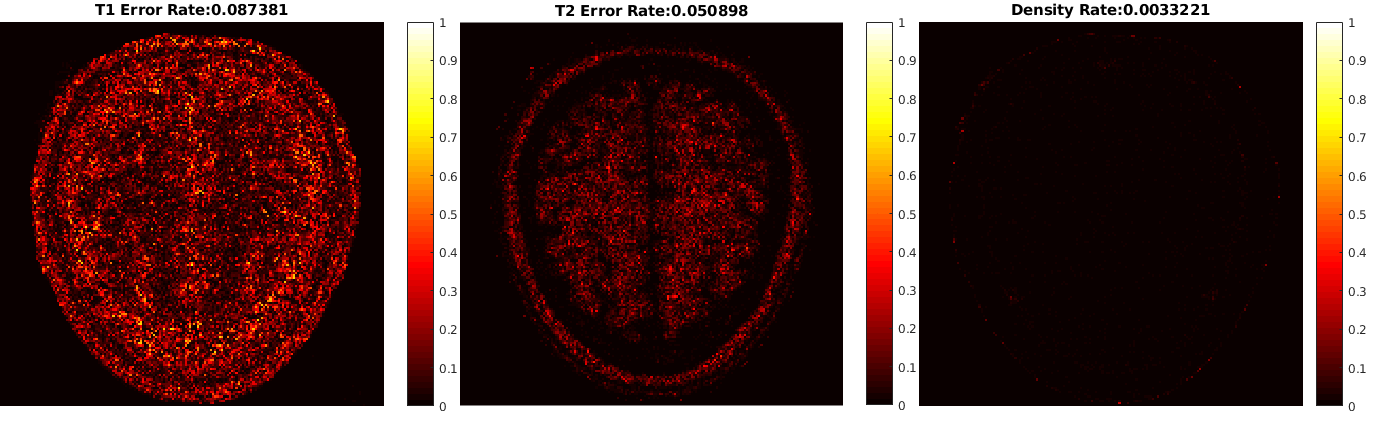}\\
	\caption{Relative errors of the estimated tissue parameters from subsampled and noisy measurements.
		First row: Error map from BLIP \cite{DavPuyVanWia14} using a fine dictionary; Second row: Error map from \cite{DonHinPap19}; Third row: Error map for our SQP solution with learning-informed model. Forth row: Error map for our SQP solution with exact formula for the Bloch map as \cite{DonHinPap19}. All errors are normalized.}
	\label{fig:errors}
\end{figure}

The images produced by the proposed algorithm with a learning-informed model are based on the $1$-hidden-layer network with a small size of DoF which is trained with medium volume data.
We can see that the proposed approach clearly gives better results for the recovering of the quantitative parameters when compared with the methods in \cite{DonHinPap19} and BLIP \cite{DavPuyVanWia14}.
In particular, we observe the $T_1$, $T_2$ parameters estimated by the proposed method are significantly better than the results from the other two methods in terms of spatial regularity. In particular, some artifacts are avoided by the proposed method. This is due to using an $H^1$ term for $u$ in the objective while the method in \cite{DonHinPap19}, for instance, uses an $L^2$ term only.

We notice that the method in \cite{DonHinPap19} is superior only if the noise in the data is small. The learning-informed operator could also be applied yielding results similar to those of the original method \cite{DonHinPap19}.
Since  for real MRI experiments, the $k$-space data may be contaminated by different sources of noise, certain spatial regularization could help to stabilize solutions.
The proposed method in this paper seems to be new to qMRI in this respect, since previous methods typically use pixel-wise estimation so that spatial regularity is harder to enforce.
Along this line, one may consider more sophisticated regularization methods such as, e.g., total variation  or total generalized variation regularization, to take care of spatial discontinuities. Such a study, however, is clearly beyond the scope of the present paper.

\section{Conclusion}

In this paper, we have proposed and analyzed a general optimization scheme for solving optimal control problems subject to constraints which are governed by  learning-informed differential equations. The applications and numerical tests have  verified the feasibility of the proposed scheme for two key applications. We envisage that our work will provide  a fundamental framework for dealing with physical models whose underlying differential equation is partially unknown and thus needed to be learned by data, with the latter typically obtained from experiments or measurements. Our approach avoids learning the full model, i.e., learning directly the solution of the overall minimization problem as this could be on the one hand too complicated and on the other, it could render the method more towards being a black box solver. By learning only a component, i.e., a nonlinearity, or the solution map of the underlying differential equation, the method is kept more faithful to the true physics-based model.

An important factor for the applicability of the proposed framework is the learnability of the operator resulting from differential equations. We observed that  the uniform boundedness of the range of the input and output data (state variable) played a crucial role, stemming from the fact that the density of neural networks holds in the topology of uniform convergence on compact sets. As we observed in the double-well potential example, learning the nonlinearity in its whole range is not necessarily needed,  but only in a range in which the state variables lie, with this range being known due to a priori estimates. Indeed, in the stationary Allen-Cahn control problem, the learning is only performed over a very local part of the nonlinearity (the double-well part), giving an almost perfect result. This shows some potential for reducing the training load by properly analyzing the properties of the nonlinearities.
From the quantitative MRI example we furthermore observed that the embedding of the learned operator in the reconstruction process led to a reduction in the computational load, since it avoids a repetitive solution of the exact physical model. 

A series of future studies arise from the present work.  The analysis implemented here asks for  smooth neural networks approximating (part of) the control-to-state map. A theory incorporating nonsmooth neural networks is an important extension as this will include networks with ReLU activation functions.  Further studies can also incorporate the network structure (in the spirit of optimal experimental design) as well as aspects of the training process into the overall minimization process to further optimize and robustify the new technique.
Finally,  the errors due to the early stopping of the numerical  algorithm as well as due to the ones from the numerical discretization, can be incorporated in the a priori error analysis. This could be of benefit for designing more suitable network architectures.

\subsection*{Acknowledgment}
The authors acknowledge the support of Tsinghua--Sanya International Mathematical Forum (TSIMF), as some of the ideas in the paper were discussed there while all the authors attended the workshop on ``Efficient Algorithms in Data Science, Learning and Computational Physics'' in January 2020.

\end{document}